\documentclass[12pt]{amsart}
\usepackage{amsmath}
\usepackage[backref=page]{hyperref}
\usepackage{tikz,float,caption}
\usetikzlibrary{decorations.markings,arrows.meta,cd,patterns}

\usepackage[margin=1.25 in,top=1.2in, bottom=1.4 in]{geometry}

\usepackage[bb=ams,scr=euler]{mathalpha}

\usepackage{setspace}
\usepackage{enumitem}
\setenumerate{
  wide=10pt,
  leftmargin=0pt,
  labelwidth=*,
  label=(\roman*),
  topsep=0pt,
  listparindent=0pt,
  parsep=4pt}

\usepackage{thmtools}
\declaretheoremstyle[headfont=\normalsize\normalfont\bfseries,notefont=\mdseries, notebraces={(}{)},bodyfont=\normalfont,postheadspace=0.5em]{basicstyle}

\declaretheorem[style=basicstyle,name=Theorem]{theorem}

\declaretheorem[name=Remark,style=basicstyle,sibling=theorem]{remark}

\declaretheorem[style=basicstyle,name=Corollary,sibling=theorem]{cor}

\declaretheorem[style=basicstyle,name=Proposition,sibling=theorem]{prop}
\declaretheorem[style=basicstyle,name=Lemma,sibling=theorem]{lemma}
\declaretheorem[style=basicstyle,name=Proof,numbered=no,qed=$\square$]{preproof}
\renewenvironment{proof}{\preproof}{\endpreproof}
\newenvironment{dmatrix}{\left[\,\begin{matrix}}{\end{matrix}\,\right]}

\makeatletter
\newcommand{\abs}[1]{\left|#1\right|}
\newcommand{\bd}{\partial}
\newcommand{\cl}[1]{\overline{#1}}
\newcommand{\C}{\mathbb{C}}
\newcommand{\db}{\overline{\partial}}
\renewcommand{\d}{\mathrm{d}}
\newcommand{\ip}[1]{\left\langle#1\right\rangle}
\newcommand{\norm}[1]{\left\lVert#1\right\rVert}
\newcommand{\pd}[2]{\frac{\partial #1}{\partial #2}}
\newcommand{\pr}{\mathrm{pr}}
\newcommand{\R}{\mathbb{R}}
\renewcommand\section{\@startsection{section}{1}{0pt}{-3.5ex \@plus -1ex \@minus -.2ex}{2.3ex \@plus.2ex}{\centering\bfseries}}
\renewcommand{\setminus}{\mathbin{\tikz[baseline={(0,0.033)}]{\draw[line width=0.4, rounded corners=0.5,line cap=round,scale=0.7] (0,0.3)--(0.25,0.1);}}}
\newcommand{\set}[1]{\left\{#1\right\}}
\newcommand{\Z}{\mathbb{Z}}

\makeatother

\title[An Index Formula For Cauchy-Riemann Operators]{An Index Formula For Cauchy-Riemann Operators on Surfaces with boundary punctures}
\author{Dylan Cant}
\date{Monday November 1st 2021}
\begin{document}
\maketitle
\begin{abstract}
  We give a self-contained proof of a formula computing the Fredholm index for asymptotically non-degenerate Cauchy-Riemann operators on surfaces with boundary punctures using the method of large anti-linear deformations. This method for computing the index was introduced in the case of closed surfaces by \cite{taubes} and generalized to the case with interior punctures by \cite{gerig}, as explained in \cite{wendl-sft}. One novel feature arising from our proof is that the Euler characteristic term in the index formula involves a non-standard weighted count of boundary zeros. We hope that this formulation of the index formula will be useful to other researchers.
\end{abstract}
\tableofcontents
\section{Introduction}
The main goal of this paper is to prove an index formula for \emph{asymptotically non-degenerate Cauchy-Riemann operators} on surfaces with boundary punctures. We generalize the result stated in \cite[Theorem 3.3.11]{schwarz-diss} (see also \cite[Theorem 3.1.2]{gerig} and \cite[Theorem 5.4]{wendl-sft}).

To actually prove the index formula we adopt the technique introduced in \cite[Section 7]{taubes} (subsequently generalized by \cite[Chapter 3]{gerig}), and deform our Cauchy-Riemann operator $D$ by an anti-linear lower order term $\sigma B$. As explained in \cite{taubes}, \cite{gerig}, and \cite{wendl-sft}, as $\sigma\to\infty$ the kernel of $D+\sigma B$ concentrates near the positive zeros of $B$ and the cokernel of $D+\sigma B$ is represented by sections supported near the negative zeros of $B$. With some further analysis, one concludes that the signed count of zeros of $B$ equals the index of $D+\sigma B$.

Our argument is complicated by the boundary $\bd\Sigma$. The most apparent difference is that the anti-linear perturbation $B$ can have zeros on the boundary, and, as we will show, the boundary zeros split into \emph{four} cases, two of which contribute $0$ to the index, and the other two contribute $+1$ and $-1$. See Figure \ref{fig:boundary-zeros} and \ref{fig:dual-yt}. This phenomenon leads to the ``Euler characteristic'' term in the index formula depending on the signs of punctures -- this is a novel phenomenon when compared with the $\bd\Sigma=\emptyset$ case. See Section \ref{sec:D1-kernel-class} for the computation which leads to some of the boundary zeros contributing $0$.

\begin{remark}  
  If the asymptotics of $D$ do not match the asymptotics of $D+\sigma B$ for $\sigma$ large, then the Fredholm index will likely change during the deformation $\sigma\to\infty$. There are two approaches to deal with this: one way is to try to find $B$ so that the asymptotics of $D$ match the asymptotics $D+\sigma B$ for all $\sigma\in [0,\infty)$ -- this is the approach taken in \cite[Chapter 3]{gerig} and \cite[Section 5.8]{wendl-sft}. The other approach is to pick $B$ without regard to $D$, and then analyze the change in index as an ``index gluing'' problem. This is the approach considered in this paper, and it leads to a natural definition of the Conley-Zehnder indices as Fredholm indices of certain operators. The necessary analytic ingredient to make this work is the \emph{linear kernel gluing operation} (see Section \ref{sec:cz-index}). See \cite[Section 3.2]{schwarz-diss} and \cite{fh-coherent} for similar gluing problems.
\end{remark}
\begin{remark}
  There is other work which proves index formulas for Cauchy-Riemann operators on surfaces with boundary punctures. See, for instance, \cite[Appendix A]{cel2009}. Our work differs from theirs in how we present the index formula (e.g.\ they do not define Conley-Zehnder indices for Reeb chords), and how we prove the result. 
\end{remark}



\subsection{Statement of result}
\label{sec:statement}
Let $D$ be an asymptotically non-degenerate Cauchy-Riemann operator for the data $(\Sigma,\bd\Sigma,\Gamma_{\pm},E,F,C,[\tau])$. Briefly:
  \begin{enumerate}[label=(\arabic*)]
  \item $\Gamma=\Gamma_{+}\cup \Gamma_{-}$ is collection of punctures which may be on the boundary (we denote the punctured surface by $\dot\Sigma$),
  \item $(E,F)$ is a complex vector bundle with totally real sub-bundle $F\subset E|_{\bd\dot\Sigma}$
  \item for each $z\in \Gamma$, $C_{z}\subset \dot\Sigma$ is a chosen cylindrical/strip-like end with holomorphic coordinate $s+it$ (there are four possibilities for $C_{z}$, depending on whether $z\in \Gamma_{\pm}$ and $z\in \bd\Sigma$),
  \item $[\tau]$ is an equivalence class of trivializations $\tau_{z}:(E|_{C_{z}},F|_{\bd C_{z}})\to (\C^{n},\R^{n})$ called \emph{asymptotic trivializations}. See Section \ref{sec:ahs} for more details.
  \end{enumerate}
  We recall that Cauchy-Riemann operators are defined by their symbol. The \emph{asymptotically non-degenerate} condition means that for any $\tau\in [\tau]$, the coordinate representation $D_{\tau}$ in the end $C_{z}$ is asymptotic to $\bd_{s}-A_{z}^{\tau}$ as $s\to\pm \infty$ where $A_{z}^{\tau}=-i\bd_{t}-S(t)$ is a non-degenerate \emph{asymptotic operator}: $$C^{\infty}([0,1],\C^{n},\R^{n})\to C^{\infty}([0,1],\C^{n})\hspace{.2cm}\text{ or }\hspace{.2cm}C^{\infty}(\R/\Z,\C^{n},\R^{n})\to C^{\infty}(\R/\Z,\C^{n}).$$ See Section \ref{sec:asymptotics} for more details.

  \begin{theorem}\label{theorem:index-formula}
    For $p>1$, $D:W^{1,p}(E,F)\to L^{p}(\Lambda^{0,1}\otimes E)$ is Fredholm and its index is given by
    \begin{equation*}
      \mathrm{ind}(D)=n\mathrm{X}(\Sigma,\Gamma_{\pm})+\mu_{\mathrm{Mas}}^{\tau}(E,F)+\sum_{z\in \Gamma_{+}}\mu_{\mathrm{CZ}}(A^{\tau}_{z})-\sum_{z\in \Gamma_{-}}\mu_{\mathrm{CZ}}(A^{\tau}_{z}),
    \end{equation*}
    where $n$ is the complex rank of $E$, $\tau$ is an asymptotic trivialization of $(E,F)$, and:
    \begin{enumerate}
    \item The Euler characteristic $\mathrm{X}(\Sigma,\Gamma_{\pm})$ is the count of zeros of a generic vector field on $\dot\Sigma$ which is tangent to $\bd\dot\Sigma$ and points inwards along $\Gamma_{-}$ and outwards along $\Gamma_{+}$ (e.g.,\ equal to $\bd_{s}$ in the ends $C_{z}$). Boundary zeros are counted according to the rules in Figure \ref{fig:boundary-zeros}. Interior are zeros are counted as usual. See Section \ref{fig:example-surfaces} for examples.
     
  \item The Maslov index $\mu^{\tau}_{\mathrm{Mas}}(E,F)$ is the signed count of zeros of a generic section $\sigma$ of $(\det E)^{\otimes 2}$ which (a) restricts to the canonical positive generator of $(\det F)^{\otimes 2}$ along the boundary, and (b) is identically $1$ in the asymptotic trivializations induced by~$\tau$. The zeros are all interior.
    
  \item The Conley-Zehnder index is the Fredholm index of any Cauchy-Riemann operator on the trivial bundle $E=\C^{n}$, $F=\R^{n}$ over an infinite strip/cylinder which equals $$\bd_{s}u+J_{0}\bd_{t}u+\cl{u}=\bd_{s}u+J_{0}\bd_{t}u+Cu$$ at the negative end and $\bd_{s}-A_{z}^{\tau}$ at the positive end. See Figure \ref{fig:cz-index}.
  \end{enumerate}
\end{theorem}

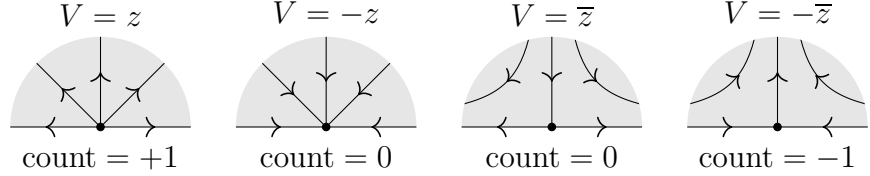
\begin{figure}[H]
  \centering
  \begin{tikzpicture}[scale=.6]
    \begin{scope}[shift={(0,0)}]
      \fill[black!10!white] (2,0) arc (0:180:2)--cycle;
      \node at (0,2) [above] {$V=z$};
      \draw[postaction={decorate,decoration={markings,mark=at position 0.5 with {\arrow[scale=1.5]{<};}}}] (-2,0)--(0,0);
      \draw[postaction={decorate,decoration={markings,mark=at position 0.5 with {\arrow[scale=1.5]{<};}}}] (2,0)--(0,0);
      \draw[postaction={decorate,decoration={markings,mark=at position 0.5 with {\arrow[scale=1.5]{<};}}}] (0,2)--(0,0);
      \draw[postaction={decorate,decoration={markings,mark=at position 0.5 with {\arrow[scale=1.5]{<};}}}] (45:2)--(0,0);
      \draw[postaction={decorate,decoration={markings,mark=at position 0.5 with {\arrow[scale=1.5]{<};}}}] (135:2)--(0,0);
      \node[draw,circle,inner sep=1pt,fill] at (0,0){};
      \node at (0,-0.2)[below]{$\mathrm{count}=+1$};
    \end{scope}
    \begin{scope}[shift={(5,0)}]
      \node at (0,2) [above] {$V=-z$};
      \fill[black!10!white] (2,0) arc (0:180:2)--cycle;
      \draw[postaction={decorate,decoration={markings,mark=at position 0.5 with {\arrow[scale=1.5]{>};}}}] (-2,0)--(0,0);
      \draw[postaction={decorate,decoration={markings,mark=at position 0.5 with {\arrow[scale=1.5]{>};}}}] (2,0)--(0,0);
      \draw[postaction={decorate,decoration={markings,mark=at position 0.5 with {\arrow[scale=1.5]{>};}}}] (0,2)--(0,0);
      \draw[postaction={decorate,decoration={markings,mark=at position 0.5 with {\arrow[scale=1.5]{>};}}}] (45:2)--(0,0);
      \draw[postaction={decorate,decoration={markings,mark=at position 0.5 with {\arrow[scale=1.5]{>};}}}] (135:2)--(0,0);
      \node[draw,circle,inner sep=1pt,fill] at (0,0){};
      \node at (0,-0.2)[below]{$\mathrm{count}=0$};
    \end{scope}
    \begin{scope}[shift={(10,0)}]
      \node at (0,2) [above] {$V=\cl{z}$};
      \fill[black!10!white] (2,0) arc (0:180:2)--cycle;
      \begin{scope}
        \clip (2,0) arc (0:180:2)--cycle;
        \draw[postaction={decorate,decoration={markings,mark=at position 0.5 with {\arrow[scale=1.5]{>};}}}] plot[domain=0.5:2] ({\x},{1/\x});
        \draw[postaction={decorate,decoration={markings,mark=at position 0.5 with {\arrow[scale=1.5]{>};}}}] plot[domain=0.5:2] ({-\x},{1/\x});
      \end{scope}
      \draw[postaction={decorate,decoration={markings,mark=at position 0.5 with {\arrow[scale=1.5]{<};}}}] (-2,0)--(0,0);
      \draw[postaction={decorate,decoration={markings,mark=at position 0.5 with {\arrow[scale=1.5]{<};}}}] (2,0)--(0,0);
      \draw[postaction={decorate,decoration={markings,mark=at position 0.5 with {\arrow[scale=1.5]{>};}}}] (0,2)--(0,0);
      \node[draw,circle,inner sep=1pt,fill] at (0,0){};
      \node at (0,-0.2)[below]{$\mathrm{count}=0$};
    \end{scope}
    \begin{scope}[shift={(15,0)}]
      \node at (0,2) [above] {$V=-\cl{z}$};
      \fill[black!10!white] (2,0) arc (0:180:2)--cycle;
      \begin{scope}
        \clip (2,0) arc (0:180:2)--cycle;
        \draw[postaction={decorate,decoration={markings,mark=at position 0.5 with {\arrow[scale=1.5]{<};}}}] plot[domain=0.5:2] ({\x},{1/\x});
        \draw[postaction={decorate,decoration={markings,mark=at position 0.5 with {\arrow[scale=1.5]{<};}}}] plot[domain=0.5:2] ({-\x},{1/\x});
      \end{scope}
      \draw[postaction={decorate,decoration={markings,mark=at position 0.5 with {\arrow[scale=1.5]{>};}}}] (-2,0)--(0,0);
      \draw[postaction={decorate,decoration={markings,mark=at position 0.5 with {\arrow[scale=1.5]{>};}}}] (2,0)--(0,0);
      \draw[postaction={decorate,decoration={markings,mark=at position 0.5 with {\arrow[scale=1.5]{<};}}}] (0,2)--(0,0);
      \node[draw,circle,inner sep=1pt,fill] at (0,0){};
      \node at (0,-0.2)[below]{$\mathrm{count}=-1$};
    \end{scope}
  \end{tikzpicture}
  \caption{Boundary zeros either contribute $\pm 1$ or $0$ to the index.}
  \label{fig:boundary-zeros}
\end{figure}

\begin{figure}[H]
  \centering
  \begin{tikzpicture}
    \begin{scope}[shift={(14,0)}]
      \fill[pattern=north west lines] (7.0,-0.5) rectangle +(1,1);
      \node at ({3+2},0.5)[above]{$\bd_{s}+i\bd_{t}+C$};
      \node at ({10},0.5)[above]{$\bd_{s}-A^{\tau}_{z}$};
      \path (10,-1)--(7.75,-0.5)node[outer sep=0pt,inner sep=1pt](A){};
      \draw (3,0.5)--(12,.5);
      \draw (3,-0.5)--(12,-.5);
      \foreach \x in {4.0,5.0} {
      \draw ({3+\x},-0.5)--+(0,1);
    }
    \end{scope}
  \end{tikzpicture}
  \caption{The Conley-Zehnder index is the Fredholm index of any Cauchy-Riemann operator on the infinite strip or cylinder which interpolates between the two asymptotic conditions. The matrix $C$ represents complex conjugation.}
  \label{fig:cz-index}
\end{figure}
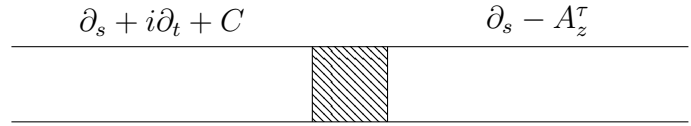

\begin{remark}
  If $\bd\Sigma=\emptyset$, then this agrees with \cite[Theorem 5.4]{wendl-sft}. If $\Gamma=\emptyset$ then this agrees with \cite[Theorem C.1.10]{mcduffsalamon}.
\end{remark}
\begin{remark}
  The definition of the Conley-Zehnder index as a Fredholm index suggests a way to define \emph{determinant lines} for asymptotic operators, namely as the Fredholm determinant of the operator in Figure \ref{fig:cz-index}. This is similar to \cite[Definition 1.4.3]{abouzaid14} or \cite[Definition 2.46]{pardon-cha}.

  A kernel gluing theorem (e.g.,\ see \cite{fh-coherent}) should establish a relationship between the Fredholm determinant of $D$, the determinant lines of the asymptotic operators, and the Fredholm determinant of a different Cauchy-Riemann operator $D^{1}$ where all the asymptotic operators are changed to $-i\bd_{t}-C$. 

  The method of large anti-linear deformations considers a family $D^{\sigma}=D_{0}+\sigma B$ (which agrees with $D^{1}$ when $\sigma=1$). Moreover, $B$ can be chosen so that $D^{\sigma}$ is Fredholm for all $\sigma\ge 1$. See Section \ref{sec:large-antilinear} for a precise definition of $D^{\sigma}$.

  For large $\sigma$, we can explicitly describe the kernel and cokernel of $D^{\sigma}$ as the $\R$-vector space generated by certain sections concentrated near certain zeros of $B$ (i.e.,\ each zero either contributes $\pm 1$, or $0$ to the index). In particular, the problem of orienting the Fredholm determinant of $D^{\sigma}$ reduces to the problem of ordering certain subsets of zeros of $B$. We do not pursue the question of ``coherently orienting'' Fredholm determinants any further in this paper.
\end{remark}

\section{Euler characteristics for Riemann surfaces with boundary punctures}
\label{sec:euler}

In this section we give a more precise definition of the Euler characteristic term appearing in the index formula. Suppose that $(\Sigma,\bd\Sigma,\Gamma_{+},\Gamma_{-},C)$ is a Riemann surface with punctures $\Gamma=\Gamma_{+}\cup \Gamma_{-}$, some of which may be on the boundary, and cylindrical/strip-like ends $C_{z}$ for each $z\in \Gamma_{\pm}$. Each puncture in $\Gamma_{+}$ has a cylindrical end biholomorphic to $[0,\infty)\times [0,1]$ or $[0,\infty)\times \R/\Z$, and similarly for $\Gamma_{-}$ with $[0,\infty)$ replaced by $(-\infty,0]$. Let $s+it$ denote the holomorphic coordinate in these cylindrical ends. 

Let $V$ be a vector field on $\dot\Sigma:=\Sigma\setminus \Gamma_{+}\setminus \Gamma_{-}$ which agrees with $\bd_{s}$ in the cylindrical ends, and which is everywhere tangent to $\bd\dot\Sigma$. See Figure \ref{fig:example-surfaces} for an illustration. By choosing $V$ generically, we can assume that the linearizations of $V$ at its zeros are non-degenerate. Let us agree to call such a vector field \emph{admissible} for $(\Sigma,\bd\Sigma,\Gamma_{\pm})$.

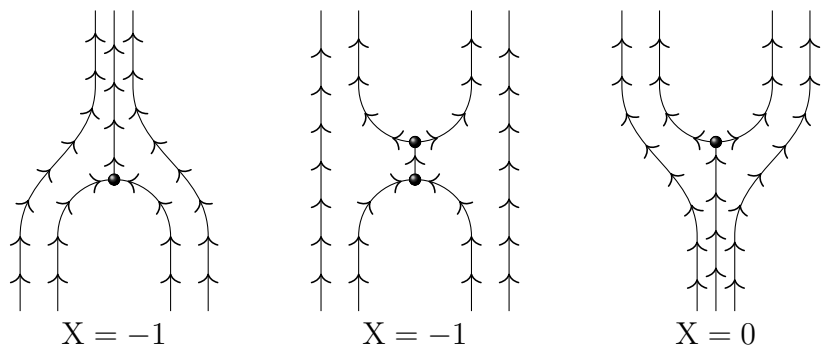
\begin{figure}[H]
  \centering
  \begin{tikzpicture}
    \path[decorate,decoration={markings,mark=between positions 5mm and 1 step 5mm with {\arrow[scale=1.5]{>};}}](0,-1)--(0,0)to[out=90,in=-90] (1,2)--(1,3);
    \path[decorate,decoration={markings,mark=between positions 5mm and 1 step 5mm with {\arrow[scale=1.5]{>};}}](2.5,-1)--(2.5,0)to[out=90,in=-90] (1.5,2)--(1.5,3);

    \path[decorate,decoration={markings,mark=between positions 5mm and 1 step 5mm with {\arrow[scale=1.5]{>};}}]
    (2,-1)--(2,0) arc (0:90:0.75);
    \draw (0,-1)--(0,0)to[out=90,in=-90] (1,2)--(1,3) (1.5,3)--(1.5,2)to[out=-90,in=90] (2.5,0)--(2.5,-1) (2,-1)--(2,0) arc (0:180:0.75)--+(0,-1);
    \draw[postaction={decorate,decoration={markings,mark=between positions 3mm and 1 step 5mm with {\arrow[scale=1.5]{>};}}}] (1.25,0.75)--+(0,{3-0.75});
    \path[decorate,decoration={markings,mark=between positions 5mm and 1 step 5mm with {\arrow[scale=1.5]{>};}}]
    (.5,-1)--(.5,0) arc (180:90:0.75) node[draw,circle,shading=ball,ball color=black,inner sep=1.5pt]{};

    \begin{scope}[shift={(4,0)}]
      \path[decorate,decoration={markings,mark=between positions 5mm and 0.95 step 5mm with {\arrow[scale=1.5]{>};}}](0,-1)--(0,3);
      \path[decorate,decoration={markings,mark=between positions 5mm and 0.95 step 5mm with {\arrow[scale=1.5]{>};}}](2.5,-1)--(2.5,3);

      \draw (0,-1)--(0,3) (2.5,3)--(2.5,-1) (2,-1)--(2,0) arc (0:180:0.75)--+(0,-1);

      \begin{scope}
        \draw[postaction={decorate,decoration={markings,mark=at position 0.65 with {\arrow[scale=1.5]{>};}}}] (1.25,0.75)--(1.25,{2-0.75});
        \draw[shift={(0,2)},yscale=-1](2,-1)--(2,0) arc (0:180:0.75)--+(0,-1);
        \path[decorate,decoration={markings,mark=between positions 5mm and 1 step 5mm with {\arrow[scale=1.5]{>};}}]
        (2,-1)--(2,0) arc (0:90:0.75);
        \path[decorate,decoration={markings,mark=between positions 5mm and 1 step 5mm with {\arrow[scale=1.5]{>};}}]
        (.5,-1)--(.5,0) arc (180:90:0.75) node[draw,circle,shading=ball,ball color=black,inner sep=1.5pt]{};
        \path[shift={(0,2)},yscale=-1,decorate,decoration={markings,mark=between positions 5mm and 1 step 5mm with {\arrow[scale=1.5]{<};}}]
        (2,-1)--(2,0) arc (0:90:0.75);
        \path[shift={(0,2)},yscale=-1,decorate,decoration={markings,mark=between positions 5mm and 1 step 5mm with {\arrow[scale=1.5]{<};}}]
        (.5,-1)--(.5,0) arc (180:90:0.75) node[draw,circle,shading=ball,ball color=black,inner sep=1.5pt]{};
      \end{scope}
    \end{scope}

    \begin{scope}[shift={(8,2)},yscale=-1]
          \path[decorate,decoration={markings,mark=between positions 5mm and 1 step 5mm with {\arrow[scale=1.5]{<};}}](0,-1)--(0,0)to[out=90,in=-90] (1,2)--(1,3);
    \path[decorate,decoration={markings,mark=between positions 5mm and 1 step 5mm with {\arrow[scale=1.5]{<};}}](2.5,-1)--(2.5,0)to[out=90,in=-90] (1.5,2)--(1.5,3);

    \path[decorate,decoration={markings,mark=between positions 5mm and 1 step 5mm with {\arrow[scale=1.5]{<};}}]
    (2,-1)--(2,0) arc (0:90:0.75);
    \draw (0,-1)--(0,0)to[out=90,in=-90] (1,2)--(1,3) (1.5,3)--(1.5,2)to[out=-90,in=90] (2.5,0)--(2.5,-1) (2,-1)--(2,0) arc (0:180:0.75)--+(0,-1);
    \draw[postaction={decorate,decoration={markings,mark=between positions 3mm and 1 step 5mm with {\arrow[scale=1.5]{<};}}}] (1.25,0.75)--+(0,{3-0.75});
    \path[decorate,decoration={markings,mark=between positions 5mm and 1 step 5mm with {\arrow[scale=1.5]{<};}}]
    (.5,-1)--(.5,0) arc (180:90:0.75) node[draw,circle,shading=ball,ball color=black,inner sep=1.5pt]{};

  \end{scope}
  \node at (1.25,-1)[below]{$\mathrm{X}=-1$};
  \node at (5.25,-1)[below]{$\mathrm{X}=-1$};
  \node at (9.25,-1)[below]{$\mathrm{X}=0$};  
  \end{tikzpicture}
  \caption{Vector fields on surfaces with boundary punctures. Positive punctures (i.e.,\ in $\Gamma_{+}$) are placed at the top of the figure while negative punctures are placed at the bottom. The Euler characteristic $\mathrm{X}$ is the count of zeros weighted as in Figure \ref{fig:boundary-zeros}.}
  \label{fig:example-surfaces}
\end{figure}

If $p\in \Sigma$ is a zero of $V$, and $z=s+it$ is a holomorphic coordinate with $z(p)=0$, we can write $V$ as
\begin{equation*}
  V=\begin{dmatrix}
    {\bd_{s}}&{\bd_{t}}
  \end{dmatrix}\begin{dmatrix}
    {a}&{b}\\
    {c}&{d}
  \end{dmatrix}\begin{dmatrix}
    {s}\\
    {t}
  \end{dmatrix}+\text{higher order terms},
\end{equation*}
where the $2\times 2$ matrix is invertible.

If $p$ is an interior zero, then we define the count of $p$ to be the sign of the determinant of the $2\times 2$ matrix. We can deform our vector field near $p$ so that $a=1$, $d=\pm 1$ and $b=c=0$ -- this uses the fact that $\mathrm{GL}_{2}(\mathbb{R})$ has two connected components. After this deformation, the local coordinate representation of $V$ is either $z$ or $\cl{z}$, depending on whether the count of the $p$ is $\pm 1$.

Suppose now that $p$ is a boundary zero. Then we can pick $z$ so that it takes values in $\cl{\mathbb{H}}$, in which case we must have $c=0$ and $a>0$. We define:
\begin{equation*}
  \text{count of $p$}=\left\{
    \begin{aligned}
      +1\text{ if $a>0$ and $c>0$},\\
      0\text{ if $a<0$ and $c<0$},\\
      0\text{ if $a>0$ and $c<0$},\\
      -1\text{ if $a<0$ and $c<0$}.
    \end{aligned}
  \right.
\end{equation*}
Unlike the case when $p$ was an interior zero, we cannot freely deform the linearization, since the linearization is required to map $T\bd \Sigma$ into $T\bd\Sigma$. The four cases above depend on whether the coordinate representation of $V$ can be deformed to $\pm z$ or $\pm \cl{z}$, as shown in Figure \ref{fig:boundary-zeros}. Note that $a\in \R^{1\times 1}$ can be thought of as the linearization of the \emph{restriction} of $V$ to $\bd\Sigma$, considered as a section of $T\bd\Sigma\to \bd\Sigma$.

\begin{prop}\label{prop:X-well-defined}
  The sum of the counts of the zeros of $V$ is independent of the choice of $V$ and the coordinate systems used. It does depend on the assignment of signs to the boundary punctures $\Gamma$. The resulting integer is denoted $\mathrm{X}(\Sigma,\bd\Sigma,\Gamma_{+},\Gamma_{-})$.
\end{prop}
\begin{proof}
  We do not actually use this invariance to prove the index formula, and hence the proposition follows from the index formula. Indeed, one can use the large anti-linear perturbation method of this paper to show that our count of zeros of $V$ equals the Fredholm index of the operator
  \begin{equation*}
    f\mapsto D(f):=\d f+i\cdot \d f\cdot j+\mu(-,V)\cl{f}
  \end{equation*}
  acting on sections of the trivial line bundle $\mathbb{C}$ which take real values on $\bd\dot\Sigma$. Here $\mu$ is a Hermitian metric on $T\dot\Sigma$ which is cylindrical in the ends. This completes the proof. 
\end{proof}

\section{Asymptotically non-degenerate Cauchy-Riemann operators}
\label{sec:cr-ops}

Fix a Riemann surface $\Sigma$ with boundary $\bd\Sigma$ and punctures $\Gamma=\Gamma_{+}\cup \Gamma_{-}$. Fix cylindrical ends around each of the punctures of $\Gamma$; this means that we pick holomorphic coordinate disks or half-disks around each $z\in \Gamma_{\pm}$, and identify the disks with $\R_{\pm}\times \R/\Z$ via the map $(s,t)\mapsto e^{\mp 2\pi(s+it)}$ and the half-disks with $\R_{\pm }\times [0,1]$ via the map $(s,t)\mapsto e^{\mp \pi(s+it)}$. Note that in order for this to make sense, we require picking lower half-disks around positive punctures and upper half-disks around negative punctures. 

For each $z\in \Gamma$, let $C_{z}$ denote the cylindrical end corresponding to $z$, and let $C_{z}(\rho)\subset C_{z}$ denote the closed which translated by $\rho$ deeper into the end, i.e.,\ if $z$ is a positive boundary puncture then $C_{z}(\rho)=[\rho,\infty)\times [0,1]$, and similarly for the other possibilities for $z$. See Figure \ref{fig:rho-deep}. Let $C(\rho)=\bigcup_{z\in \Gamma} C_{z}(\rho)$, with $C=C(0)$. We let $\Sigma(\rho)=\dot\Sigma\setminus C(\rho)$, so that $\Sigma(\rho)$ is a precompact sub-domain of $\dot\Sigma$. 

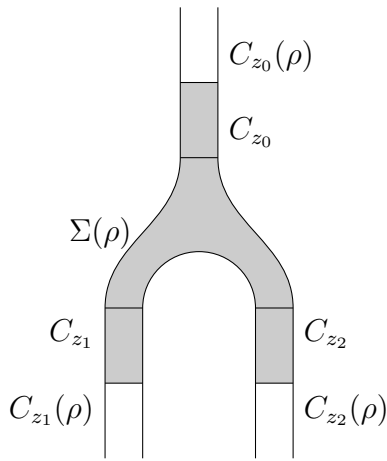
\begin{figure}[H]
  \centering
  \begin{tikzpicture}
    \fill[black!20!white] (0,-1)--(0,0) to [out=90,in=-90] (1,2)--+(0,1)--(1.5,3)--+(0,-1)to[out=-90,in=90] (2.5,0)--(2.5,-1)--(2,-1)--+(0,1) arc(0:180:0.75)--+(0,-1);
    \draw (0,-2)--(0,0) to [out=90,in=-90]node[pos=0.5,left]{$\Sigma(\rho)$} (1,2)--(1,4) (2.5,-2)--+(0,2)to[out=90,in=-90](1.5,2)--+(0,2);
    \draw (0.5,-2)--+(0,2) arc (180:0:0.75)--+(0,-2);
    \draw (2,0)--+(0.5,0)node[below right]{$C_{z_{2}}$};
    \draw (0,0)node[below left] {$C_{z_{1}}$}--+(0.5,0) (1,2)--+(0.5,0)node[above right]{$C_{z_{0}}$} (2,-1)--+(0.5,0)node[below right] {$C_{z_{2}}(\rho)$} (0,-1)node[below  left] {$C_{z_{1}}(\rho)$}--+(0.5,0) (1,3)--+(0.5,0)node[above right] {$C_{z_{0}}(\rho)$};
  \end{tikzpicture}
  \caption{A surface $\Sigma$ with $\Gamma_{+}=\set{z_{0}}$ and $\Gamma_{-}=\set{z_{1},z_{2}}$ and chosen cylindrical ends. The precompact sub-domain $\Sigma(\rho)$ is shown as the shaded region. The bundle $E$ has an equivalence class of unitary trivializations defined on the ends.}
  \label{fig:rho-deep}
\end{figure}

\subsection{Asymptotically Hermitian structures}\label{sec:ahs} Suppose that $(E,J)$ is a complex vector bundle of rank $n$ over $\dot\Sigma$ and $F\subset E|_{\bd\dot\Sigma}$ is a totally real sub-bundle. Similarly to \cite[Section 4.1]{wendl-sft}, we define an \emph{asymptotically Hermitian structure} on $(E,F,J)$ to be an equivalence class of trivializations $\tau$ of $E|_{C}\simeq \R^{2n}$ which identify $J$ with the standard complex structure $J_{0}$ and send $F$ to $\R^{n}$. Two trivializations are equivalent provided the transition map between them is an $s$-independent unitary transformation (i.e.,\ multiplication by a $t$-dependent family $\Omega(t)\in U(n)$). The inverse $X=\tau^{-1}$ of a trivialization will be called an \emph{asymptotic unitary frame}.

\begin{remark}
  There are variants of the notion of asymptotic Hermitian structure; for instance we could change the notion of equivalence to be that the transition between $\tau_{1}$ and $\tau_{2}$ is multiplication by $\Omega(s,t)$ and $\Omega\in W^{k,p}$ for all $k\ge 1$. This variant is weaker than the one defined above, but it is still strong enough to define $W^{k,p}$ spaces.
\end{remark}

\subsection{Sobolev spaces} We recall that the space of sections $W^{k,p}_{\mathrm{loc}}(E)$ is well-defined independently of any choice of auxiliary data on $\dot\Sigma$ for all $k\ge 0,p\ge 1$. These are sections which are of class $W^{k,p}_{\mathrm{loc}}$ in any coordinate chart equipped with a trivialization of $E$.

For $k-2/p>0$, the Sobolev embedding theorem (see \cite[Theorem B.1.11]{mcduffsalamon}) implies that $W^{k,p}_{\mathrm{loc}}(E)$ sections are continuous, and hence we can define $W^{k,p}_{\mathrm{loc}}(E,F)\subset W^{k,p}_{\mathrm{loc}}(E)$ as the sections taking boundary values in $F$.

For $k=1$ and $p\in [1,2]$, we say $\xi\in W^{1,p}_{\mathrm{loc}}(E,F)$ if, for any choice of $\cl{\mathbb{H}}$-valued coordinates equipped with trivializations identifying $E$ with $\R^{2n}$ and $F$ with $\R^{n}$, the doubling\footnote{To make this precise, we double $\xi$ in the sense of distributions.} of $\xi$ (i.e.,\ extension by $\cl{\xi}$ on $-\mathbb{H}$) is still of class $W^{1,p}$. It can be shown that this agrees with the other definition of $W^{1,p}_{\mathrm{loc}}(E,F)$ when $p>2$. See Remark \ref{remark:pg2} for more details.

Using the asymptotic Hermitian structure we can define Sobolev spaces which admit Banach space norms. We say $\xi\in W^{k,p}(E,F)$ if $\xi\in W^{k,p}_{\mathrm{loc}}(E,F)$ and $\tau\circ \xi\in W^{k,p}$ using the standard Euclidean structure on the cylindrical ends (for any asymptotic trivialization $\tau$). To define a Banach space topology on $W^{k,p}(E,F)$, we introduce the norm:
\begin{equation*}
  \norm{\xi}_{\tau,k,p,g,\mu,\nabla}:=\sum_{z\in \Gamma}\sum_{\ell=0}^{k}\sum_{a+b=\ell}\left[\int_{C_{z}}\abs{\bd_{s}^{a}\bd_{t}^{b}(\tau\circ \xi)}^{p}\d s\d t\right]^{1/p}+\norm{u}_{W^{k,p}_{g,\mu,\nabla}(\Sigma(1))}.
\end{equation*}
Here we make an arbitrary choice of metric $g$ on $\dot\Sigma$, and fiber-wise metric $\mu$ and connection $\nabla$ on $E\to\dot\Sigma$. It is straightforward to show that for any other choice of $g,\mu,\nabla$ we obtain an equivalent norm (since $\Sigma(1)$ is precompact). It is also not hard to show that two different choices of $\tau$ give equivalent norms.

The same process defines $W^{k,p}(E)$ for $k\ge 0$. We denote $W^{0,p}(E)=:L^{p}(E)$.

\subsection{Cauchy-Riemann operators with non-degenerate asymptotics}
\label{sec:asymptotics}

A first order partial differential operator $D:\Gamma(E)\to \Gamma(\Lambda^{0,1}\otimes E)$ is called a \emph{Cauchy-Riemann operator} if $$D(f\otimes \xi)=\d f\otimes \xi+(\d f\cdot j)\otimes J\xi+f\cdot D\xi$$ for all real-valued functions $f$ and smooth sections $\xi$. Here $j$ is the complex structure on $\Sigma$ and $J$ is the fiber-wise complex structure on $E$.

We begin with a discussion of the local coordinate representations of Cauchy-Riemann operators. If $z=s+it$ is a holomorphic coordinate, then $\d s-i\d t$ trivializes $\Lambda^{0,1}$. Suppose that $\tau:E\to \mathbb{C}^{n}$ is a complex linear trivialization over $E$. Then $\tau^{-1}(e_{k})=X_{k}$ and $\tau^{-1}(ie_{k})=JX_{k}$ define a local frame for $E$.

Write $\xi=\tau^{-1}(u)=\sum_{k} u_{k}X_{k}=\sum_{k}a_{k}X_{k}+b_{k}JX_{k}$, where $u=a+ib$ is a $\C^{n}$-valued function. We obtain:
\begin{equation*}
  \begin{aligned}
    D(\sum a_{k}X_{k})&=\sum\d a_{k}\otimes X_{k}+(\d a_{k}\cdot j)\otimes JX_{k}+a_{k}\cdot DX_{k}\\
    D(\sum b_{k}JX_{k})&=\sum\d b_{k}\otimes JX_{k}-(\d b_{k}\cdot j)\otimes X_{k}+b_{k}\cdot D(JX_{k}).
  \end{aligned}
\end{equation*}
Hence, using $c\otimes JX_{k}=ic\otimes X_{k}$ (for sections of $\Lambda^{0,1}\otimes E$) we obtain:
\begin{equation*}
  D(\tau^{-1}(u))=\sum (\d u_{k}+i\cdot \d u_{k}\cdot j)\otimes X_{k}+a_{k}\cdot DX_{k}+b_{k}\cdot D(JX_{k}).
\end{equation*}
It is straightforward to compute $\d u_{k}+i\cdot \d u_{k}\cdot j=(\bd_{s}u_{k}+i\bd_{t}u_{k})(\d s-i \d t)$. In particular, we have
\begin{equation*}
  D(\tau^{-1}(u))=\sum_{k=1}^{n} (\bd_{s}u_{k}+i\bd_{t}u_{k})\cdot (\d s-i\d t)\otimes X_{k}.
\end{equation*}
We note that $(\d s-i\d t)\otimes X_{k}$ is a local (complex) frame for $\Lambda^{0,1}\otimes E$. We denote the inverse trivialization by $\tau_{1}$. If conjugate $D$ by the complex trivializations $\tau^{-1}$ and $\tau_{1}^{-1}:w\mapsto w(\d s-i\d t)\otimes X$, we conclude that
\begin{equation}\label{eq:trivialization}
  \tau_{1}\circ D\circ \tau^{-1}=:D_{\tau}(u)=\bd_{s}u+i\bd_{t}u+S(s,t)u,
\end{equation}
where $S(s,t)$ is some smooth family of real linear matrices. Note that $D_{\tau}$ depends both on the holomorphic coordinates used on the base and on the trivialization $\tau$.

Now let $\tau$ be an asymptotic trivialization for $E$. Using the holomorphic coordinate in the ends $C$, we can compute the coordinate representation for $D$ using $\tau$. We say that $D$ has \emph{non-degenerate asymptotics} provided that \eqref{eq:trivialization} satisfies
\begin{equation}\label{eq:asymptotic}
  \sup_{t}\text{$\abs{\bd_{s}^{k}\bd_{t}^{\ell}(S(s,t)-S_{\infty}(t))}\to 0$ as $\abs{s}\to\infty$,}
\end{equation}
for all $k,\ell\in \mathbb{N}$, for some smooth family of \emph{symmetric} matrices $S_{\infty}$, and the corresponding \emph{asymptotic operator}
\begin{equation}\label{eq:asymptotic}
  A_{\tau}=-i\bd_{t}-S_{\infty}(t):\left\{
    \begin{aligned}
      C^{\infty}([0,1],\mathbb{C}^{n},\mathbb{R}^{n})\to C^{\infty}([0,1],\C^{n})\\
      C^{\infty}(\R/\Z,\mathbb{C}^{n})\to C^{\infty}(\R/\Z,\C^{n})
    \end{aligned}
  \right.
\end{equation}
is injective. In this case we say that $A_{\tau}$ is \emph{non-degenerate}. The two cases in \eqref{eq:asymptotic} are whether the cylindrical end corresponds to a boundary or interior puncture. The notation $C^{\infty}([0,1],\C^{n},\R^{n})$ means ``smooth $\C^{n}$-valued functions which are $\R^{n}$-valued at $t=0,1$.''

Since the transition function between two asymptotic trivializations is a smooth $s$-independent family of unitary matrices, the condition that $D$ has non-degenerate asymptotics is independent of the chosen $\tau$. Indeed, if $\tau_{1}\tau_{2}^{-1}=\Omega(t)$, then
\begin{equation*}
  A_{\tau_{2}}=\Omega(t)^{-1}A_{\tau_{1}}\Omega(t)=-i\bd_{t}-i\Omega(t)^{-1}\Omega'(t)-\Omega(t)^{-1}S_{\infty}(t)\Omega(t).
\end{equation*}
A straightforward computation shows that $i\Omega(t)^{-1}\Omega'(t)-\Omega(t)^{-1}S_{\infty}(t)\Omega(t)$ is still symmetric. 

\subsection{Some facts about non-degenerate asymptotic operators}
\label{sec:asymptotic-ops}

In this section we fix a non-degenerate asymptotic operator $A=-i\bd_{t}-S(t)$ on $[0,1]$. The analogous results with $[0,1]$ replaced by $\R/\Z$ are left to the reader.

\begin{prop}\label{lemma:nondegenerate-A}
  The map $A:C^{\infty}([0,1],\C^{n},\R^{n})\to C^{\infty}([0,1],\C^{n})$ extends to a self-adjoint isomorphism $$A:W^{1,2}([0,1],\C^{n},\R^{n})\to L^{2}([0,1],\C^{n}).$$
  By self-adjoint we mean that $\ip{Av,w}=\ip{v,Aw}$ for all $v,w\in W^{1,2}([0,1],\C^{n},\R^{n})$.
\end{prop}
See \cite[Corollary 3.14]{wendl-sft} for an alternative approach, yielding a proof in the $\R/\Z$ case.
\begin{proof}  
  It is clear that $A$ extends to a bounded linear operator between the advertised Banach spaces. Since $C^{\infty}([0,1],\C^{n},\R^{n})$ is dense in $W^{1,2}([0,1],\C^{n},\R^{n})$ it suffices to prove the self-adjointness for smooth functions $u,v$. This follows from a straightforward integration-by-parts computation, using the fact that the matrix $S(t)$ is symmetric, and $i$ is anti-symmetric. We leave this computation to the reader. Note that it is crucial that both $u,v$ take boundary values in $\R^{n}$, otherwise the integration by parts will fail.

  It suffices to prove that $A$ is a bijection, since continuous bijections between Banach spaces are isomorphisms. Observe that any element in the kernel of $A$ must be smooth (by 1-dimensional elliptic regularity). Since we assume that $u\mapsto Au$ is injective for smooth $u$, we conclude that $A$ is injective on $W^{1,2}$.
  
  Now we will prove that $A$ is surjective. Fix a smooth $\eta$, and we attempt to solve $A(\xi)=\eta$ for a smooth $\xi$:
  \begin{equation}\label{eq:computation}
    \begin{aligned}
      &\hphantom{\iff}\ i\pd{\xi}{t}+S(t)\xi=-\eta(t)
      \iff \pd{\xi}{t}-iS(t)\xi=i\eta(t).\\
      &\iff \pd{}{t}(\exp(\Sigma(t))\xi(t))=\exp(\Sigma(t))i\eta(t),\\
      &\iff \xi(t)=\exp(-\Sigma(t))\xi(0)+\exp(-\Sigma(t))\int_{0}^{t}\exp(\Sigma(t'))i\eta(t')\d t'.
    \end{aligned}
  \end{equation}
  where $\Sigma'(t)=-iS(t)$ and $\Sigma(0)=0$. This shows that we can solve $A(\xi)=\eta$ for many different choices of $\xi$, namely there is an $\R^{2n}$ dimensional family of solutions corresponding to the choice of $\xi(0)$. We claim that (exactly) one of these solutions will satisfy $\xi(0),\xi(1)\in \R^{n}$. To see why, consider the affine map:
  \begin{equation*}
    F:\xi(0)\in \R^{n}\mapsto \exp(-\Sigma(1))\xi(0)+\exp(-\Sigma(1))\int_{0}^{1}\exp(\Sigma(t'))i\eta(t')\d t'\in \R^{2n}.
  \end{equation*}
  This map parameterizes an $n$-dimensional affine subspace of $\R^{2n}$. Note that the associated linear subspace $\exp(-\Sigma(1))\R^{n}$ is transverse to $\R^{n}$ (otherwise we could find a vector $v\in \R^{n-1}$ so $\exp(-\Sigma(1))v\in \R^{n}$, and the above computation with $\eta=0$ would imply $\xi(t)=\exp(-\Sigma(t))v$ lies in the kernel of $A$).

  Therefore $F(\R^{n})$ intersects $\R^{n}$ in a unique point $F(\xi(0))=\xi(1)$. Thus \eqref{eq:computation} with this special $\xi(0)$ shows that $A$ is surjective onto the smooth elements $\eta$. 

  To show that $A$ is surjective in general, it suffices to prove that the image of $A$ is closed. This follows from the estimate
  \begin{equation*}
    \norm{\xi}_{W^{1,2}}\le C(\norm{A(\xi)}_{L^{2}}+\norm{\xi}_{L^{2}})
  \end{equation*}
  and the fact that $W^{1,2}\to L^{2}$ is a compact inclusion. This completes the proof of the lemma.
\end{proof}

\begin{prop}\label{prop:spectral-prop}
  There exists an orthonormal basis of $L^{2}([0,1],\R^{n})$ consisting of (smooth) eigenvectors of $A$. The union of all the eigenvalues is a discrete set $\Lambda\subset \R$ disjoint from $0$.
\end{prop}
\begin{proof}
  The key observation is that the following composition is a compact self-adjoint operator (called the \emph{resolvent} of $A$):
  \begin{equation*}
    \begin{tikzcd}
      {L^{2}}\arrow[r,"{A^{-1}}"] &{W^{1,2}}\arrow[r,"{\subset}"] &{L^{2}}.
    \end{tikzcd}
  \end{equation*}
  This is because $W^{1,2}\subset L^{2}$ is a compact inclusion (Proof: if $\bd_{t}f_{n}$ is bounded in $L^{2}$ then $\abs{f_{n}(x)-f_{n}(x+t)}\le ct^{1/2}$, and hence $f_{n}$ is equicontinuous). Self-adjoint compact operators have orthonormal eigenbases whose spectrum accumulates only at~$0$ (see \cite[Theorem 3.2.3]{barry-simon-4}). The desired result follows.
\end{proof}

\subsection{Formal adjoints}
\label{sec:formal-adjoints}
The purpose of this section is to define the \emph{formal adjoint} of a Cauchy-Riemann operator. Formal adjoints will play an important role in establishing the Fredholm property. A good reference in the case when $\bd\Sigma=\emptyset$ is \cite[Section 4.7]{wendl-sft}.

Let $D$ be a Cauchy-Riemann operator for the data $(\Sigma,\bd\Sigma,\Gamma_{\pm},C,E,F,[\tau])$ as explained above. Fix a $j$-invariant Riemannian metric $g$ on $T\Sigma$ which is the Euclidean metric in the cylindrical ends. The corresponding volume form is given by $\d\mathrm{vol}=g(j-,-)$.

Pick a Hermitian structure on $(E,F)$ which agrees with the asymptotically Hermitian structure in the cylindrical ends $C$. This means that $E$ is equipped with a fiber-wise metric $g$ which is $J$-invariant and $F$ is $g$-orthogonal to $JF$. In other words, $F$ is Lagrangian for the symplectic form $g(J-,-)$.

We define $\C$-valued Hermitian metrics on $E$ (and $T\dot\Sigma$) by the formulas: $$\mu(X,Y)=g(X,Y)+ig(JX,Y).$$ By our conventions, $\mu(X,JY)=i\mu(X,Y)$ and $\mu(JX,Y)=-i\mu(X,Y)$.

The bundle isomorphism $Y\mapsto \mu(-,Y)$ identifies $T\dot\Sigma$ with $\Lambda^{0,1}$. We use this to push forward a Hermitian metric onto $\Lambda^{0,1}$. 

Given two complex vector bundles $E_{1},E_{2}$ with Hermitian metrics $\mu_{1},\mu_{2}$ we can endow a Hermitian metric $\mu_{1}\otimes \mu_{2}$ on $E_{1}\otimes E_{2}$ by the formula:
\begin{equation*}
  \mu_{1}\otimes \mu_{2} (X_{1}\otimes X_{2},Y_{1}\otimes Y_{2})=  \mu_{1}(X_{1},Y_{1})\mu_{2}(X_{2},Y_{2}).
\end{equation*}
Via this construction, the bundles $E$ and $\Lambda^{0,1}\otimes E$ are both equipped with Hermitian metrics.

With these preliminaries out of the way, we say that $D^{*}:\Gamma(\Lambda^{0,1}\otimes E)\to \Gamma(E)$ is a \emph{formal adjoint} of $D$ if
\begin{equation}\label{eq:formal-adjoint}
  \mathrm{Re}\int_{\dot\Sigma}\mu(D(\xi),\eta)\,\d \mathrm{vol}=  \mathrm{Re}\int_{\dot\Sigma}\mu(\xi,D^{*}(\eta))\,\d \mathrm{vol},
\end{equation}
for all $\xi\in \Gamma_{0}(E,F)$ and $\eta\in \Gamma_{0}(\Lambda^{0,1}\otimes E,F^{*})$. Here $F^{*}\subset \Lambda^{0,1}\otimes E$ is the totally-real sub-bundle of maps which map $T\bd\Sigma$ into $F$, and $\Gamma_{0}(E,F)$ is the set of smooth compactly supported sections of $E$ which take boundary values in $F$.

Since $\mathrm{Re}(\mu)$ is a Riemannian metric, formal adjoints are necessarily unique. We will derive a formula for the formal adjoint in local trivializations below. By patching together the local descriptions we deduce that formal adjoints always exist.

Let $z=s+it$ be a holomorphic coordinate and $\tau:E\to \C^{n}$ a local unitary trivialization of $E$ defined on the domain of $z$. Let $X_{1},\cdots,X_{n}$ be the unitary frame of $E$ induced by $\tau$. Recall that we have an associated trivialization $\tau_{1}:\Lambda^{0,1}\otimes E\to \C^{n}$ which satisfies
\begin{equation*}
  \tau_{1}^{-1}(w)=\sum_{k}w_{k}(\d s-i\d t)\otimes X_{k}.
\end{equation*}
The equation \eqref{eq:trivialization} shows $\tau_{1}\circ D\circ \tau^{-1}(u)=\bd_{s}u+i\bd_{t}u+S(s,t)u$.

To incorporate the boundary conditions, we require that $z$ takes values in $\R\times [0,1]$, $z(\bd \dot\Sigma)\subset \R\times \set{0,1}$, and $\tau$ identifies $F$ with $\R^{n}$. We do not require that $z$ is surjective, e.g.,\ it could take values in $D(1)\cap \cl{\mathbb{H}}$, or $i/2+D(1/2)$.

\begin{lemma}
  If $D^{*}$ is a formal adjoint for $D$, then for sections $w$ with compact support in the above coordinate chart we have
  \begin{equation}\label{eq:formal-formula}
    \abs{\bd_{s}}^{2}\tau \circ D^{*}\circ \tau_{1}^{-1}(w)=-\bd_{s}w+i\bd_{t}w+S(s,t)^{T}w,
  \end{equation}
  where $\abs{\bd_{s}}^{2}=\mu(\bd_{s},\bd_{s})$ and $S(s,t)^{T}$ is the transpose matrix.
\end{lemma}
\begin{proof}
  The first thing we do is derive formulas for the Hermitian metrics $\mu$. Because $\tau$ is a unitary transformation, we have 
  \begin{equation*}
    \mu(\tau^{-1}(u),\tau^{-1}(v))=\sum_{k}\cl{u_{k}}v_{k}=:\mu_{0}(u,v).
  \end{equation*}
  Unfortunately, $\tau_{1}$ is not a unitary transformation because $\d s-i\d t$ is not a unitary frame of $\Lambda^{0,1}$. We easily compute $\mu(-,\bd s)=\abs{\bd_{s}}^{2}(\d s-i\d t)$ (by inserting $\bd_{s},\bd_{t}$ into both sides). Since the Hermitian metric on $\Lambda^{0,1}$ is pushed forward from $T\Sigma$ we have
  \begin{equation*}
    \abs{\bd_{s}}^{4}\mu(\d s-i\d t,\d s-i\d t)=\mu(\bd_{s},\bd_{s})=\abs{\bd_{s}}^{2}\implies \abs{\d s-i\d t}^{2}=\abs{\bd_{s}}^{-2}.
  \end{equation*}
  We can therefore compute
  \begin{equation*}
    \mu(\tau_{1}^{-1}(u),\tau_{1}^{-1}(v))=\sum_{k,\ell}\mu(u_{k}(\d s-i\d t)\otimes X_{k},v_{\ell}(\d s-i\d t)\otimes X_{\ell})=\abs{\bd_{s}}^{-2}\mu_{0}(u,v).
  \end{equation*}
  It is also easy to compute that $\d\mathrm{vol}=\abs{\bd_{s}}^{2}\d s\d t$.

  Now let $w$ and $u$ be $\C^{n}$ valued functions which takes values in $\R^{n}$ on $\R\times \set{0,1}$. We compute
  \begin{equation*}
    \mu(\tau^{-1}(u),D^{*}\circ \tau_{1}^{-1}(w))=\mu_{0}(u,\tau\circ D^{*}\circ \tau_{1}^{-1}(w)).
  \end{equation*}
  On the other hand, we have
  \begin{equation*}
    \mu(D\circ \tau^{-1}(u),\tau_{1}^{-1}(w))=\abs{\bd_{s}}^{-2}\mu_{0}(\tau_{1}\circ D\circ \tau^{-1}(u),w).
  \end{equation*}
  Since real linear combinations of $X_{k}$ lie in $F$, we are allowed to apply the formal adjoint property with $\xi=\tau^{-1}\circ u$ and $\eta=\tau_{1}^{-1}\circ w$:
  \begin{equation*}
    \mathrm{Re}\int \mu(D\circ \tau^{-1}(u),\tau_{1}^{-1}(w))\d \mathrm{vol}=\mathrm{Re}\int \mu(\tau^{-1}(u),D^{*}\circ \tau_{1}^{-1}(w))\d
    \mathrm{vol}.
  \end{equation*}
  This implies that:
  \begin{equation*}
    \mathrm{Re}\int \mu_{0}(\tau_{1}\circ D\circ \tau^{-1}(u),w)\d s\wedge\d t=\mathrm{Re}\int\mu_{0}(u,\tau\circ D^{*}\circ \tau_{1}^{-1}(w))\abs{\bd_{s}}^{2}\d s\wedge\d t.
  \end{equation*}
  In particular, $\abs{\bd_{s}}^{2}\tau\circ D^{*}\circ \tau_{1}^{-1}$ is the formal adjoint of $D_{\tau}:=\tau_{1}\circ D\circ \tau^{-1}$ with respect to the standard metric $\mu_{0}$ and volume form $\d s\wedge \d t$. Equation \eqref{eq:trivialization} gives a formula for $D_{\tau}$ and so we can explicitly compute its adjoint:
  \begin{equation*}
    \mathrm{Re}\int \mu_{0}(\bd_{s}u+i\bd_{t}u+S(s,t)u,w)\d s\d t=\mathrm{Re}\int \mu_{0}(u_{k},-\bd_{s}w+i\bd_{t}u+S(s,t)^{T}w)\d s\d t.
  \end{equation*}
  The boundary terms in the integration by parts are given by
  \begin{equation*}
    \mathrm{Re}\int_{\R\times \set{0,1}} \mu_{0}(iu,w)\d s\d t=0.
  \end{equation*}
  It follows that $\abs{\bd_{s}}^{2}\tau\circ D^{*}\circ \tau_{1}^{-1}=-\bd_{s}+i\bd_{t}+S(s,t)^{T}$ as desired.
\end{proof}

As a consequence, if $s+it$ is the coordinate system in a cylindrical end $C_{z}$ and $\tau$ is an asymptotic trivialization, then
\begin{equation*}
  \tau\circ D^{*}\circ \tau_{1}^{-1}=-\bd_{s}+i\bd_{t}+S(s,t)^{T}\to -\bd_{s}-A\text{ as $s\to\infty$},
\end{equation*}
where $A_{\tau}w=-i\bd_{t}w-S_{\infty}(t)w$ is the asymptotic operator for $D_{\tau}$ in the end $C_{z}$.


\section{Regularity and the Fredholm property}
The references for this section are \cite[Section 2.3]{salamon1997}, \cite[Chapter 4]{wendl-sft}, \cite[Chapter 3]{schwarz-diss}, and \cite[Appendix B]{mcduffsalamon} (for the local $L^{p}$ elliptic estimates).

\subsection{Local elliptic estimates}
Our first result is the following local elliptic estimate for $u\mapsto \bd_{s}u+i\bd_{t}u$. 
\begin{theorem}\label{theorem:local-reg}
  Fix $r<1$ and $q>1$. There is a constant $c_{q,r}$ so that for all smooth maps $u:\cl{D}(1)\cap \cl{\mathbb{H}}\to \C^{n}$ satisfying $u(D(1)\cap \R)\subset \R^{n}$ we have
  \begin{equation*}
    \int_{D(r)\cap \cl{\mathbb{H}}}\abs{u}^{q}+\abs{\bd_{x}u}^{q}+\abs{\bd_{y}u}^{q}\d x\d y\le c_{q,r}\int_{D(1)\cap \cl{\mathbb{H}}}\abs{u}^{q}+\abs{\bd_{x}u+i\bd_{y}u}^{q}\d x \d y.
  \end{equation*}
\end{theorem}
\begin{proof}
  The theorem follows from \cite[Theorem B.3.2]{mcduffsalamon} which concerns weak solutions of the equation
  \begin{equation*}
    \Delta w=f_{0}+\bd_{x}f_{1}+\bd_{y}f_{2},
  \end{equation*}
  with $w,f_{0},f_{1},f_{2}\in L^{p}(D(1))$. The conclusion is that the $W^{1,q}$ size of $w$ on a smaller disk is bounded by the $L^{q}$ sizes of $w,f_{0},f_{1},f_{2}$. This uses the Calderon-Zygmund inequality proved in \cite[Section B.2]{mcduffsalamon}.

  To apply their result to our setting, we extend $u$ across the boundary by $$u(x,-y)=\cl{u(x,y)}.$$ The extended function is no longer smooth. Let $\eta=\bd_{x}u+i\bd_{y}u$ (which potentially has a jump discontinuity along $\R$, but is still in $L^{q}$). We note that
  \begin{equation*}
    \eta(x,-y)=\bd_{x}\cl{u}-i\bd_{y}\cl{u}=\cl{\bd_{x}u+i\bd_{y}u}=\cl{\eta(x,y)}.
  \end{equation*}
  In particular, the size $\abs{\eta}$ is invariant under $y\mapsto -y$.

  Using the fact that $(\bd_{x}-i\bd_{y})(\bd_{x}+i\bd_{y})=\Delta$ we have
  \begin{equation*}
    \int_{D(1)} g_{0}(u,\Delta\phi)\,\d x\d y=-\int_{D(1)}g_{0}(\eta,(\bd_{x}+i\bd_{y})\phi)\d x \d y.
  \end{equation*}
  To see why, apply Stokes' Theorem separately on the upper and lower half-disks, and then observe that the boundary terms will cancel; this uses $u(D(1)\cap \R)\subset \R^{n}$. The equality above satisfies the hypothesis of \cite[Theorem B.3.2]{mcduffsalamon} and allows us to conclude that $W^{1,q}$ size of $u$ is controlled by the $L^{q}$ size of $\eta$ and $u$. This implies the desired result.
\end{proof}

\subsection{Local elliptic regularity}
\label{sec:local-ell-reg}

In this section we wish to prove that weak solutions of $D^{*}(\eta)=f$ are in fact smooth, provided $f$ is smooth. In order to talk about $D^{*}(\eta)$, we require the choice of Hermitian metrics $\mu$ on $E,T\dot\Sigma$, as in Section \ref{sec:formal-adjoints}.

More precisely, we wish to prove the following:
\begin{prop}\label{prop:coord-inv-ell-reg}
  Let $q>1$. If $\eta\in L^{q}_{\mathrm{loc}}(\Lambda^{0,1}\otimes E)$, $f$ is smooth, and $D^{*}(\eta)=f$ weakly in the sense that
  \begin{equation*}
    \mathrm{Re}\int_{\dot\Sigma} \mu(D(\xi),\eta)\d\mathrm{vol}=\mathrm{Re}\int_{\dot\Sigma} \mu(\xi,f)\d\mathrm{vol}
  \end{equation*}
  for all $\xi\in \Gamma_{0}(E,F)$, then $\eta$ is in smooth and lies in $\Gamma(\Lambda^{0,1}\otimes E,F^{*})$.

  The same holds true with $(\eta,\Lambda^{0,1}\otimes E, F^{*},D^{*})$ swapped with $(\xi,E,F,D)$ throughout the statement.
\end{prop}
\begin{remark}
  Note that ``weakly'' solving the equation implicitly incorporates the boundary conditions, since we allow the test functions to be non-zero along the boundary. We do require, however, that the test functions take values the appropriate sub-bundle $F$.
\end{remark}

Since smoothness is a local property, we can prove Proposition \ref{prop:coord-inv-ell-reg} by restricting our attention to a coordinate chart $z=s+it$ on which we have a unitary trivialization $\tau:E\to \C^{n}$. Without loss of generality, let us suppose that $z$ takes values in $D(1)\cap \cl{\mathbb{H}}$. Writing $\tau(\xi)=u$, $\tau_{1}(\eta)=w$, and $f:=\tau(f)$, we compute
\begin{equation*}
  \begin{aligned}
    D^{*}(\eta)=f\text{ weakly}&\implies -\bd_{s}w+i\bd_{t}w+S(s,t)^{T}w=\abs{\bd_{s}}^{-2}f\text{ weakly}\\
    D(\xi)=f\text{ weakly}&\implies \bd_{s}u+i\bd_{t}u+S(s,t)u=f\text{ weakly}.
  \end{aligned}
\end{equation*}
We can simplify this a bit further by observing that
\begin{equation*}
  \begin{aligned}
    -\bd_{s}w+i\bd_{t}w+S(s,t)^{T}w&=\abs{\bd_{s}}^{-2}f\\
    \iff \bd_{s}\cl{w}+i\bd_{t}\cl{w}-CS(s,t)^{T}C\cl{w}&=-\abs{\bd_{s}}^{-2}\cl{f},
  \end{aligned}
\end{equation*}
where $C$ is the matrix representing complex conjugation. Thus, Proposition \ref{prop:coord-inv-ell-reg} follows from:
\begin{lemma}\label{lemma:coord-ell-reg}
  Let $q>1$. Write $\Omega(r)=D(r)\cap \cl{\mathbb{H}}$, and suppose that
  \begin{equation*}
    \text{$u\in L^{q}(\Omega(1),\mathbb{C}^{n})$, $f\in C^{\infty}(\cl{\Omega(1)},\mathbb{C}^{n})$, and $S\in C^{\infty}(\cl{\Omega(1)},\R^{2n\times 2n})$}
  \end{equation*}
  satisfy
  \begin{equation}\label{eq:weak-eq}
    \bd_{s}u+i\bd_{t}u+S(s,t)u=f\text{ weakly,}
  \end{equation}
  in the sense that
  \begin{equation*}
    \mathrm{Re}\int_{\Omega(1)} \mu_{0}(u,-\bd_{s}\varphi+i\bd_{t}\varphi+S(s,t)^{T}\varphi)=\mathrm{Re}\int_{\Omega(1)} \mu_{0}(f,\varphi),
  \end{equation*}
  for all compactly supported test functions $\varphi$ which take values in $\R^{n}$ on $\Omega(1)\cap \R$. Then $u$ is smooth and takes boundary values in $\R^{n}$. Moreover for $k\in \mathbb{N}$ and $r<1$, there exists a constant $c=c(k,q,S)$ so that
  \begin{equation}\label{eq:elliptic-estimate}
    \norm{u}_{W^{k,q}(D(r)\cap \mathbb{H})}\le c(\norm{u}_{L^{q}(D(1)\cap \mathbb{H})}+\norm{f}_{W^{k-1,q}(D(1)\cap \mathbb{H})}).
  \end{equation}
\end{lemma}
\begin{proof}  
  Throughout the argument we will need to shrink the domain countably many times. For this purpose, fix a sequence $1>r_{1}>r_{2}>\cdots>r_{\infty}=r$. Each time we need to shrink the domain we will pass from $\Omega(r_{j})$ to $\Omega(r_{j+1})$. To obtain the constant in \eqref{eq:elliptic-estimate}, we will only need to shrink the domain finitely many times. 
  
  Our first goal is to upgrade $u$ to a $W^{1,q}$ distribution. We observe that
  \begin{equation*}
    (\bd_{s}+i\bd_{t})u=-S(s,t)u+f\text{ weakly}.
  \end{equation*}
  Notice that the right hand side lies in $L^{q}$. More generally, let us consider equations of the form
  \begin{equation*}
    (\bd_{s}+i\bd_{t})u=F\text{ weakly},
  \end{equation*}
  where $F\in L^{q}$. Our strategy is to approximate $u$ by a sequence of smooth sections $u_{n}$ taking real values along the boundary so that:
  \begin{enumerate}
  \item $u_{n}\to u\in L^{q}$ and
  \item $\norm{(\bd_{s}+i\bd_{t})u_{n}}_{L^{q}(\Omega(r_{1}))}$ is bounded by $c_{1}\norm{F}_{L^{q}(\Omega(1))}$.
  \end{enumerate}
  We will explain how to do this approximation at the end of the proof. See \cite[Section 2.4]{wendl-sft} for another approach in the case with $\cl{\mathbb{H}}$ replaced by $\C$. See \cite[Section B.4]{mcduffsalamon} for a similar bootstrapping argument (in a non-linear context).
  The estimate from Theorem \ref{theorem:local-reg} then implies that $u_{n}$ is bounded in the $W^{1,q}$ topology (on a smaller domain $\Omega(r_{2})$). Indeed, we have
  \begin{equation*}
    \norm{u_{n}}_{W^{1,q}(\Omega(r_{2}))}\le c_{2}(\norm{u_{n}}_{L^{q}(\Omega(1))}+\norm{F}_{L^{q}(\Omega(1))}).
  \end{equation*}

  Since the $W^{1,q}$ spaces are reflexive, the Banach-Alaoglu theorem implies that some subsequence of $u_{n}$ converges in the weak topology to an element $u^{\prime}\in W^{1,q}$. Since $(L^{q})^{*}\subset (W^{1,q})^{*}$ we conclude that $\lim_{n\to\infty}\ip{u_{n},w}=\ip{u^{\prime},w}$ for all $w\in (L^{q})^{*}$. However, the same holds with $u^{\prime}$ replaced by $u$ (because $u_{n}$ converges to $u$ in the $L^{q}$ norm). Thus $u=u^{\prime}$, and hence $u\in W^{1,q}$. Moreover, the Banach-Alaoglu theorem implies the $W^{1,q}$ norm of $u$ is bounded above by $\limsup \norm{u_{n}}_{W^{1,q}}$, and hence we conclude that
  \begin{equation*}
    \norm{u}_{W^{1,q}(\Omega(r_{2}))}\le c_{2}(\norm{u}_{L^{q}(\Omega(1))}+\norm{F}_{L^{q}(\Omega(1))}).
  \end{equation*}
  
  Now, suppose that we have shown that $u$ is of class $W^{k,q}$ on some region $\Omega(r_{2k})$. Moreover, suppose that the $\norm{u}_{W^{k,q}(\Omega(r_{2k}))}$ is bounded by $c(\norm{u}_{L^{q}(\Omega(1))}+\norm{f}_{W^{k-1,q}(\Omega(1))})$ for some $c$. Then we can differentiate the equation \eqref{eq:weak-eq} $k$ times in the $s$-direction to conclude:
  \begin{equation}\label{eq:weak-k-eq}
    (\bd_{s}+i\bd_{t})\bd_{s}^{k}u=\bd_{s}^{k}f-\sum_{\ell=0}^{k}\bd_{s}^{\ell}S(s,t)\cdot\bd_{s}^{k-\ell}u=F_{k}\text{ weakly}.
  \end{equation}
  This differentiation is a bit subtle because the ``weak'' condition incorporates the boundary conditions; we will explain this step in greater detail at the end of the proof. 

  By our assumption on $u$, the right hand side is in $L^{q}$. The same argument given above implies that $\bd_{s}^{k}u$ is in $W^{1,q}$ on a smaller region $\Omega(r_{2k+2})$ and that
  \begin{equation*}
    \norm{\bd_{s}^{k}u}_{W^{1,q}(\Omega(r_{2k+2}))}\le c^{\prime}(\norm{\bd_{s}^{k}u}_{L^{q}(\Omega(r_{2k}))}+\norm{\bd_{s}^{k}f}_{L^{q}(\Omega(r_{2k}))}+C(S)\norm{u}_{W^{k,q}(\Omega(r_{2k}))}).
  \end{equation*}  
  Now it is straightforward to use \eqref{eq:weak-eq} to establish that, for $a+b=k$,
  \begin{equation*}
    \bd_{s}^{a}\bd_{t}^{b}u=i^{b}\bd_{s}^{k}u+\text{lower order terms}.
  \end{equation*}
  This equality should be interpreted as saying that both sides agree when integrated against a test function which is supported in the interior of the domain (i.e.,\ we do not need to worry about the boundary).
  Since $\bd_{s}^{k}u$ and the ``lower order terms'' are of class $W^{1,q}$ we conclude that all the $k$th order derivatives are in $W^{1,q}(\Omega(r_{2k+2}))$, and hence $u$ is in $W^{k+1,q}(\Omega(r_{2k+2}))$. Keeping track of the various estimates implies that
  \begin{equation*}
    \norm{u}_{W^{k+1,q}(\Omega(r_{2k+2}))}\le c^{\prime\prime}(\norm{u}_{L^{q}(\Omega(1))}+\norm{f}_{W^{k,q}(\Omega(1))}).
  \end{equation*}
  The Sobolev embedding theorem \cite[Theorem B.1.11]{mcduffsalamon} implies that $u$ is smooth on $\Omega(r_{\infty})$. Part of the conclusion of the Sobolev embedding theorem is that $u$ extends smoothly to the boundary. We claim that $u$ takes $\R^{n}$ values along the boundary. This follows from \eqref{eq:weak-eq}; pick any test function $\varphi$ taking boundary values in $\R^{n}$. It is easy to see (by integration by parts) that
  \begin{equation*}
    \mathrm{Re}\int_{D(1)\cap \cl{\mathbb{H}}}\mu_{0}(\bd_{s}u+i\bd_{t}u,\varphi)-\mu_{0}(u,-\bd_{s}\varphi+i\bd_{t}\varphi)\d s\d t=\mathrm{Re}\int_{D(1)\cap \mathbb{R}}\mu_{0}(u,i\varphi)\d s.
  \end{equation*}
  If $u$ did not take $\R^{n}$-values along $D(1)\cap \R$, we could pick $\varphi$ so that the right hand side was non-zero. This would contradict \eqref{eq:weak-eq}. 

  This completes the proof, modulo our explanation of how to pick the approximations $u_{n}\to u$ so that (i) and (ii) hold, and also why we can differentiate the weak equation with respect to $\bd_{s}$ to obtain \eqref{eq:weak-k-eq}.

  First we explain how to differentiate the weak equation. Suppose that $(\bd_{s}+i\bd_{t})w=F$ weakly and $w,F\in W^{1,q}$. Then for any test function $\varphi$ taking real-values along the boundary, $\bd_{s}\varphi$ still takes real values along the boundary, and hence
  \begin{equation}\label{eq:partial}
    \mathrm{Re}\int\mu_{0}(w,(-\bd_{s}+i\bd_{t})\bd_{s}\varphi)\d s\d t=\mathrm{Re}\int\mu_{0}(F,\bd_{s}\varphi)\d s\d t.
  \end{equation}
  The distributional derivative $\bd_{s}$ is defined (by duality) by how it integrates against sections $\psi$ supported in the \emph{interior} of $\Omega(r)$:
  \begin{equation*}
    \mathrm{Re}\int \mu_{0}(F,\bd_{s}\psi)\d s\d t=\mathrm{Re}\int -\mu_{0}(\bd_{s}F,\psi)\d s\d t,
  \end{equation*}
  However, the above holds even if $\psi$ is non-zero along the boundary $\Omega(r)\cap \R$. To see why, observe that
  \begin{equation*}
    \mathrm{Re}\int_{\Omega(r)} \mu_{0}(\bd_{s}F,\psi)\d s\d t=\mathrm{Re}\lim_{\delta\to 0}\int_{\Omega(r)} \mu_{0}(\bd_{s}F,\beta(t/\delta)\psi)\d s\d t,
  \end{equation*}
  where $\beta:[0,\infty)\to [0,1]$ vanishes near $0$ and equals $1$ on $[1,\infty)$. Since $\beta(t/\delta)$ is independent of $s$, we can integrate by parts and conclude
  \begin{equation*}
    \begin{aligned}
      \mathrm{Re}\int_{\Omega(r)} \mu_{0}(\bd_{s}F,\psi)\d s\d t&=\mathrm{Re}\lim_{\delta\to 0}\int_{\Omega(r)}\mu_{0}(F,\beta(t/\delta)\bd_{s}\psi)\d s\d t\\ &=\mathrm{Re}\int_{\Omega(r)}\mu_{0}(F,\bd_{s}\psi)\d s\d t.
    \end{aligned}
  \end{equation*}
  In particular, this observation applied to \eqref{eq:partial} yields
  \begin{equation*}
    \mathrm{Re}\int\mu_{0}(\bd_{s}w,(-\bd_{s}+i\bd_{t})\varphi)\d s\d t=\mathrm{Re}\int\mu_{0}(\bd_{s}F,\varphi)\d s\d t,
  \end{equation*}
  which implies that $(\bd_{s}+i\bd_{t})\bd_{s}w=\bd_{s}F$ still holds weakly.

  Finally, we explain how to choose the approximations $u_{n}\to u$ so that (i) and (ii) hold. First we extend $u$ as an $L^{q}$ distribution to $D(1)$ by $E(u)(s,-t)=\cl{u}(s,t)$ for $t\le 0$. This can be defined in the sense of distributions as
  \begin{equation*}
    \ip{E(u),\varphi}=\mathrm{Re}\int_{D(1)\cap \Omega} \mu_{0}(u(s,t),\varphi(s,t)+\cl{\varphi}(s,-t))\d s\d t.
  \end{equation*}
  Let $\Phi$ be a radially symmetric bump function of unit mass supported in $D(1)$, and let $\Phi_{n}(s,t)=\Phi(sn,tn)$. Then define $u_{n}=\Phi_{n}\ast E(u).$ Clearly (i) holds. It can be shown that
  \begin{equation*}
    \ip{\Phi_{n}\ast E(u),(-\bd_{s}+i\bd_{t})\varphi}=\ip{E(u),(-\bd_{s}+i\bd_{t})(\Phi_{n}\ast \varphi)}=\ip{E(F),(\Phi_{n}\ast \varphi)}.
  \end{equation*}
  This uses the distributional definition of $E(u)$ and the assumption that $(\bd_{s}+i\bd_{t})u=F$ weakly. It also uses the fact that convolution commutes with $\bd_{s}+i\bd_{t}$ (as it is a differential operator with constant coefficients). 
  
  We therefore conclude that $L^{q}$ size of $(\bd_{s}+i\bd_{t})(\Phi_{n}\ast E)(u)$ is bounded by the $L^{q}$ size of $F$. This proves (ii). We observe that since $\Phi_{n}$ is a radially symmetric and $u(s,-t)=\cl{u}(s,t)$, $u_{n}$ must take real values along the real axis. This completes the proof.
\end{proof}
Note that a consequence of the above proof is the following smooth approximation result:
\begin{prop}\label{prop:smooth-approximation}
  Let $\Omega(r)=\cl{\mathbb{H}}\cap D(r)$ and $q>1$. Suppose that $u\in L^{q}(\Omega(1),\C^{n})$ has the property that $\bd_{s}u+i\bd_{t}u=F$ holds weakly for some $F\in L^{q}(\Omega(1),\C^{n})$. Then for any $r<1$, the doubling $E(u)$ lies in $W^{1,q}(D(r))$ and there is a family of smooth functions $u_{n}$ on $\Omega(r)$ taking real values on $\bd\Omega(r)$ so that $u_{n}\to u$ in $W^{1,q}(\Omega(r))$.
\end{prop}
\begin{proof}
  Let $E(u)$ be the doubling of $u$, as in the previous proof, and recall that $\Phi_{n}\ast E(u)$ converges to $E(u)$ in $L^{q}(D(r^{\prime}))$ and is bounded in $W^{1,q}(D(r^{\prime}))$. As we argued above, this implies that some subsequence of $\Phi_{n}\ast E(u)$ converges to $E(u)$ in the weak topology for $W^{1,q}(\Omega(r^{\prime}))$. In particular $E(u)$ is in $W^{1,q}(D(r^{\prime}))$. Basic properties of convolutions ensure that $\Phi_{n}\ast E(u)$ converges to $E(u)$ in the $W^{1,q}(D(r))$ norm. Thus we can set $u_{n}=\Phi_{n}\ast E(u)$, as desired.
\end{proof}
\begin{remark}
  Let $u\in W^{1,q}(\Omega(1),\C^{n},\R^{n})$, i.e.,\  $E(u)\in W^{1,q}(D(1),\C^{n},\R^{n})$. Then $(\bd_{s}+i\bd_{t})\Phi_{n}\ast E(u)=F_{n}$ converges to some element $F\in L^{q}$ (in the sense of distributions). We claim that
  \begin{equation*}
    \bd_{s}u+i\bd_{t}u=F\text{ holds weakly}.
  \end{equation*}
  This is a sort of converse to the above proposition. Indeed, if $\varphi$ takes real-values along the boundary, we compute
  \begin{equation*}
    \ip{u,-\bd \varphi}=\lim_{n\to\infty}\ip{\Phi_{n}\ast E(u),-\bd\varphi}=\lim_{n\to\infty}\ip{\db \Phi_{n}\ast E(u),\varphi}=\ip{F,\varphi},
  \end{equation*}
  as desired.
\end{remark}
\begin{remark}\label{remark:pg2}  
  If $u\in W^{1,p}(\Omega(1),\C^{n})$ with $p>2$, then $u$ has well-defined boundary values. Suppose that $u$ takes real values along the boundary. We will show that $E(u)\in W^{1,p}_{\mathrm{loc}}(D(1),\C^{n})$. Let $\varphi$ be a test function taking real values along the boundary. Let $h:\cl{\mathbb{H}}\to [0,1]$ be a function which (a) vanishes on $\R\times [0,1]$, (b) which equals $1$ on $\R\times [2,\infty)$ and (c) which depends only on the $t$ coordinate. Let $F=\db u$ (an $L^{p}$ distribution). We compute:
  \begin{equation*}
    \ip{u,-\bd \varphi}=\lim_{\sigma\to \infty}\ip{h(\sigma t)u,-\bd \varphi}=\lim_{\sigma}[\sigma \ip{\db(h)(\sigma t)u,\varphi}+\ip{h(\sigma t)F,\varphi}].
  \end{equation*}
  Note that $\db(h)(\sigma t)$ is concentrated on a region $R_{\sigma}$ of area at most $\sigma^{-1}$. Since $\db(h)$ is purely imaginary, we can write
  \begin{equation*}
    \ip{\db(h)(\sigma t)u,\varphi}=\ip{\db(h)(\sigma t)\mathrm{Im}(u),\text{Re}(\varphi)}+\ip{\db(h)(\sigma t)\mathrm{Re}(u),\text{Im}(\varphi)}.
  \end{equation*}
  Our discussion of $R_{\sigma}$ and $\db(h)$ implies that
  \begin{equation*}
    \abs{\sigma\ip{\db(h)(\sigma t)u,\varphi}}\le C(\sup_{R_{\sigma}}\abs{\mathrm{Re}(\varphi)}\abs{\mathrm{Im}(u)}+\abs{\mathrm{Re}(u)}\abs{\mathrm{Im}(\varphi)}).
  \end{equation*}
  Because $\mathrm{Im}(u)$ and $\mathrm{Im}(\varphi)$ are both continuous and vanish on the boundary, we can take the limit as $\sigma\to\infty$ and conclude that $\lim_{\sigma}\sigma\ip{\db(h)(\sigma t)u,\varphi}$ vanishes. We are left with
  \begin{equation*}
    \ip{u,-\bd \varphi}=\ip{F,\varphi}
  \end{equation*}
  This says that $\db u=F$ weakly. As a consequence of Proposition \ref{prop:smooth-approximation} we conclude that $E(u)\in W^{1,p}(D(r))$ and $u$ can be approximated in $W^{1,p}(D(r))$ by smooth functions taking real values along the boundary. This completes the proof.
\end{remark}

\subsection{Injectivity estimates for translation invariant operators}\label{sec:injectivity-estimates}
The next result concerns various estimates for operators of the form $$u\mapsto \bd_{s}u+i\bd_{t}u+S(t)u=\bd_{s}u-Au$$ on the infinite strip $\R\times [0,1]$ with $A$ a non-degenerate asymptotic operator. See \cite[Lemma 2.4]{salamon1997}, \cite[Section 4.4]{wendl-sft}, and \cite[Section 3.1.2]{schwarz-diss} for similar results for the infinite cylinder.

\begin{prop}\label{prop:multi-est}
  Let $D(u)=\bd_{s}u-Au$ on the infinite strip $\R\times [0,1]$, where $A$ is a non-degenerate asymptotic operator. Let $\norm{-}$ denote the $L^{2}$ norm over $[0,1]$. 

  There exist constants $c_{1},c_{2},c_{3,p}$ so that, for all $u\in C^{\infty}_{0}(\R\times [0,1],\C^{n},\R^{n})$, we have
  \begin{enumerate}
  \item $\displaystyle\int_{\R}\norm{u}^{2}+\norm{\bd_{s}u}^{2}+\norm{\bd_{t}u}^{2}\d s\le c_{1}\int_{\R}\norm{D(u)}^{2}\d s$,
  \item $\displaystyle\int_{\R}\norm{u}^{p}\d s\le c_{2}^{p}\int_{\R}\norm{D(u)}^{p}\d s$,
  \item $\displaystyle\norm{u}_{W^{1,p}(\R\times [0,1])}\le c_{3,p}\norm{D(u)}_{L^{p}(\R\times [0,1])}$, for $p\ge 2$.
  \end{enumerate}

  The same result holds with $[0,1]$ replaced by $\R/\Z$.
\end{prop}
\begin{remark}
  Before we prove the theorem, we wish to make a few remarks.
  \begin{enumerate}[label=(\arabic*)]
  \item All of these estimates roughly measure the injectivity of $D$.
  \item The results are proved for smooth functions with compact support, although they imply estimates on various Banach space completions of $C^{\infty}_{0}$ by taking smooth approximations. For instance, the reflection plus convolution technique used in the proof of Lemma \ref{lemma:coord-ell-reg} can be used to approximate $u\in W^{1,p}(\R\times [0,1],\C^{n},\R^{n})$ by smooth $u_{n}$ taking real values along the boundary.
  \item Note that (i) is (iii) in the case $p=2$. After we prove the proposition, we will be able to upgrade (iii) to include the case $q<2$. See Theorem \ref{theorem:isomorphism}.
  \item Note that (ii) can be considered as an estimate on a mixed $(2,p)$ norm. 
  \item We will give an elementary proof of (i), which is similar to the one given in \cite{schwarz-diss}. See \cite{wendl-sft} for an alternate proof of (i) which considers the Fourier transformation in the $s$-variable. 
  \item Our proofs of (ii) and (iii) are directly inspired by \cite{salamon1997}. The proof of (ii) will use the spectral properties of $A$ proved in Proposition \ref{prop:spectral-prop}. See \cite{schwarz-diss} for an alternative proof of (iii). 
  \end{enumerate}
\end{remark}
\begin{proof}[of Proposition \ref{prop:multi-est}]
  Suppose that $D(u)=\eta$, i.e.,\ $\bd_{s}u=Au+\eta$. To prove (i), the idea is to consider the quantity $\gamma(s)=\norm{u(s,t)}^{2}=\ip{u,u}$, where $\ip{-,-}$ denotes the real inner product on $L^{2}([0,1],\mathbb{C}^{n})$. We differentiate $\gamma(s)$ twice:
  \begin{equation*}
    \begin{aligned}
      \gamma^{\prime\prime}(s)&=\ip{\bd_{s}u,\bd_{s}u}+\ip{u,\bd_{s}\bd_{s}u}=\norm{\bd_{s}u}^{2}+\ip{u,\bd_{s}(Au+\eta)}\\
      &=\norm{\bd_{s}u}^{2}+\ip{Au,\bd_{s}u}+\bd_{s}\ip{u,\eta}-\ip{\bd_{s}u,\eta}\\
      &=\norm{\bd_{s}u}^{2}+\norm{Au}^{2}+\ip{Au,\eta}+\bd_{s}\ip{u,\eta}-\ip{\bd_{s}u,\eta}.
    \end{aligned}
  \end{equation*}
  Here we have used the fact that $\bd_{s}Au=A\bd_{s}u$ and $\ip{f,Ag}=\ip{Af,g}$. Now we integrate this equality over $\R$. Since $u$ is smooth and compactly supported the integrals of $\gamma^{\prime\prime}(s)$ and $\bd_{s}\ip{u,\eta}$ both vanish. We are left with:
  \begin{equation*}
    \int \norm{\bd_{s}u}^{2}+\norm{Au}^{2}\d s=\int\ip{\bd_{s}u-Au,\eta}\d s.
  \end{equation*}
  Now using Cauchy-Schwarz and $2ab\le a^{2}+b^{2}$ we have
  \begin{equation*}
    \begin{aligned}
      \int\ip{\bd_{s}u-Au,\eta}\d s&\le \int \norm{\bd_{s}u}\norm{\eta}+\norm{Au}\norm{\eta}\d s\\ &\le \frac{1}{2}\int \norm{\bd_{s}u}^{2}+\norm{Au}^{2}\d s+\int \norm{\eta}^{2}\d s.
    \end{aligned}
  \end{equation*}
  Recalling that $D(u)=\eta$, it follows that
  \begin{equation*}
    \int \norm{\bd_{s}u}^{2}+\norm{Au}^{2}\d s\le 2\int\norm{D(u)}^{2}\d s.
  \end{equation*}
  Finally, using the fact that $A:W^{1,2}([0,1],\C^{n},\R^{n})\to L^{2}(\C^{n})$ is an isomorphism, we conclude a constant $c\ge 1$ so that $\norm{u}^{2}+\norm{\bd_{t}u}^{2}\le c\norm{Au}^{2}$, and hence
  \begin{equation*}
    \int \norm{u}^{2}+\norm{\bd_{s}u}^{2}+\norm{\bd_{t}u}^{2}\d s\le 2c\int\norm{D(u)}^{2}\d s,
  \end{equation*}
  as desired. This completes the proof of (i).

  For (ii) we use the spectral properties of $A$. Let $E_{\pm}$ denote the splitting of $L^{2}([0,1],\C^{n})$ into positive and negative eigenspaces of $A$. The operator $\exp(-sA)$ converges on $E_{+}$ for $s\ge 0$ while the operator $\exp(-sA)$ converges on $E_{-}$ for $s\le 0$.

  We can decompose $u=u_{+}+u_{-}$ where $u_{+}(s,-)\in E_{+}$ and $u_{-}(s,-)\in E_{-}$ for all $s$. It is straightforward to show that
  \begin{equation*}
    \begin{aligned}
      \text{for $s\ge 0$:}\hspace{.3cm}\bd_{s}(\exp(-sA)u_{+}(s+s_{0},t))&=\exp(-sA)\eta_{+}(s_{0}+s,t),\\
      \text{for $s\le 0$:}\hspace{.3cm}\bd_{s}(\exp(-sA)u_{-}(s+s_{0},t))&=\exp(-sA)\eta_{-}(s_{0}+s,t),
    \end{aligned}
  \end{equation*}
  where $\eta_{\pm}=\bd_{s}u_{\pm}-Au_{\pm}$. Integrate the first ODE over $[0,\infty)$ and integrate the second ODE over $(-\infty,0]$, concluding that
  \begin{equation}\label{eq:inverse-formula}
    \begin{aligned}
      u_{+}(s_{0},t)&=-\int_{0}^{\infty}\exp(-sA)\eta_{+}(s_{0}+s,t)\d s,\\
      u_{-}(s_{0},t)&=\int_{-\infty}^{0}\exp(-sA)\eta_{-}(s_{0}+s,t)\d s.
    \end{aligned}
  \end{equation}
  Following \cite[Lemma 2.4]{salamon1997}, the idea is now to interpret this as a convolution $u_{\pm}=K_{\pm}\ast \eta_{\pm}$, and then apply Young's convolution inequality to conclude (ii). 

  Here are the details of the argument. First, we show the mixed $(2,p)$ norm satisfies a variational definition:
  \begin{equation*}
    \left[\int_{\R}\norm{u}^{p}\d s\right]^{1/p}=\sup\set{\int_{\R}\ip{u,g}\d s:\int_{\R}\norm{g}^{q}\d s=1,\text{ where $p^{-1}+q^{-1}=1$}}.
  \end{equation*}
  It is easy to show that $\ge$ holds, and to show $\le$ it suffices to prove it when the left hand side is $1$. In this case we can simply take $g=\norm{u}^{p-2}u$ (if $p>2$ this is fine, if $p<2$ then we can take a sequence $g$ approximating $\norm{u}^{p-2}u$).

  Now fix $g$ and compute
  \begin{equation*}
    \int_{\R}\ip{u_{+}(s_{0}),g(s_{0})}\d s_{0}=-\int_{\R}\int_{0}^{\infty}\ip{\exp(-sA)\eta_{+}(s_{0}+s),g(s_{0})}\d s\d s_{0}.
  \end{equation*}
  It is straightforward to check that
  \begin{equation*}
    \ip{\exp(-sA)\eta_{+}(s_{0}+s),g(s_{0})}\le e^{-s\lambda_{\mathrm{min}}^{+}}\norm{\eta_{+}(s_{0}+s)}\norm{g(s_{0})},
  \end{equation*}
  where $\lambda_{\mathrm{min}}^{+}$ is the smallest positive eigenvalue.
  
  Switching the order of integration and using H\"older's inequality yields:
  \begin{equation*}
    \abs{\int_{\R}\int_{0}^{\infty}\ip{\exp(-sA)\eta_{+}(s_{0}+s),g(s_{0})}\d s\d s_{0}}\le \int_{0}^{\infty}e^{-s\lambda_{\mathrm{min}}^{+}}\d s \left[\int_{\R}\norm{\eta_{+}}^{p}\d s_{0}\right]^{1/p}.
  \end{equation*}
  As a consequence the variational definition of the mixed $2,p$ norm implies that
  \begin{equation*}
    \left[\int_{\R} \norm{u_{+}(s)}^{p}\d s\right]^{1/p}\le \frac{1}{\lambda_{\mathrm{min}}^{+}}\left[\int_{\R}\norm{\eta_{+}}^{p}\d s\right]^{1/p}.
  \end{equation*}
  A similar argument shows that
  \begin{equation*}
    \left[\int_{\R} \norm{u_{-}(s)}^{p}\d s\right]^{1/p}\le \frac{-1}{\lambda_{\mathrm{max}}^{-}}\left[\int_{\R}\norm{\eta_{-}}^{p}\d s\right]^{1/p},
  \end{equation*}
  where $\lambda_{\mathrm{max}}^{-}$ is the largest negative eigenvalue.
  
  Using the fact that $\norm{\eta_{\pm}}\le \norm{\eta}$ we conclude that
  \begin{equation*}
    \int_{\R}\norm{u}^{p}\d s\le (\frac{1}{\lambda_{\mathrm{min}}^{+}}-\frac{1}{\lambda_{\mathrm{max}}^{-}})^{p}\int_{\R}\norm{D(u)}^{p}\d s.
  \end{equation*}
  This proves (ii).

  To prove (iii) we again follow \cite[Lemma 2.4]{salamon1997}. Fix $p>2$. Let $\Omega(r)=[-r,r]\times [0,1]$. It is straightforward to apply Theorem \ref{theorem:local-reg} and Sobolev embedding \cite[Theorem B.1.11]{mcduffsalamon} to conclude constants $\kappa_{1},\kappa_{2},\kappa_{3},\kappa_{4}$ so that
  \begin{equation*}
    \begin{aligned}
      \norm{u}_{W^{1,p}(\Omega(1))}&\le \kappa_{1}(\norm{D(u)}_{L^{p}(\Omega(1.5))}+\norm{u}_{L^{p}(\Omega(1.5))}),\\
      \norm{u}_{L^{p}(\Omega(1.5))}&\le \kappa_{2}\norm{u}_{W^{1,2}(\Omega(1.5))},\\
      \norm{u}_{W^{1,2}(\Omega(1.5))}&\le \kappa_{3}(\norm{D(u)}_{L^{2}(\Omega(2))}+\norm{u}_{L^{2}(\Omega(2))}),\\
      \norm{D(u)}_{L^{2}(\Omega(2))}&\le \kappa_{4}\norm{D(u)}_{L^{p}(\Omega(2))}.
    \end{aligned}
  \end{equation*}
  The constant $\kappa_{4}$ can be explicitly computed as $4^{1-1/p}$. Combining these yields
  \begin{equation*}
    \norm{u}_{W^{1,p}(\Omega(1))}\le \kappa_{1}(1+\kappa_{2}\kappa_{3}\kappa_{4})\norm{D(u)}_{L^{p}(\Omega(1.5))}+\kappa_{1}\kappa_{2}\kappa_{3}\norm{u}_{L^{2}(\Omega(2))}.
  \end{equation*}
  Using $(a+b)^{p}\le 2^{p}(a^{p}+b^{p})$ we conclude
  \begin{equation*}
    \norm{u}^{p}_{W^{1,p}(\Omega(1))}\le 2^{p}\kappa^{p}_{1}(1+\kappa_{2}\kappa_{3}\kappa_{4})^{p}\norm{D(u)}^{p}_{L^{p}(\Omega(2))}+(\kappa_{1}\kappa_{2}\kappa_{3})^{p}\norm{u}_{L^{2}(\Omega(2))}^{p}.
  \end{equation*}
  It is straightforward to compute that
  \begin{equation*}
    \norm{u}_{L^{2}(\Omega(2))}^{p}=[\int_{-2}^{2}\norm{u}^{2}\d s]^{p/2}\le 4^{p/2-1}\int_{-2}^{2}\norm{u}^{p}\d s.
  \end{equation*}
  The above holds with $\Omega(r)$ replaced by $2k+\Omega(r)$ for the same constants since $D$ is translation invariant. Hence we conclude that there is a constant $C$ so that
  \begin{equation*}
    \norm{u}^{p}_{W^{1,p}(2k+\Omega(1))}\le C(\norm{D(u)}^{p}_{L^{p}(2k+\Omega(2))}+\int_{2k-2}^{2k+2}\norm{u}^{p}\d s).
  \end{equation*}
  Summing over all $k\in \Z$ yields
  \begin{equation*}
    \norm{u}^{p}_{W^{1,p}(\R\times [0,1])}\le 2C(\norm{D(u)}^{p}_{L^{p}(\R\times [0,1])}+\int_{\R}\norm{u}^{p}\d s).
  \end{equation*}
  The factor of $2$ is because the domains $2k+\Omega(2)$ cover $\R\times [0,1]$ ``twice over.'' Finally, using part (ii), we conclude that
  \begin{equation*}
    \norm{u}^{p}_{W^{1,p}(\R\times [0,1])}\le 2C(1+c_{2}^{p})\norm{D(u)}^{p}_{L^{p}(\R\times [0,1])}.
  \end{equation*}
  Setting $c_{3,p}=(2C(1+c_{2}^{p}))^{1/p}$ completes the proof.
\end{proof}

The formula \eqref{eq:inverse-formula} can be used to prove the following regularity result. We still assume that $D=\bd_{s}-A$ is a translation invariant operator on the infinite strip or cylinder. As in the proof of Proposition \ref{prop:multi-est} we assume that $\lambda^{-}_{\mathrm{max}}<0<\lambda^{+}_{\mathrm{min}}$ are the maximal negative and minimal positive eigenvalues of $A$.
\begin{lemma}\label{lemma:decay-regularity}
  Let $q>1$. If $u\in L^{q}$ and $D(u)=\eta$ for smooth $\eta$ with compact support, then $u$ is smooth. Moreover, for all $k,\ell\in \mathbb{N}$ there are constants $C^{+}_{k,\ell}$ and $C^{-}_{k,\ell}$ depending on $\eta$ so that
  \begin{equation}\label{eq:decay-expo}
    \begin{aligned}
      \abs{\bd^{k}\bd^{\ell}u(s,t)}&\le C_{k,\ell}^{+}e^{\lambda_{\max}^{-} s}\text{ as $s\to+\infty$},\\
      \abs{\bd^{k}\bd^{\ell}u(s,t)}&\le C_{k,\ell}^{-}e^{\lambda_{\min}^{+} s}\text{ as $s\to-\infty$}.
    \end{aligned}
  \end{equation}
  Both estimates are of exponential decay type. In particular, $u\in W^{k,q}$ for all $k$ and all $q$.
\end{lemma}
\begin{proof}
  Let $\Omega(r)=[-r,r]\times [0,1]$. By the local elliptic regularity result (Lemma \ref{lemma:coord-ell-reg}) we know that $u$ is smooth, and that it satisfies elliptic estimates of the form
  \begin{equation*}
    \norm{u}_{W^{k+1,q}(s+\Omega(1))}\le c(\norm{\eta}_{W^{k,q}(s+\Omega(2))}+\norm{u}_{L^{q}(s+\Omega(2))}).
  \end{equation*}
  We can take $c$ to be independent of $s$ since $D$ is translation invariant. In particular, it is clear that the $L^{2}([0,1])$ size $\norm{u(s,-)}$ decays as $s\to\infty$.

  Decompose $u=u_{+}+u_{-}$. It is straightforward to show that $u_{\pm}$ are still elements of $C^{\infty}(\R,W^{1,2}([0,1],\C^{n},\R^{n}))$. As in the proof of Proposition \ref{prop:multi-est}, we think of the equation $\bd_{s}u_{+}-Au_{+}=\eta_{+}$ as an ordinary differential equation which we can explicitly solve:
  \begin{equation*}
    \begin{aligned}
      &\bd_{s}(\exp(-sA)u_{+}(s_{0}+s))=\exp(-sA)\eta(s_{0}+s)\\ \implies& \exp(-NA)u_{+}(s_{0}+N)-u_{+}(s_{0})=\int_{0}^{N}\exp(-sA)\eta(s_{0}+s)\,\d s.
    \end{aligned}
  \end{equation*}
  Taking the limit as $N\to\infty$ and using the fact that $\lim_{N\to\infty}\norm{u_{+}(s_{0}+N)}=0$ we conclude that
  \begin{equation*}
    u_{+}(s_{0},t)=-\int_{0}^{\infty}\exp(-sA)\eta_{+}(s_{0}+s,t)\,\d s.
  \end{equation*}
  A similar argument shows that the other equation in \eqref{eq:inverse-formula} also holds, and hence we have:
  \begin{equation*}
    u(s_{0},t)=\int_{-\infty}^{0}\exp(-sA)\eta_{-}(s_{0}+s,t)\,\d s-\int_{0}^{\infty}\exp(-sA)\eta_{+}(s_{0}+s,t)\,\d s.
  \end{equation*}
  Now suppose that $\eta$ is supported in $[-R,R]$. Then for $s_{0}<-R$, the first integral always vanishes, and the second integrand is supported on the region where $s>-s_{0}-R$ and so we have:
  \begin{equation*}
    \norm{u(s_{0},-)}\le e^{(s_{0}+R)\lambda^{+}_{\mathrm{min}}}\int_{-\infty}^{\infty}\norm{\eta_{+}}\d s=C(\eta_{-},R)e^{s_{0}\lambda^{+}_{\mathrm{min}}}\hspace{.5cm}\text{ (as $s_{0}\to-\infty$)}.    
  \end{equation*}
  A similar deduction proves that
  \begin{equation*}
    \norm{u(s_{0},-)}\le C(\eta_{+},R)e^{s_{0}\lambda^{-}_{\mathrm{max}}}\hspace{.5cm}\text{ (as $s_{0}\to+\infty$)}.
  \end{equation*}
  Now, by simply integrating the norm $\norm{u(s,-)}$ over $s\in[s_{0}-2,s_{0}+2]$ we conclude that
  \begin{equation}\label{eq:like-me}
    \begin{aligned}
      \norm{u}_{L^{2}(s+\Omega(2))}&\le C_{2}e^{s\lambda^{-}_{\mathrm{max}}}\text{ as $s\to+\infty$}\\
      \norm{u}_{L^{2}(s+\Omega(2))}&\le C_{2}e^{s\lambda^{+}_{\mathrm{max}}}\text{ as $s\to-\infty$}.
    \end{aligned}
  \end{equation}
  Using the elliptic estimates for $q=2$, we conclude that the $W^{k,2}$ size of $u$ on $s+\Omega(1)$ also decays exponentially like \eqref{eq:like-me}. Since the $C^{\ell}$ size is controlled by the $W^{k+2,2}$ size, we ultimately conclude the desired result \eqref{eq:decay-expo}.
\end{proof}

We can upgrade the injectivity estimates to the following important result:
\begin{theorem}\label{theorem:isomorphism}
  Let $D=\bd_{s}u-Au$ with $A$ a non-degenerate asymptotic operator. Let $q>1$. The induced map $D:W^{1,q}(\R\times [0,1],\C^{n},\R^{n})\to L^{q}(\R\times [0,1],\C^{n})$ is an isomorphism.
\end{theorem}
\begin{proof}
  First we prove the case when $q\ge 2$. Part (iii) of Proposition \ref{prop:multi-est} implies that $D$ is injective and has closed image. Thus it suffices to prove that the image of $D$ is dense. If $\eta$ is a smooth function with compact support, then \eqref{eq:inverse-formula} gives an explicit formula for some $u$ satisfying $D(u)=\eta$. As in the proof of Lemma \ref{lemma:decay-regularity}, $u$ is smooth and the formula \eqref{eq:inverse-formula} implies that $u$ and its derivatives decay exponentially as $s\to\pm\infty$, hence $u\in W^{1,p}$.

  Now we prove the case when $q<2$. We follow the argument outlined in \cite[Exercise 2.5]{salamon1997}. The idea is to prove an injectivity estimate for $D:L^{q}\to W^{-1,q}$, and then upgrade this to a $D:W^{1,q}\to L^{q}$ injectivity estimate.

  By definition, we set $W^{-1,q}=(W^{1,p})^{*}$ where $p$ is H\"older dual to $q$, and
  \begin{equation*}
    \norm{u}_{-1,q}=\sup_{\norm{\varphi}_{1,p}=1}\ip{u,\varphi}.
  \end{equation*}

  Let $D^{*}=-\bd_{s}-A$. By the above results (applied to $-D^{*}$) we conclude $D^{*}$ is an isomorphism $W^{1,p}\to L^{p}$. Thus
  \begin{equation*}
    c^{-1}\norm{u}_{L^{q}}\le \sup_{\norm{\varphi}_{1,p}=1}\ip{u,D^{*}(\varphi)}\le c\norm{u}_{L^{q}}.
  \end{equation*}
  Observe that $D^{*}=-\bd_{s}-A$ is the formal adjoint to $D$, and hence (using distributional definitions) we have:
  \begin{equation*}
    \norm{u}_{L^{q}}\le c\norm{D(u)}_{-1,q}.
  \end{equation*}
  In particular, if $v\in W^{1,q}$, then we can apply the above to $u=\bd_{s}v$ and conclude that
  \begin{equation*}
    \norm{\bd_{s}u}_{L^{q}}\le c\norm{D(\bd_{s}u)}_{-1,q}.
  \end{equation*}
  Now it is clear that, in the sense of distributions, we have $D(\bd_{s}u)=\bd_{s}D(u)$. We claim that
  \begin{equation*}
    \norm{\bd_{s}D(u)}_{-1,q}\le c_{2}\norm{D(u)}_{L^{q}}.
  \end{equation*}
  This is easy to see using the above variational definition of the $W^{-1,q}$ norm. Thus we conclude that
  \begin{equation*}
    \norm{\bd_{s}u}_{L^{q}}\le c_{3}\norm{D(u)}_{L^{q}}.
  \end{equation*}
  It is clear that $W^{-1,q}$ norm is less than the $L^{q}$ norm, hence $\norm{u}_{L^{q}}\le c\norm{D(u)}_{L^{q}}$. It then follows easily that $\norm{\bd_{t}u}_{L^{q}}\le c_{4}\norm{D(u)}_{L^{q}}$, and so we conclude the desired injectivity estimate
  \begin{equation*}
    \norm{u}_{W^{1,q}}\le c\norm{D(u)}_{L^{q}}.
  \end{equation*}
  It follows easily that $D(u)$ has closed range and hence it suffices to prove that the image of $D$ is dense. The arguments given in \eqref{eq:inverse-formula} and Lemma \ref{lemma:decay-regularity} show that we can (explicitly) solve for compactly supported smooth functions and the solutions are certainly of class $W^{1,q}$. Thus $D$ is surjective. This completes the proof that $D$ is an isomorphism.
\end{proof}

\subsection{Proof of the Fredholm property}
\label{sec:proof-of-fredholm}
The main result of this section is the following:
\begin{prop}
  Let $p>1$ and let $D$ be an asymptotically non-degenerate Cauchy-Riemann operator for the data $(\Sigma,\Gamma_{\pm},E,F,C,[\tau])$. Then the induced maps
  \begin{equation*}
    D:W^{1,p}(E,F)\to L^{p}(\Lambda^{0,1}\otimes E)\text{ and }D^{*}:W^{1,p}(E,F)\to L^{p}(\Lambda^{0,1}\otimes E)
  \end{equation*}
  are Fredholm.
\end{prop}
Similar arguments can be found in \cite[Section 2.3]{salamon1997}, \cite[Proposition 3.1.30]{schwarz-diss}, and \cite[Section 4.5]{wendl-sft}. 
\begin{proof}
  Let $\varphi_{\rho}$ be a cutoff function supported in the ends which equals $0$ on $\Sigma(\rho-1)$ and equals $1$ on $C(\rho)$. We can choose $\varphi_{\rho}$ so that its derivatives are bounded as $\rho\to\infty$. 

  We observe that
  \begin{equation*}
    D(\varphi_{\rho}u)=\bd_{s}(\varphi_{\rho}u)-Au(\varphi_{\rho}u)+\Delta(s)\varphi_{\rho}u,
  \end{equation*}
  where $\Delta(s)$ is a lower order term which converges to $0$ as $s\to\pm \infty$. We know that $\bd-A:W^{1,p}\to L^{p}$ is an isomorphism and so $\norm{\varphi_{\rho}u}_{W^{1,p}}\le C\norm{(\bd_{s}-A)(\varphi_{\rho}u)}$ for some $C$. We estimate
  \begin{equation*}
    \begin{aligned}
      \norm{\varphi_{\rho}u}_{W^{1,p}}&\le C(\norm{D(\varphi_{\rho}u)}_{L^{p}}+\norm{\Delta(s)\varphi_{\rho}u}_{L^{p}})\\
      &\le C(\norm{D(u)}_{L^{p}}+\norm{\Delta(s)\varphi_{\rho}u}_{L^{p}}+\norm{\db \varphi_{\rho}\cdot u}_{L^{p}})\\
      \implies \norm{\varphi_{\rho}u}_{L^{p}}&\le C^{\prime}(\norm{D(u)}_{L^{p}}+\norm{u}_{L^{p}(\Sigma(\rho))}),
    \end{aligned}
  \end{equation*}
  where we pick $\rho$ large enough so that $C\abs{\Delta(s)}<0.5$ on the support of $\varphi_{\rho}$. We also use that $\db\varphi_{\rho}$ is supported in $\Sigma(\rho)$.

  Next we combine the local elliptic estimates from \ref{lemma:coord-ell-reg} (to finitely many disks covering $\Sigma(\rho)$) and conclude some constant $C(\rho)$ so that
  \begin{equation*}
    \norm{(1-\varphi_{\rho})u}_{W^{1,p}}\le C(\rho)(\norm{Du}_{L^{p}}+\norm{u}_{L^{p}(\Sigma(\rho+1))}).
  \end{equation*}  
  Combining our two estimates (and updating the constant) yields
  \begin{equation}\label{eq:semi-fredholm}
    \norm{u}_{W^{1,p}}\le C(\norm{Du}_{L^{p}}+\norm{u}_{L^{p}(\Sigma(\rho+1))}).
  \end{equation}
  Crucially, $\rho$ does not depend on $u$. Since $W^{1,p}\to L^{p}(\Sigma(\rho+1))$ (inclusion followed by restriction) is a compact operator, we conclude from \eqref{eq:semi-fredholm} that $D$ is semi-Fredholm; i.e.,\ has closed image and finite dimensional kernel. See \cite[Appendix A]{mcduffsalamon} for the argument. The same argument shows that $D^{*}$ is semi-Fredholm.

  Now suppose that $D$ were not surjective. Since the image of $D$ is closed, we can apply the Hahn-Banach theorem to find $w\in L^{q}(\Lambda^{0,1}\otimes E)$ so that $\ip{D(u),w}=0$ for all $u\in W^{1,p}(E,F)$ and $w\ne 0$. But then $w$ is smooth and takes boundary values in $F^{*}$ by the local regularity results. Moreover, in the ends we have $(\bd_{s}+A)w=\Delta^{*} w$ which implies that $w\in W^{1,q}(\Lambda^{0,1}\otimes E,F^{*})$ (using the injectivity estimates for $\bd_{s}+A$). We conclude that $w\in \ker D^{*}$. Since $D^{*}$ is semi-Fredholm, its kernel is finite dimensional. This implies that $\mathrm{coker}\,D$ is finite dimensional, and this completes the proof that $D$ is Fredholm. The same argument works for $D^{*}$.
\end{proof}

\section{Conley-Zehnder indices and kernel gluing}\label{sec:cz-index}
In this section our goal is to prove that the index behaves additively under a gluing operation. See \cite[Section 3.2]{schwarz-diss} for a similar argument.

Throughout this section we fix an asymptotic trivialization $\tau$ (i.e.,\ fix $\tau_{z}$ for each $z\in \Gamma$). Suppose that $D$ is an asymptotic operator on $(\Sigma,\Gamma_{\pm},E,F)$ whose restriction to the cylindrical ends $C_{z}$ equals $D=\bd_{s}-A_{z}$ with respect to $\tau_{z}$ and where each $A_{z}$ is a non-degenerate asymptotic operator. We have shown that $D$ is Fredholm. Moreover, it is clear that if $D^{\prime}$ has the same asymptotic operators $A_{z}$ (in the same trivialization), then we can homotope $D$ to $D^{\prime}$ while remaining in the space of Fredholm operators. Then the index of $D$ will equal the index of $D^{\prime}$. Therefore, the index depends only on the choice of non-degenerate asymptotic operators $z\mapsto A_{z}$ (and $(\Sigma,\Gamma,E,F)$ of course).

Introduce the reference operator $D^{\mathrm{al}}$ whose restrictions to the cylindrical ends equals $\bd_{s}+i\bd_{t}+C$ with respect to the same trivialization $\tau$. Here $C$ is the matrix of complex conjugation, i.e., in each end we have $D^{\mathrm{al}}(u)=\bd_{s}u+i\bd_{t}u+\cl{u}$. The ``$\mathrm{al}$'' stands for ``anti-linear.'' The associated asymptotic operator is $A^{\mathrm{al}}=-i\bd_{t}-C$. In other words, $D^{\mathrm{al}}$ has all of its asymptotics equal to $A^{\mathrm{al}}$.

In this section we will prove the following formula for index difference
\begin{equation*}
  \mathrm{ind}(D)-\mathrm{ind}(D^{\mathrm{al}})=\sum_{z\in \Gamma_{+}}\mu_{\mathrm{CZ}}(A_{z})-\sum_{z\in \Gamma_{+}}\mu_{\mathrm{CZ}}(A_{z}),
\end{equation*}
where $\mu_{\mathrm{CZ}}(A_{z})$ is the \emph{Conley-Zehnder} index of $A_{z}$, defined in Section \ref{sec:cz-indexs} below. This formula determines how the index depends on changing asymptotic operators (i.e.\ we can compute $\mathrm{ind}(D_{1})-\mathrm{ind}(D_{2})$ for any pair $D_{1},D_{2}$). In Section \ref{sec:final-frontier} we will prove that $\mathrm{ind}(D^{\mathrm{al}})=n\mathrm{X}+\mu_{\mathrm{Mas}}^{\tau}(E,F)$, which will complete the proof of the index formula.

\subsection{Conley-Zehnder indices}
\label{sec:cz-indexs}

First we need to show that $D^{\mathrm{al}}$ is actually Fredholm. This follows from:
\begin{lemma}\label{lemma:reference-non-degen}
  For $\sigma>0$, the reference operator $A^{\mathrm{al},\sigma}=-i\bd_{t}-\sigma C$ is non-degenerate.
\end{lemma}
\begin{proof}
  We prove the strip case, leaving the (very similar) $\R/\Z$ case to the reader. Suppose $u:[0,1]\to \C^{n}$ takes real values when $t=0$ and $A^{\mathrm{al},\sigma}(u)=0$. A straightforward computation shows that  
  \begin{equation*}
    \bd_{t}u_{j}=i\sigma\cl{u_{j}}\iff \bd_{t}(x_{j}+iy_{j})=\sigma(y_{j}+ix_{j})\implies u_{j}=x_{j}(0)(\cosh(\sigma t)+i\sinh(\sigma t)).
  \end{equation*}
  In particular, since $\sinh(\sigma t)>0$ for $t>0$, we cannot also have $u_{j}(1)\in \R^{n}$. This proves that $A^{\mathrm{al},\sigma}$ is non-degenerate. 
\end{proof}

Now fix a non-degenerate asymptotic operator $A$. We will define a special Cauchy-Riemann operator on the infinite strip/cylinder which will interpolate between $\bd_{s}-A^{\mathrm{al}}$ and $\bd_{s}-A$. Let $s\mapsto \beta(s)$ be a $[0,1]$-valued bump function which equals $0$ on $(-\infty,0]$ and $1$ on $[1,\infty)$, and define:
\begin{equation*}
  D^{\mathrm{CZ}}_{A}:=\bd_{s}-(1-\beta(s))A^{\mathrm{al}}-\beta(s)A.
\end{equation*}
As a corollary to Lemma \ref{lemma:reference-non-degen}, the operator $D_{A}^{\mathrm{CZ}}$ is Fredholm. We define the \emph{Conley-Zehnder index} of $A$ as the Fredholm index of $D_{A}^{\mathrm{CZ}}$:
\begin{equation*}
  \mu_{\mathrm{CZ}}(A):=\mathrm{ind}(D_{A}^{\mathrm{CZ}}).
\end{equation*}
It is clear that $\mu_{\mathrm{CZ}}(A)$ is independent of the choice of $\beta$ used to define $D^{\mathrm{CZ}}_{A}$, since any deformation of bump functions will keep $D_{A}^{\mathrm{CZ}}$ in the space of Fredholm operators.

\begin{remark}
  See \cite[page 595]{floer-ham} for an argument which explains why $\mathrm{ind}(D_{A}^{\mathrm{CZ}})$ is the \emph{spectral flow} of the path of self-adjoint operators $A(s)=(1-\beta(s))A^{\mathrm{al}}+\beta(s)A$.
\end{remark}

Note that, since $\bd_{s}-A^{\mathrm{al}}$ is an isomorphism $W^{1,p}\to L^{p}$ (Theorem \ref{theorem:isomorphism}) we conclude that $\mu_{\mathrm{CZ}}(A^{\mathrm{al}})=0$.

The main result of this section is:
\begin{prop}\label{prop:main-gluing}
  Let $D,D^{\mathrm{al}}$ be as above, i.e., for a fixed choice of trivialization $\tau$ and for each $z\in \Gamma$ the restrictions of $D,D^{\mathrm{al}}$ are
  \begin{equation*}
    \begin{aligned}
      D&=\bd_{s}-A_{z},\hspace{1cm}&D^{\mathrm{al}}&=\bd_{s}-A^{\mathrm{al}}.
    \end{aligned}
  \end{equation*}
  Note that the operators $D^{\mathrm{al}}$ and $A_{z}$ depend on the choice of triviliazation. We have
  \begin{equation*}
    \mathrm{ind}(D)=\mathrm{ind}(D^{\mathrm{al}})+\sum_{z\in \Gamma_{+}}\mu_{\mathrm{CZ}}(A_{z})-\sum_{z\in \Gamma_{-}}\mu_{\mathrm{CZ}}(A_{z}).
  \end{equation*}
\end{prop}
\begin{proof}
  The proposition follows from a kernel gluing lemma for stabilized operators (Lemma \ref{lemma:kernel-gluing}) proved below, as explained in Remark \ref{remark:main-gluing}.
\end{proof}

The kernel gluing argument we will use is similar to the one used in \cite[Section 3.2]{schwarz-diss}. See also \cite[Proposition 9]{fh-coherent}. The rough idea is to deform $D$ by a parameter $\rho$ so that it equals a ``glued'' operator $D^{\rho}$ obtained from $D^{\mathrm{al}}$ by gluing on the asymptotic operator $D^{\mathrm{CZ}}_{A_{z}}$ for each $z\in \Gamma$, as suggested by in Figure \ref{fig:gluing-step}.

Note that at negative ends we actually need to glue $D^{\mathrm{CZ}}_{A_{z}}$ on ``backwards.'' For this reason, we define:
\begin{equation*}
  D^{\mathrm{ZC}}_{A}:=\bd_{s}-(1-\beta(s))A-\beta(s)A^{\mathrm{al}},
\end{equation*}
which interpolates from $\bd_{s}-A$ on the negative end to $\bd_{s}-A^{\mathrm{al}}$ at the positive end.

Our kernel gluing argument will imply two things:
\begin{enumerate}
\item $\mathrm{ind}(D)=\mathrm{ind}(D^{\mathrm{al}})+\sum_{z\in \Gamma_{+}}\mathrm{ind}(D^{\mathrm{CZ}}_{A})+\sum_{z\in \Gamma_{-}}\mathrm{ind}(D^{\mathrm{ZC}}_{A})$,
\item $\mathrm{ind}(D^{\mathrm{CZ}}_{A})+\mathrm{ind}(D^{\mathrm{ZC}}_{A})=0\implies \mathrm{ind}(D^{\mathrm{ZC}}_{A})=-\mu_{\mathrm{CZ}}(A)$. 
\end{enumerate}
These results together imply Proposition \ref{prop:main-gluing}.

Before we perform the gluing argument we will explain how to \emph{stabilize} the relevant operators in order to make them surjective. This is the topic of the next subsection.

\begin{figure}[H]
  \centering
  \begin{tikzpicture}[yscale=0.75]
    \begin{scope}[xscale=1.55]
      \draw (0,0)to[out=90,in=-90](1.75,3) (2.5,3)to[out=-90,in=90](4,0);
      \draw (0.75,0) arc (180:0:1.25)node[pos=0.5,above,shift={(0,0.2)}]{$D^{\mathrm{al}}$};
      \path[shift={(0,-2.2)}] (0.36,1) node{$D_{A_{z_{1}}}^{\mathrm{ZC}}$};
      \path[shift={(3.25,-2.2)}] (0.36,1) node{$D_{A_{z_{2}}}^{\mathrm{ZC}}$};
      \path[shift={(1.75,3.2)}] (0.36,1) node{$D_{A_{z_{0}}}^{\mathrm{CZ}}$};

      \draw[shift={(0,-2.2)}] (0,0)node[left]{$\bd_{s}-A_{z_{1}}$}--+(0,2)node[left]{$\bd_{s}-A^{\mathrm{al}}$} (0.75,0)--+(0,2);
      \draw[shift={(3.25,-2.2)}] (0,0)--+(0,2) (0.75,0)node[right]{$\bd_{s}-A_{z_{2}}$}--+(0,2)node[right]{$\bd_{s}-A^{\mathrm{al}}$};
      \draw[shift={(1.75,3.2)}] (0,0)node[left]{$\bd_{s}-A^{\mathrm{al}}$}--+(0,2)node[left]{$\bd_{s}-A_{z_{0}}$} (0.75,0)--+(0,2);
    \end{scope}
    \draw[->] (7,1.5)--node[above]{glue ($\text{parameter}=\rho$)}(11,1.5) node[right,draw,outer sep=15pt] {$D^{\rho}$};
  \end{tikzpicture}
  \caption{Gluing together operators $D^{\mathrm{al}}$, $D^{\mathrm{CZ}}_{A}$, and backwards versions $D^{\mathrm{ZC}}_{A}$
    to form $D^{\rho}$, which can be deformed back to $D$ through Fredholm operators. For large gluing parameter $\rho$, we will be able to relate the kernel of $D^{\rho}$ to the kernels of $D^{\mathrm{CZ}}_{A_{z}}$, $D^{\mathrm{ZC}}_{A_{z}}$, and $D^{\mathrm{al}}$.}
  \label{fig:gluing-step}
\end{figure}
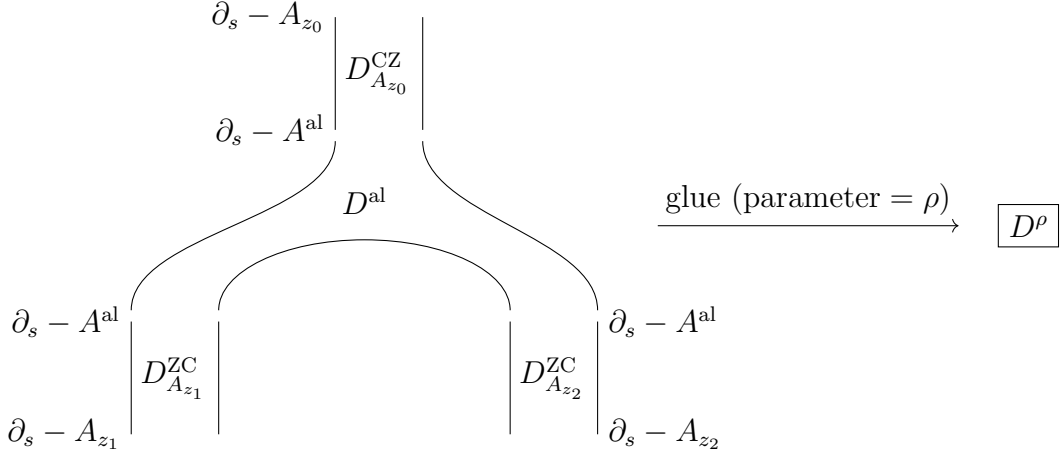

\subsection{Stabilizing Cauchy-Riemann operators}
\label{sec:stabilized-operators}
Let $D$ be a Cauchy-Riemann operator on $(E,F,\Sigma,\Gamma,C,[t])$ as usual.

As we have seen in Section \ref{sec:proof-of-fredholm}, the operator $D:W^{1,p}(E,F)\to L^{p}(\Lambda^{0,1}\otimes E)$ has a finite dimensional cokernel which can be identified with $\ker D^{*}\subset W^{1,p}(\Lambda^{0,1}\otimes E,F^{*})$.

Pick a basis $c:\R^{d}\to \ker(D^{*})$, considered as a map $c:\R^{d}\to L^{p}(\Lambda^{0,1}\otimes E)$.
\begin{equation}\label{eq:stabilized}
  \begin{dmatrix}
    {D}&{c}
  \end{dmatrix}:W^{1,p}(E,F)\oplus \R^{d}\to L^{p}(\Lambda^{0,1}\otimes E).
\end{equation}
For our choice of $c$, this operator is surjective and its kernel is $\ker(D)\oplus 0$. Since $\ker(D)$ is finite dimensional, the above operator has a right inverse. Since having a right inverse is open in the norm topology, we can smoothly ``cut-off'' the cokernel elements $c_{1},\cdots,c_{d}$ so that they vanish outside of $\Sigma(\rho_{0})$ for $\rho_{0}$ sufficiently large (i.e.\ they vanish on the ends $C(\rho_{0})$). 

This leads us to the following definition: a \emph{stabilized} operator for $D$ is any surjective operator $D_{\mathrm{st}}$ of the form \eqref{eq:stabilized} where $d=\dim \mathrm{coker}(D)$ and the cokernel elements $c_{1},\cdots,c_{d}$ are smooth and supported in $\Sigma(\rho_{0})$ for some $\rho_{0}$. The preceding discussion shows that stabilized operators always exist.

By computing the Fredholm index of \eqref{eq:stabilized} when $c=0$, we easily see that
\begin{equation*}
  \mathrm{ind}(D^{\mathrm{st}})=\mathrm{ind}(D)+d=\mathrm{ind}(D)+\dim \mathrm{coker}(D)=\dim \ker(D).
\end{equation*}
Since $D^{\mathrm{st}}$ is surjective, $\dim\ker(D^{\mathrm{st}})=\dim \ker(D)$, and hence
\begin{equation}\label{eq:kernal-d-st}
  \ker D^{\mathrm{st}}=\ker D\oplus 0.
\end{equation}

\subsection{The kernel gluing argument}
\label{sec:kernel-gluing}
Let $D$ be a Cauchy-Riemann operator as above. Fix a single positive puncture $z$ with asymptotic trivialization $\tau$, and suppose that $D$ is asymptotic to $\bd_{s}-A$ in the end $C_{z}$. 

By perturbing $D$ through the space of Fredholm operators, we may suppose that on $C_{z}$ we have
\begin{equation*}
  D=\bd_{s}-(1-\beta(s))A^{\mathrm{al}}-\beta(s)A.
\end{equation*}
Here $\beta$ is the bump function from before (i.e.,\ $0$ on $(-\infty,0]$ and $1$ on $[1,\infty)$). This local model is nice because it is the beginning of a family of Fredholm operators, namely
\begin{equation*}
  D^{\rho}=\bd_{s}-(1-\beta(s-3\rho))A^{\mathrm{al}}-\beta(s-3\rho)A.
\end{equation*}
We suppose that $D^{\rho}$ is fixed away from $C_{z}$. Consequently, the index of $D^{\rho}$ is constant since it is always Fredholm (its asymptotics are fixed).

Introduce the operator $D^{-}=\lim_{\rho\to\infty}D^{\rho}$ (pointwise limit). In other words, $D^{-}$ agrees with $D$ on the complement of $C_{z}$ and equals $\bd_{s}-A^{\mathrm{al}}$ on $C_{z}$.

Observe that the restriction of $D^{\rho}$ to $C_{z}$ is a translated copy of 
\begin{equation*}
  D_{+}:=D^{\mathrm{CZ}}_{A}=\bd_{s}-(1-\beta(s))A^{\mathrm{al}}-\beta(s)A.
\end{equation*}
We can therefore think of $D^{\rho}$ as obtained by \emph{gluing} $D^{+}$ to the positive end of $D^{-}$. See Figure \ref{fig:agree-on-overlap}.

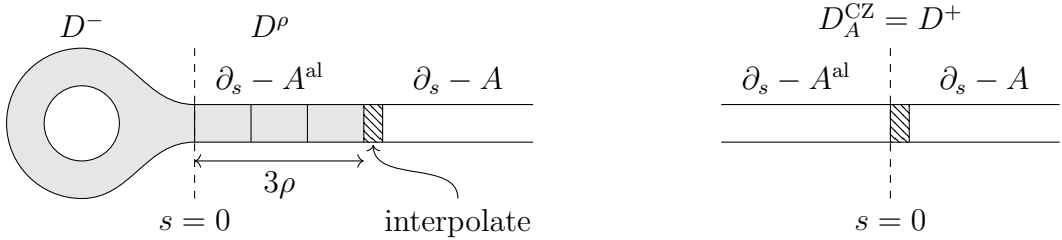
\begin{figure}[H]
  \centering
  \begin{tikzpicture}[scale=0.5]
    \fill[black!10!white,even odd rule] (0,0) circle (1) (7.5,0.5)--(3,0.5)to[out=180,in=0](0,2) arc (90:270:2) to[out=0,in=180] (3,-0.5)--(7.5,-0.5);
    \fill[pattern=north west lines] (7.5,-0.5) rectangle +(0.5,1);
    \node at ({3+2},0.5)[above]{$\bd_{s}-A^{\mathrm{al}}$};
    \node at ({10},0.5)[above]{$\bd_{s}-A$};
    \path (10,-1)--(7.75,-0.5)node[outer sep=0pt,inner sep=1pt](A){};
    \draw[->] (10,-2)node[below]{interpolate}to[out=90,in=-90](A);
    \draw[<->] (3,-1)--node[below]{$3\rho$}(7.5,-1);
    \draw[dashed] (3,2)--(3,-2)node[below]{$s=0$};
    
    \foreach \x in {0,1.5,3,4.5,5} {
      \draw ({3+\x},-0.5)--+(0,1);
    }
    \draw (0,0) circle (1) (12,0.5)--(3,0.5)to[out=180,in=0](0,2) arc (90:270:2) to[out=0,in=180] (3,-0.5)--(12,-0.5);
    \node at (0,2)[above]{$D^{-}$};
    \node at (5,2)[above]{$D^{\rho}$};
    \begin{scope}[shift={(14,0)}]
      \fill[pattern=north west lines] (7.5,-0.5) rectangle +(0.5,1);
      \node at ({3+2},0.5)[above]{$\bd_{s}-A^{\mathrm{al}}$};
      \node at ({10},0.5)[above]{$\bd_{s}-A$};
      \path (10,-1)--(7.75,-0.5)node[outer sep=0pt,inner sep=1pt](A){};
      \draw (3,0.5)--(12,.5);
      \draw (3,-0.5)--(12,-.5);
      \draw[dashed] (7.5,2)node[above]{$D^{\mathrm{CZ}}_{A}=D^{+}$}--(7.5,-2)node[below]{$s=0$};
      \foreach \x in {4.5,5} {
      \draw ({3+\x},-0.5)--+(0,1);
    }
    \end{scope}
  \end{tikzpicture}
  \caption{The relationship between $D_{-}$, $D^{\rho}$ and $D^{+}$. We can think of $D_{-}$ as the pointwise limit of $D^{\rho}$. However, $D^{\rho}$ is always a translated (and truncated) version of $D^{+}$ on $C_{z}$.}
  \label{fig:agree-on-overlap}
\end{figure}
   

To perform the actual gluing argument, we will need to stabilize the operators. Let $c=(c_{1},\cdots,c_{d})$ be cokernel elements for $D^{-}$ and let $\gamma=(\gamma_{1},\cdots,\gamma_{\delta})$ be cokernel elements for $D^{+}$. We suppose that the $c_{j}$ are supported in $\Sigma(\rho_{0})$ and similarly the $\gamma_{k}$ are supported where $\abs{s}<\rho_{0}$. These choices define stabilized operators:
\begin{equation*}
  \begin{aligned}
    D^{-}_{\mathrm{st}}:W^{1,p}(E,F)\oplus \R^{d}\to L^{p}(\Lambda^{1,0}\otimes E)\hspace{1cm}(\xi_{1},a)&\mapsto D^{-}(\xi_{1})+\sum a_{j}c_{j},\\
    D^{+}_{\mathrm{st}}:W^{1,p}(\C^{n},\R^{n})\oplus \R^{\delta}\to L^{p}(\C^{n})\hspace{1cm}(\xi_{2},b)&\mapsto D^{+}(\xi_{2})+\sum b_{k}\gamma_{k}.
  \end{aligned}
\end{equation*}
Then, for $\rho>\rho_{0}$, we define:
\begin{equation}
  \begin{aligned}
    D^{\rho}_{\mathrm{st}}&:W^{1,p}(E,F)\oplus \R^{d}\oplus \R^{\delta}\to L^{p}(\Lambda^{0,1}\otimes E)\\&\hspace{3cm}\text{ by }(\xi,a,b)\mapsto D^{\rho}(\xi)+\sum a_{j}c_{j}+\sum b_{k}\gamma_{k}(s-3\rho).
  \end{aligned}
\end{equation}
Notice that $D^{\rho}_{\mathrm{st}}$ is well-defined since $\gamma_{k}(s-3\rho)$ is supported in $C_{z}(2\rho)$ for $\rho>\rho_{0}$.

The following lemma establishes a relationship between $D^{-}_{\mathrm{st}},D^{\rho}_{\mathrm{st}}$ and $D^{+}_{\mathrm{st}}$.
\begin{lemma}[Kernel gluing lemma]\label{lemma:kernel-gluing}
  For $\rho$ sufficiently large,
  \begin{enumerate}
  \item $D^{\rho}_{\mathrm{st}}$ is surjective,
  \item $\dim \ker D^{\rho}_{\mathrm{st}}=\dim \ker D^{-}_{\mathrm{st}}+\dim \ker D^{+}_{\mathrm{st}}$.
  \end{enumerate}
\end{lemma}

\begin{remark}\label{remark:main-gluing}
  Before we give the proof we explain why Lemma \ref{lemma:kernel-gluing} implies Proposition \ref{prop:main-gluing}. First we observe that (i) and (ii) above imply $$\mathrm{ind}(D^{\rho}_{\mathrm{st}})=\mathrm{ind}(D^{-}_{\mathrm{st}})+\mathrm{ind}(D^{+}_{\mathrm{st}}),$$ since all the operators are surjective. Using $\mathrm{ind}(D^{\rho}_{\mathrm{st}})=\mathrm{ind}(D^{\rho})+d+\delta$ and similar formulas for $\mathrm{ind}(D^{\pm}_{\mathrm{st}})$, we conclude
  \begin{equation}\label{eq:index-relate}
    \mathrm{ind}(D)=\mathrm{ind}(D^{\rho})=\mathrm{ind}(D^{-})+\mathrm{ind}(D^{+}).
  \end{equation}
  Once we recall the definitions of $D^{+}$, $D^{-}$ and how they compare with $D^{\mathrm{al}}$ and $D^{\mathrm{CZ}}_{A}$, we conclude Proposition \ref{prop:main-gluing} in the case when $\Gamma_{+}=\set{z}$ and $\Gamma_{-}=\emptyset$.

  More generally, we can apply Lemma \ref{lemma:kernel-gluing} one time for each positive puncture and conclude that Proposition \ref{prop:main-gluing} holds when $\Gamma^{-}=\emptyset$.

  There is an obvious variant of Lemma \ref{lemma:kernel-gluing} in the case of a negative puncture $z$, where we consider the deformation
  \begin{equation*}
    D^{\rho}=\bd_{s}-(1-\beta(s+3\rho))A-\beta(s+3\rho)A^{\mathrm{al}},
  \end{equation*}
  defined for $s\le 0$. As above, we suppose $D^{\rho}$ is fixed on the complement of $C_{z}$. The same gluing argument shows that $\mathrm{ind}(D^{\rho})$ agrees with the sum of the indices of the operators 
  \begin{equation*}
    D^{+}=\bd_{s}-A^{\mathrm{al}}\hspace{1cm}D^{-}=\bd_{s}-(1-\beta(s))A-\beta(s)A^{\mathrm{al}}=:D^{\mathrm{ZC}}_{A}.
  \end{equation*}
  Here $D^{+}$ extends to $\dot\Sigma$ (i.e.,\ $D^{\rho}=D^{+}$ is fixed on the complement of $C_{z}$) while $D^{-}$ is defined on an infinite strip or cylinder.

  By performing these deformations at all punctures (one at a time), we ultimately conclude that
  \begin{equation}\label{eq:ultimate-conclusion}
    \mathrm{ind}(D)=\mathrm{ind}(D^{\mathrm{al}})+\sum_{z\in \Gamma_{+}}\mathrm{ind}(D^{\mathrm{CZ}}_{A})+\sum_{z\in \Gamma_{-}}\mathrm{ind}(D^{\mathrm{ZC}}_{A}).
  \end{equation} 
  Finally, consider the following family of operators on the infinite cylinder or strip:
  \begin{equation*}
    D^{\rho}=\bd_{s}-(1-\beta(s))A-\beta(s)(1-\beta(s-3\rho))A^{\mathrm{al}}-\beta(s)(\beta(s-3\rho))A.
  \end{equation*}
  We can think of this as gluing $D^{\mathrm{CZ}}_{A}$ to the positive end of $D^{\mathrm{ZC}}_{A}$. Indeed, this fits into the framework considered in Lemma \ref{lemma:kernel-gluing}, and so we conclude that
  \begin{equation*}
    \mathrm{ind}(D^{\rho})=\mathrm{ind}(D^{\mathrm{ZC}}_{A})+\mathrm{ind}(D^{\mathrm{CZ}}_{A}).
  \end{equation*}
  It is clear that if we let $\rho$ become very negative, then $D^{\rho}$ agrees with $\bd_{s}-A$, which has Fredholm index $0$ (by Theorem \ref{theorem:isomorphism}). Since the Fredholm index of $D^{\rho}$ is constant as a function of $\rho$ we must have
  \begin{equation*}
    \mathrm{ind}(D^{\mathrm{ZC}}_{A})=-\mathrm{ind}(D^{\mathrm{CZ}}_{A})=-\mu_{\mathrm{CZ}}(A).
  \end{equation*}
  This combined with \eqref{eq:ultimate-conclusion} completes the proof of Proposition \ref{prop:main-gluing}.
\end{remark}

\begin{proof}[of Lemma \ref{lemma:kernel-gluing}]
  To prove that $D^{\rho}_{\mathrm{st}}$ is surjective, we will attempt to solve the equation $D^{\rho}_{\mathrm{st}}(\xi)=\eta$ for some $\eta\in L^{p}$. Fix three bump functions $b_{1}^{\rho},b_{2}^{\rho},b_{3}^{\rho}$, all supported in $C_{z}$ by the formulas
  \begin{equation*}
    b_{1}^{\rho}(s)=\beta(s/\rho)\hspace{1cm}    b_{2}^{\rho}(s)=\beta(1-s/\rho)\hspace{1cm}    b_{3}^{\rho}(s)=\beta(2-s/\rho).
  \end{equation*}
  See Figure \ref{fig:bump-fxns}. 
  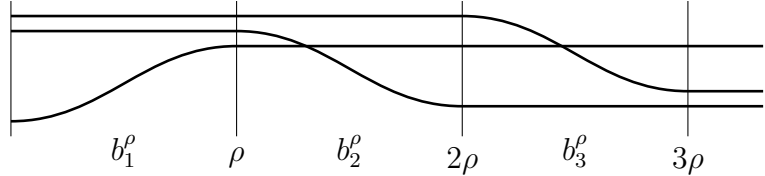
\begin{figure}[H]
    \centering
    \begin{tikzpicture}
      \draw (0,-0.2)--(0,1.6) (3,-0.2)node[below]{$\rho$}--(3,1.6) (6,-0.2)node[below]{$2\rho$}--(6,1.6) (9,-0.2)node[below]{$3\rho$}--(9,1.6);
      \draw[shift={(0,0.2)},line width=1pt] (0,1)--(3,1) to[out=0,in=180] (6,0) -- (10,0);
      \draw[shift={(0,0.4)},line width=1pt] (0,1)--(6,1) to[out=0,in=180] (9,0) -- (10,0);
      \draw[line width=1pt] (0,0) to[out=0,in=180] (3,1) -- (10,1);
      \path (0,-0.4)--+(1.5,0)node{$b_{1}^{\rho}$}--+(4.5,0)node{$b_{2}^{\rho}$}--+(7.5,0)node{$b_{3}^{\rho}$};
    \end{tikzpicture}
    \caption{Three bump functions drawn with slight vertical offsets to better show their behavior.}
    \label{fig:bump-fxns}
  \end{figure}
  By picking $\rho$ large enough, we may suppose that $(c_{1},\cdots,c_{d})$ are supported where $b_{2}^{\rho}=1$ and $(\gamma_{1}(s-3\rho),\cdots,\gamma_{\delta}(s-3\rho))$ are supported where $b_{2}^{\rho}=0$. This assumption will simplify some calculations later on.
  
  Now let $\eta\in L^{p}(E)$ be some section. Since $D^{-}_{\mathrm{st}}$ has a bounded right inverse, we can find $\xi_{1}$ and $\mathfrak{c}_{1}=\sum a_{j}c_{j}$ so that
  \begin{equation*}
    D^{-}(\xi_{1})+\mathfrak{c}_{1}=b_{2}^{\rho}\eta.
  \end{equation*}
  Moreover we can achieve this so that $\norm{(\xi_{1},\mathfrak{c}_{1})}=\norm{\mathfrak{c}_{1}}+\norm{\xi_{1}}_{W^{1,p}}$ is bounded by $C^{-}\norm{\eta}_{L^{p}}$ for a fixed constant $C^{-}$ (by fixing a bounded right inverse for $D^{-}_{\mathrm{st}}$). Here $\norm{\mathfrak{c}_{1}}$ is any norm on $\R^{d}$ (which we fix throughout the proof).
  
  Because of $b_{3}^{\rho}\mathfrak{c}_{1}=\mathfrak{c}_{1}$ and $b_{3}^{\rho}b_{2}^{\rho}=b_{2}^{\rho}$, we have
  \begin{equation*}
D^{-}(b_{3}^{\rho}\xi_{1})+\mathfrak{c}_{1}-b_{2}^{\rho}\eta=    D^{-}(b_{3}^{\rho}\xi_{1})+b_{3}^{\rho}\mathfrak{c}_{1}-b_{3}^{\rho}b_{2}^{\rho}\eta=\cl{\bd}(b_{3}^{\rho})\otimes \xi_{1}.
  \end{equation*}
  Since $b_{3}^{\rho}\xi_{1}$ is supported in the region where $D^{-}=D^{\rho}$, we can rewrite the above as
  \begin{equation*}
    D^{\rho}(b_{3}^{\rho}\xi_{1})+\mathfrak{c}_{1}=b_{2}^{\rho}\eta+(\cl{\bd}b_{3}^{\rho})\otimes \xi_{1}.
  \end{equation*}
  Now observe that $\Delta=\eta-b_{2}^{\rho}\eta$ is supported in the region where $s\ge \rho$. Since $D^{+}_{\mathrm{st}}$ is surjective, we can find $\xi^{\prime}_{2}$ and $\mathfrak{c}^{\prime}_{2}=\sum b_{k}\gamma_{k}$ so that
  \begin{equation*}
D^{+}(\xi_{2}^{\prime})+\mathfrak{c}^{\prime}_{2}=    \Delta(s+3\rho,t).
  \end{equation*}
  We can achieve this with $\norm{\mathfrak{c}_{2}^{\prime}}+\norm{\xi_{2}^{\prime}}_{W^{1,p}}\le C^{+}\norm{\eta}_{L^{p}}$ for a fixed constant~$C^{+}$.
  
  Let $\xi_{2}(s,t)=\xi_{2}(s-3\rho,t)$ and $\mathfrak{c}_{2}(s,t)=\mathfrak{c}^{\prime}_{2}(s-3\rho,t)$. Since $b_{1}^{\rho}\Delta=\Delta$ and $b_{1}^{\rho}\mathfrak{c}_{2}=\mathfrak{c}_{2}$, we conclude that
  \begin{equation*}
    D^{\rho}(b_{1}^{\rho}\xi_{2})+\mathfrak{c}_{2}=\Delta+(\cl{\bd}b_{1}^{\rho})\otimes \xi_{2}.
  \end{equation*}
  Consequently, we have
  \begin{equation*}
    D^{\rho}(b_{3}^{\rho}\xi_{1}+b_{1}^{\rho}\xi_{2})+\mathfrak{c}_{1}+\mathfrak{c}_{2}=\eta+(\cl{\bd}b_{3}^{\rho})\otimes \xi_{1}+(\cl{\bd}b_{1}^{\rho})\otimes \xi_{2}.
  \end{equation*}
  Observe that the derivatives of $b^{\rho}_{i}$ are of order $\rho^{-1}$. We think of this as \emph{approximately} solving $D^{\rho}_{\mathrm{st}}(b_{3}^{\rho}\xi_{1}+b_{1}^{\rho}\xi_{2},\mathfrak{c}_{1},\mathfrak{c}_{2})=\eta$. Indeed, we have just shown that for any $\eta$ we can find $\xi,\mathfrak{c}_{1},\mathfrak{c}_{2}$ so that
  \begin{equation}\label{eq:important-3}
    \norm{\xi}_{W^{1,p}}+\norm{\mathfrak{c}_{1}}+\norm{\mathfrak{c}_{2}}\le C\norm{\eta}_{L^{p}} \text{ and }\norm{D^{\rho}_{\mathrm{st}}(\xi,\mathfrak{c}_{1},\mathfrak{c}_{2})-\eta}_{L^{p}}\le C\rho^{-1}\norm{\eta}_{L^{p}},
  \end{equation}
  for constants $C$ independent of $\rho$.

  The equation \eqref{eq:important-3} implies that $D^{\rho}_{\mathrm{st}}$ is surjective for $\rho$ large enough, as follows: pick $\rho$ so $C\rho^{-1}<1/2$. By \eqref{eq:important-3} with $\eta:=\eta-D^{\rho}_{\mathrm{st}}(\xi,\mathfrak{c}_{1},\mathfrak{c}_{2})$ we obtain $\xi^{1},\mathfrak{c}_{1}^{1},\mathfrak{c}_{2}^{1}$ so that
  \begin{equation*}    \norm{D^{\rho}_{\mathrm{st}}(\xi^{1},\mathfrak{c}_{1}^{1},\mathfrak{c}_{2}^{1})-(\eta-D^{\rho}_{\mathrm{st}}(\xi,\mathfrak{c}_{1},\mathfrak{c}_{2}))}_{L^{p}}\le \frac{1}{4}\norm{\eta}_{L^{p}},
  \end{equation*}
  and $\norm{(\xi^{1},\mathfrak{c}_{1}^{1},\mathfrak{c}_{2}^{1})}\le C2^{-1}\norm{\eta}_{L^{p}}$. In other words, if we try to solve for the \emph{error} arising from our first attempt to solve for $\eta$, then $D^{\rho}_{\mathrm{st}}(\xi,\mathfrak{c}_{1},\mathfrak{c}_{2})+D^{\rho}_{\mathrm{st}}(\xi^{1},\mathfrak{c}_{1}^{1},\mathfrak{c}_{2}^{1})$ is a better approximation by a factor of $1/2$ than our initial attempt. 
  
  By repeating this process, we can find a sequence $\xi^{n},\mathfrak{c}_{1}^{n},\mathfrak{c}_{2}^{n}$ so that
  \begin{equation*}
    \norm{(\xi^{n},\mathfrak{c}_{1}^{n},\mathfrak{c}_{2}^{n})}\le C2^{-n}\norm{\eta}_{L^{p}}\text{ and }\norm{\textstyle{\sum_{j=0}^{n}}D^{\rho}_{\mathrm{st}}(\xi^{j},\mathfrak{c}_{1}^{j},\mathfrak{c}_{2}^{j})-\eta}_{L^{p}}\le 2^{-n-1}\norm{\eta}_{L^{p}}.
  \end{equation*}
  The above series then converges to an element in the preimage of $\eta$, as desired. This completes the proof that $D^{\rho}_{\mathrm{st}}$ is surjective for $\rho$ sufficiently large. Moreover, we see that $D^{\rho}_{\mathrm{st}}$ actually has a right inverse which is bounded in norm by $2C$. This uniformly bounded right inverse will play a role later on.

  Next we need to prove that $\dim \ker D^{\rho}_{\mathrm{st}}=\dim \ker D^{-}_{\mathrm{st}}+\dim \ker D^{+}_{\mathrm{st}}$. First we will prove that
  \begin{equation}\label{eq:inequ}
    \dim \ker D^{\rho}_{\mathrm{st}}\le \dim \ker D^{-}_{\mathrm{st}}+\dim \ker D^{+}_{\mathrm{st}}.
  \end{equation}

  Suppose that $(\xi,\mathfrak{c}_{1},\mathfrak{c}_{2})$ lies in the kernel of $D^{\rho}_{\mathrm{st}}$. Using the same bump functions from before, we compute
  \begin{equation*}
    D^{-}(b_{2}^{\rho}\xi)+\mathfrak{c}_{1}=D^{\rho}(b_{2}^{\rho}\xi)+\mathfrak{c}_{1}=b_{2}^{\rho}(D^{\rho}(\xi)+\mathfrak{c}_{1}+\mathfrak{c}_{2})+\cl{\bd}(b_{2}^{\rho})\otimes \xi=\cl{\bd}(b_{2}^{\rho})\otimes \xi.
  \end{equation*}
  In particular, $(b_{2}^{\rho}\xi,\mathfrak{c}_{1})$ is \emph{close} to lying in the kernel of $D^{-}_{\mathrm{st}}$ (up to an error of size $\rho^{-1}\norm{\xi}$). Indeed, using the bounded right inverse for $D^{-}_{\mathrm{st}}$ we can estimate
  \begin{equation*}
    \norm{(b_{2}^{\rho}\xi,\mathfrak{c}_{1})-\ker(D_{\mathrm{st}}^{-})}\le C\norm{\xi}\rho^{-1}.
  \end{equation*}

  On the other hand, we have
  \begin{equation*}    0=(1-b_{2}^{\rho})(D^{\rho}(\xi)+\mathfrak{c}_{1}+\mathfrak{c}_{2})=D^{\rho}((1-b_{2}^{\rho})\xi)+\mathfrak{c}_{2}-\cl{\bd}b^{\rho}_{2}\otimes \xi,
  \end{equation*}
  so the translated element $((1-b_{2}^{\rho})\xi,\mathfrak{c}_{2})(s+3\rho,t)$ is close to the kernel of $D^{+}_{\mathrm{st}}$.

  We can encode these as a linear map $\Phi:\ker D^{\rho}_{\mathrm{st}}\to (L^{p}\times \R^{d})\oplus (L^{p}\times \R^{\delta})$:
  \begin{equation*}
    \Phi(\xi,\mathfrak{c}_{1},\mathfrak{c}_{2})=\begin{dmatrix}
      {(b_{2}^{\rho}\xi,\mathfrak{c}_{1})}\\
      {((1-b_{2}^{\rho})\xi,\mathfrak{c}_{2})(s+3\rho,t)}
    \end{dmatrix}.
  \end{equation*}
  It is clear that $\Phi$ is uniformly injective (we simply add together its components to recover $\xi,\mathfrak{c}_{1},\mathfrak{c}_{2}$ -- this defines a fixed left inverse). We will now estimate the rank of $\Phi$. By the preceding remarks, we have
  \begin{equation*}  \norm{\Phi(\xi,\mathfrak{c}_{1},\mathfrak{c}_{2})-\ker(D^{-}_{\mathrm{st}})\oplus \ker(D^{+}_{\mathrm{st}})}\le C\norm{\xi}\rho^{-1}.
  \end{equation*}
  Let $\Pi$ be a projection onto $\ker(D^{-}_{\mathrm{st}})\oplus \ker (D^{+}_{\mathrm{st}})$. We have
  \begin{equation*}
    \norm{(1-\Pi)\circ \Phi(\xi,\mathfrak{c}_{1},\mathfrak{c_{2}})}\le C^{\prime}\norm{\xi}\rho^{-1}.
  \end{equation*}
  Thus $\Phi=\Pi\circ \Phi+\text{(error of size $\rho^{-1}$)},$ where the error is measured in the operator norm. Since $\Phi$ is uniformly injective, we conclude that $\Pi\circ \Phi$ must also be injective for $\rho$ large enough. Hence $\Pi\circ \Phi$ is an injection from $\ker(D^{\rho}_{\mathrm{st}})$ into $\ker(D^{-}_{\mathrm{st}})\oplus \ker(D^{+}_{\mathrm{st}})$, proving \eqref{eq:inequ}.
  
  Finally we prove the reverse inequality:
  \begin{equation}
    \label{eq:revinequ}
    \dim \ker D^{\rho}_{\mathrm{st}}\ge \dim \ker(D^{-}_{\mathrm{st}})+\dim \ker(D^{+}_{\mathrm{st}}).
  \end{equation}
  The strategy will be ``glue'' together elements in the kernels of $D^{\pm}_{\mathrm{st}}$ and obtain elements approximately in the kernel of $D^{\rho}_{\mathrm{st}}$, and then use the fact that $D^{\rho}_{\mathrm{st}}$ has a (uniformly) bounded right  inverse (which we proved above) to show that we can deform these approximate kernel elements into actual kernel elements.

  So, let $(\xi_{1},\mathfrak{c}_{1})\in \ker(D^{-}_{\mathrm{st}})$ and let $(\xi_{2}^{\prime},\mathfrak{c}_{2}^{\prime})\in \ker(D^{+}_{\mathrm{st}})$. Let $(\xi_{2},\mathfrak{c}_{2})=(\xi_{2}^{\prime},\mathfrak{c}_{2}^{\prime})(s-3\rho,t)$. Recall from Section \ref{sec:stabilized-operators} that we must have $\mathfrak{c}_{1}=\mathfrak{c}_{2}=0$.

  Then it is straightforward to check that:
  \begin{equation}\label{eq:DrhoP}
    D^{\rho}(\beta_{2}^{\rho}\xi_{1}+(1-\beta_{2}^{\rho})\xi_{2})=(\cl{\bd}\beta_{2}^{\rho})(\xi_{1}-\xi_{2}).
  \end{equation}
  Let $\Phi(\xi_{1},\xi_{2}^{\prime})=\beta_{2}^{\rho}\xi_{1}+(1-\beta_{2}^{\rho})\xi_{2}$. First we show that $\Phi$ is uniformly injective. Indeed, the injectivity estimates for $\bd_{s}u-Au=0$ from Section \ref{sec:injectivity-estimates} imply that 
  \begin{equation}\label{eq:unif-injective}
    \begin{aligned}
      \norm{\xi_{1}}\le C\norm{\xi_{1}}_{\Sigma(\rho)}\le C\norm{\Phi(\xi_{1},\xi_{2})},\\
      \norm{\xi_{2}^{\prime}}\le C\norm{\xi_{2}^{\prime}}_{(-\rho,\infty)\times I}\le C\norm{\Phi(\xi_{1},\xi_{2}^{\prime})},
    \end{aligned}
  \end{equation}
  for a uniform constant $C$.\footnote{The idea is to write
    \begin{equation*}
      \xi_{1}=(1-\beta_{1}^{\rho})\xi_{1}+\beta_{1}^{\rho}\xi_{1}.
    \end{equation*}
    Then $D^{-}(\beta_{1}^{\rho}\xi_{1})=\cl{\bd}{\beta_{1}^{\rho}}\otimes \xi_{1}$. Observe that $\beta_{1}^{\rho}\xi_{1}$ is supported in the region where $D^{-}$ is translation invariant. Thus we can apply the injectivity estimates and conclude that the $W^{1,p}$ size of $\beta_{1}^{\rho}\xi_{1}$ is controlled by the $L^{p}$ size of $\xi_{1}$ on $[0,\rho]\times I$. The constant $C$ gets better (closer to $1$) as $\rho$ increases. Similar considerations establish the second part of \eqref{eq:unif-injective}.} In particular
  \begin{equation}\label{eq:unif-inj}
    \norm{\xi_{1}}+\norm{\xi_{2}^{\prime}}\le2C\norm{\Phi(\xi_{1},\xi_{2})}.
  \end{equation}

  Now let $B$ be a bounded right inverse for $D^{\rho}_{\mathrm{st}}$, and consider
  \begin{equation*}
    \Phi^{\prime}=\Phi-B\circ D^{\rho}_{\mathrm{st}}\circ \Phi.
  \end{equation*}
  Because of \eqref{eq:DrhoP}, $B\circ D^{\rho}_{\mathrm{st}}\circ \Phi$ has operator norm bounded by $\rho^{-1}$ (here we assume that the operator norm of $B$ is bounded as $\rho\to\infty$; the first part of our proof shows that this can be achieved).

  Then for $\rho$ large enough, $\Phi'$ is also injective as it is a small perturbation of an injective operator (i.e.,\ the estimate \eqref{eq:unif-inj} will still hold, modulo increasing $C$ slightly).

  Thus $\Phi^{\prime}$ injects $\ker(D^{-})\oplus \ker(D^{+})$ into $\ker(D^{\rho}_{\mathrm{st}})$, establishing \eqref{eq:revinequ}. This completes the proof of the Lemma.
\end{proof}

\section{The index formula for large anti-linear deformations}\label{sec:final-frontier}

Fix a trivialization $\tau$ of $(\Sigma,\Gamma,E,F,C,[\tau])$, as above, and let $D^{\mathrm{al}}$ be a Cauchy-Riemann operator whose restriction to each end $C_{z}$ is equal to $\bd_{s}-A^{\mathrm{al}}$ (in the trivialization $\tau$).

Our goal in this section is to compute the Fredholm index $\mathrm{ind}(D^{\mathrm{al}})$. The formula will be in terms of the following invariants:
\begin{enumerate}
\item The \emph{Euler characteristic} $\mathrm{X}:=\mathrm{X}(\Sigma,\bd\Sigma,\Gamma_{\pm})$ is the weighted count of zeros of a transverse vector field which equals $\bd_{s}$ in each end $C_{z}$ and is everywhere tangent to $\bd\Sigma$ (the zeros are counted as explained in Section \ref{sec:euler}, see also Figure \ref{fig:boundary-zeros}).

\item The \emph{Maslov index} $\mu^{\tau}_{\mathrm{Mas}}:=\mu^{\tau}_{\mathrm{Mas}}(E,F)$ is the signed count of zeros of a transverse section of $(\det E)^{\otimes 2}$ which (a) restricts to the canonical generator of $(\det F)^{\otimes 2}$ along the boundary and (b) equals $1$ in each end $C_{z}$ (this last part uses $\tau$). Notice that all the zeros will necessarily be interior. 
\end{enumerate}
The main result in this section is:
\begin{prop}\label{prop:20}
  The Fredholm index of $D^{\mathrm{al}}:W^{1,p}(E,F)\to L^{p}(\Lambda^{1,0}\otimes E)$ is
  \begin{equation*}
    \mathrm{ind}(D^{\mathrm{al}})=n\mathrm{X}+\mu^{\tau}_{\mathrm{Mas}},
  \end{equation*}
  where $n$ is the complex rank of $E$.
\end{prop}

The proof of Proposition \ref{prop:20} breaks into two parts. In Section \ref{sec:lb-reduc} we reduce to the case when $E$ is a line bundle (so $E=\det(E)$ and $F=\det(F)$). In Section \ref{sec:large-antilinear} we prove Proposition \ref{prop:20} in the case when $E$ is a line bundle by considering the $\sigma\to\infty$ limiting behavior of $D^{\mathrm{al}}+\sigma B^{\mathrm{al}}$ where $B^{\mathrm{al}}$ is a special anti-linear deformation (we only deform the lower order terms). This is the strategy introduced in \cite[Section 7]{taubes} and generalized in \cite[Chapter 3]{gerig}.

\subsection{Reduction to the case of line bundles}
\label{sec:lb-reduc}

In this section we assume that Proposition \ref{prop:20} is true for line bundles, and we deduce it holds for all bundles. We will split $(E,F)$ into a direct sum $$(E,F)=\underbrace{(\C,\R)\oplus \cdots\oplus (\C,\R)}_{n-1\text{ copies}}\oplus (\det(E),\det(F))$$ in a way compatible with the trivialization $\tau$. In order to do the splitting, we fix a Hermitian metric $\mu$ on $(E,F)$ extending the Hermitian metric in the ends $C_{z}$.

Consider the trivialization $\tau$. This defines a unitary frame $X_{1},\cdots,X_{n}$ in the ends. If $n>1$, we can extend $X_{1}$ over $\bd\dot\Sigma$ as a non-zero section of $F$, which we may normalize so $\abs{X_{1}}=1$. Let $E_{1}$ denote the $\mu$-orthogonal complement of $X_{1}$, and let $F_{1}=E_{1}\cap F$. Note that $F_{1}$ is $n-1$ dimensional and is totally real for $E_{1}$.

Notice that $X_{2},\cdots,X_{n}$ are all sections of $(E_{1},F_{1})\subset (E,F)$ in the ends. If $n>2$ then we can extend $X_{2}$ as a nonzero section of $(E_{1},F_{1})$. We continue in this fashion until we conclude that $X_{1},\cdots,X_{n-1}$ extend as a global unitary frame in $(E,F)$ (in the sense that they are mutually $\mu$-orthogonal and all unit vectors).

Let $E_{n}$ be the $\mu$ orthogonal complement to $X_{1},\cdots,X_{n-1}$ and $F_{n}=E_{n}\cap F$, and notice that $X_{n}$ trivializes $(E_{n},F_{n})$ in the ends.

By construction, $D^{\mathrm{al}}$ is given by
\begin{equation*}
  D^{\mathrm{al}}(\sum u_{k}X_{k})=(\bd_{s}u_{k}+i\bd_{k}u_{k}+\cl{u_{k}})(\d s-i\d t\otimes X_{k})
\end{equation*}
in the ends. In particular $D^{\mathrm{al}}$ \emph{splits} in the ends. By perturbing $D^{\mathrm{al}}$ away from the ends, we may suppose it splits everywhere. This means that if $u$ takes values in the line $\C X_{k}$ (resp.\ $E_{n}$), then $D^{\mathrm{al}}(u)$ takes values in $\Lambda^{0,1}\otimes \C X_{k}$ (resp.\ $F_{n}$). It follows that the induced operator splits as a diagonal matrix of Cauchy-Riemann operators asymptotic to the one-dimensional version of $\bd_{s}-A^{\mathrm{al}}$:
\begin{equation*}
  D^{\mathrm{al}}:[\bigoplus_{k=1}^{n-1}W^{1,p}(\C X_{k},\R X_{k})]\oplus (E_{n},F_{n})\to [\bigoplus_{k=1}^{n-1}L^{p}(\Lambda^{0,1}\otimes \C X_{k})]\oplus L^{p}(\Lambda^{0,1}\otimes E_{n}).
\end{equation*}
Let $D^{\mathrm{al}}_{k}$ be the $k$th factor in the above decomposition. The Fredholm index is additive under diagonal decompositions. Since Proposition \ref{prop:20} applies to $D^{\mathrm{al}}_{k}$ we conclude that
\begin{equation*}  \mathrm{ind}(D^{\mathrm{al}})=[\sum_{k=1}^{n-1}\mathrm{ind}(D^{\mathrm{al}}_{k})]+\mathrm{ind}(D^{\mathrm{al}}_{n})=n\mathrm{X}+\mu_{\mathrm{Mas}}^{\tau}(E_{n},F_{n}).
\end{equation*}
Finally, fix $\mathfrak{s}$ a transverse section of $E_{n}^{\otimes 2}$ which restricts to the canonical generator of $F_{n}^{\otimes 2}$ and which equals $1\simeq X_{n}^{\otimes 2}$ in the end. Locally write $\mathfrak{s}=\mathfrak{s}_{1}\otimes \mathfrak{s}_{2}$, and define
\begin{equation}\label{eq:identification}
  \mathfrak{s}'=(X_{1}\wedge \cdots \wedge X_{n-1}\wedge \mathfrak{s}_{1})\otimes(X_{1}\wedge \cdots \wedge X_{n-1}\wedge \mathfrak{s}_{2}).
\end{equation}
This does not depend on the decomposition $\mathfrak{s}=\mathfrak{s}_{1}\otimes\mathfrak{s}_{2}$ since $E_{n}$ is a complex line bundle.

Then $\mathfrak{s}^{\prime}$ is a transverse section of $\det(E)^{\otimes 2}$ which restricts to the canonical generator of $\det(F)^{\otimes 2}$. The signed count of zeros of $\mathfrak{s}^{\prime}$ agrees with the count of zeros of $\mathfrak{s}$ as they locally differ by application of a fiber-wise complex linear isomorphism (namely, the map induced by \eqref{eq:identification}). Thus we conclude $\mu^{\tau}_{\mathrm{Mas}}(E,F)=\mu^{\tau}_{\mathrm{Mas}}(E_{n},F_{n}).$ This completes the proof of the reduction to the line bundle case.

\subsection{Large anti-linear deformations}
\label{sec:large-antilinear}
Let $(E,F)$ be a line bundle with asymptotic trivialization $\tau$. As in the previous section, we can consider $\tau$ as defining a non-vanishing section $X$ in the ends which takes boundary values in $F$. In other words $(E,F)=(\C X,\R X)$ in each end.

Our strategy will be to define a particular family of operators $D^{\sigma}$, $\sigma>0$, whose asymptotic form with respect to the trivialization $\tau$ is equal to $\bd_{s}+i\bd_{t}+\sigma C$. Since we have shown $A^{\mathrm{al},\sigma}=-i\bd_{t}-\sigma C$ is non-degenerate for all $\sigma>0$ (Lemma \ref{lemma:reference-non-degen}), we conclude that $D^{\sigma}$ is always Fredholm. Moreover, when $\sigma=1$, $D^{\sigma}=D^{\mathrm{al}}$. Therefore
\begin{equation*}
  \mathrm{ind}(D^{\mathrm{al}})=\lim_{\sigma\to\infty}\mathrm{ind}(D^{\sigma}).
\end{equation*}
Via another index gluing argument, we will be able to relate $\mathrm{ind}(D^{\sigma})$ for large $\sigma$ to the weighted count of zeros of a certain section used to define $D^{\sigma}$, and ultimately conclude
\begin{equation*}
  \mathrm{ind}(D^{\sigma})=\mathrm{X}+\mu^{\tau}_{\mathrm{Mas}}\text{ for $\sigma\gg 0$}.
\end{equation*}
This will complete the proof of Proposition \ref{prop:20}.

\subsubsection{Defining the family $D^{\sigma}$}
\label{sec:defining-d-sigma}

We will now carefully define the family $D^{\sigma}$ in such a way which will facilitate the later analysis. Pick a Hermitian metric $\mu$ on all of $E$ so that $\abs{X}=1$. Now consider the section $M=X\otimes X$ of $F^{\otimes 2}\to  \bd C_{z}.$

We can extend this section as a \emph{non-vanishing} section of $F^{\otimes 2}\to \bd\dot\Sigma$ as follows: on any contractible open subset of $\bd\dot\Sigma$, let $M=Y\otimes Y$ where $Y\in \Gamma(F)$ satisfies $\abs{Y}=1$ using the metric $\mu$. Since there is a unique unit vector lying in $F$ up to $\pm 1$, we conclude that these local descriptions of $M$ agree on their overlaps. Clearly, in each end, $M=X\otimes X$. We should note that $X$ may \emph{not} extend over the boundary $\bd\dot\Sigma$, (but, as we have seen, $M$ always does).

Now extend $M$ to the interior of $\dot\Sigma$ as a section of $E\otimes E$ so that all of its zeros are transverse. By the same considerations of the linearization of a vector field given in Section \ref{sec:euler}, we can deform $M$ near each zero $\zeta$ so that, for some $D(1)$-valued holomorphic coordinate $z$ centered at $\zeta$, and some unitary frame $Y$ for $E$, we have $M=-z Y\otimes Y$ or $M=\cl{z}Y\otimes Y$, depending on the sign of the determinant of the linearization of $M$ at $\zeta$. By definition, $\mu^{\tau}_{\mathrm{Mas}}(E,F)$ is the signed count of zeros of $M$. See Figure \ref{fig:standard-rep-int}.

It will be useful to recall that $E\otimes E$ is complex linearly isomorphic to $\mathrm{Hom}^{0,1}(E,E)$ via the map $Y\otimes X\mapsto \mu(-,Y)X$, where $\mu$ is our chosen Hermitian metric. Let $M_{*}$ denote the image of $M$ under this isomorphism (so $M_{*}$ is a section of $\mathrm{Hom}^{0,1}(E,E)$).

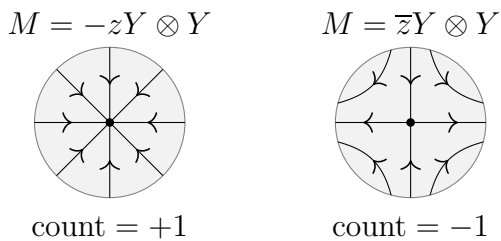
\begin{figure}[H]
  \centering
  \begin{tikzpicture}[scale=0.5]
    \begin{scope}[shift={(0,0)}]
          \node at (0,2) [above] {$M=-zY\otimes Y$};
          \draw[black!50!white,fill=black!5!white] (0,0) circle (2);
          
          \draw[postaction={decorate,decoration={markings,mark=at position 0.5 with {\arrow[scale=1.5]{>};}}}] (-2,0)--(0,0);
          \draw[postaction={decorate,decoration={markings,mark=at position 0.5 with {\arrow[scale=1.5]{>};}}}] (2,0)--(0,0);
          \draw[postaction={decorate,decoration={markings,mark=at position 0.5 with {\arrow[scale=1.5]{>};}}}] (0,2)--(0,0);
          \draw[postaction={decorate,decoration={markings,mark=at position 0.5 with {\arrow[scale=1.5]{>};}}}] (0,-2)--(0,0);
          \draw[postaction={decorate,decoration={markings,mark=at position 0.5 with {\arrow[scale=1.5]{>};}}}] (45:2)--(0,0);
          \draw[postaction={decorate,decoration={markings,mark=at position 0.5 with {\arrow[scale=1.5]{>};}}}] (135:2)--(0,0);
          \draw[postaction={decorate,decoration={markings,mark=at position 0.5 with {\arrow[scale=1.5]{>};}}}] (225:2)--(0,0);
          \draw[postaction={decorate,decoration={markings,mark=at position 0.5 with {\arrow[scale=1.5]{>};}}}] (315:2)--(0,0);
          \node[draw,circle,inner sep=1pt,fill] at (0,0){};
          \node at (0,-2.2)[below]{$\mathrm{count}=+1$};
        \end{scope}

\begin{scope}[shift={(8,0)}]
          \node at (0,2) [above] {$M=\cl{z}Y\otimes Y$};
          \draw[black!50!white,fill=black!5!white] (0,0) circle (2);
          \begin{scope}
            \clip (0,0) circle (2);
            \draw[postaction={decorate,decoration={markings,mark=at position 0.5 with {\arrow[scale=1.5]{>};}}}] plot[domain=0.5:2] ({\x},{1/\x});
            \draw[postaction={decorate,decoration={markings,mark=at position 0.5 with {\arrow[scale=1.5]{>};}}}] plot[domain=0.5:2] ({\x},{-1/\x});
            \draw[postaction={decorate,decoration={markings,mark=at position 0.5 with {\arrow[scale=1.5]{>};}}}] plot[domain=0.5:2] ({-\x},{1/\x});
            \draw[postaction={decorate,decoration={markings,mark=at position 0.5 with {\arrow[scale=1.5]{>};}}}] plot[domain=0.5:2] ({-\x},{-1/\x});            
          \end{scope}

          \draw[postaction={decorate,decoration={markings,mark=at position 0.5 with {\arrow[scale=1.5]{<};}}}] (-2,0)--(0,0);
          \draw[postaction={decorate,decoration={markings,mark=at position 0.5 with {\arrow[scale=1.5]{<};}}}] (2,0)--(0,0);
          \draw[postaction={decorate,decoration={markings,mark=at position 0.5 with {\arrow[scale=1.5]{>};}}}] (0,2)--(0,0);
          \draw[postaction={decorate,decoration={markings,mark=at position 0.5 with {\arrow[scale=1.5]{>};}}}] (0,-2)--(0,0);
          \node[draw,circle,inner sep=1pt,fill] at (0,0){};
          \node at (0,-2.2)[below]{$\mathrm{count}=-1$};
        \end{scope}
  \end{tikzpicture}
  \caption{After a slight deformation in a neighborhood of each zero, we may assume the zeros of $M$ have coordinate representations as either $-z$ or $\cl{z}$.}
  \label{fig:standard-rep-int}
\end{figure}

Next, we extend the vector field $\bd_{s}$ (defined in the ends) to all of $\Sigma$. We let $V$ be a vector field which (a) is everywhere tangent to $\bd\dot\Sigma$, (b) equals $\bd_{s}$ in the ends, (c) has non-degenerate zeros, and (d) its zeros are disjoint from the zeros of $M$. Unlike the section $M=X\otimes X$, we expect $V$ to have boundary zeros.

As explained in Section \ref{sec:euler} we can slightly deform $V$ (away from the ends), so that near each interior zero $p$ there is a holomorphic coordinate $z=s+it$ so that $V=-z\bd_{s}$ or $V=\cl{z}\bd_{s}$ (similarly to Figure \ref{fig:standard-rep-int}), and near each boundary zero we have one of four possibilities $V=\pm z\bd_{s}$, $V=\pm \cl{z}\bd_{s}$, as shown in Figure \ref{fig:boundary-zeros}. 
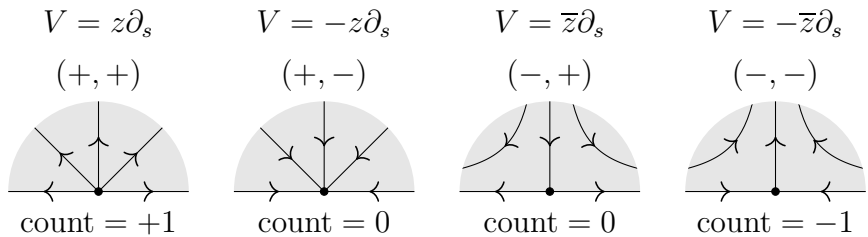
\begin{figure}[H]
  \centering
  \begin{tikzpicture}[scale=.6]
    \begin{scope}[shift={(0,0)}]
      \fill[black!10!white] (2,0) arc (0:180:2)--cycle;
      \node (A) at (0,2) [above] {$(+,+)$};
      \node at (A.north) [above] {$V=z\bd_{s}$};
      \draw[postaction={decorate,decoration={markings,mark=at position 0.5 with {\arrow[scale=1.5]{<};}}}] (-2,0)--(0,0);
      \draw[postaction={decorate,decoration={markings,mark=at position 0.5 with {\arrow[scale=1.5]{<};}}}] (2,0)--(0,0);
      \draw[postaction={decorate,decoration={markings,mark=at position 0.5 with {\arrow[scale=1.5]{<};}}}] (0,2)--(0,0);
      \draw[postaction={decorate,decoration={markings,mark=at position 0.5 with {\arrow[scale=1.5]{<};}}}] (45:2)--(0,0);
      \draw[postaction={decorate,decoration={markings,mark=at position 0.5 with {\arrow[scale=1.5]{<};}}}] (135:2)--(0,0);
      \node[draw,circle,inner sep=1pt,fill] at (0,0){};
      \node at (0,-0.2)[below]{$\mathrm{count}=+1$};
    \end{scope}
    \begin{scope}[shift={(5,0)}]
      \node (A) at (0,2) [above] {$(+,-)$};
      \node at (A.north) [above] {$V=-z\bd_{s}$};
      \fill[black!10!white] (2,0) arc (0:180:2)--cycle;
      \draw[postaction={decorate,decoration={markings,mark=at position 0.5 with {\arrow[scale=1.5]{>};}}}] (-2,0)--(0,0);
      \draw[postaction={decorate,decoration={markings,mark=at position 0.5 with {\arrow[scale=1.5]{>};}}}] (2,0)--(0,0);
      \draw[postaction={decorate,decoration={markings,mark=at position 0.5 with {\arrow[scale=1.5]{>};}}}] (0,2)--(0,0);
      \draw[postaction={decorate,decoration={markings,mark=at position 0.5 with {\arrow[scale=1.5]{>};}}}] (45:2)--(0,0);
      \draw[postaction={decorate,decoration={markings,mark=at position 0.5 with {\arrow[scale=1.5]{>};}}}] (135:2)--(0,0);
      \node[draw,circle,inner sep=1pt,fill] at (0,0){};
      \node at (0,-0.2)[below]{$\mathrm{count}=0$};
    \end{scope}
    \begin{scope}[shift={(10,0)}]
      \node (A) at (0,2) [above] {$(-,+)$};
      \node at (A.north) [above] {$V=\cl{z}\bd_{s}$};
      \fill[black!10!white] (2,0) arc (0:180:2)--cycle;
      \begin{scope}
        \clip (2,0) arc (0:180:2)--cycle;
        \draw[postaction={decorate,decoration={markings,mark=at position 0.5 with {\arrow[scale=1.5]{>};}}}] plot[domain=0.5:2] ({\x},{1/\x});
        \draw[postaction={decorate,decoration={markings,mark=at position 0.5 with {\arrow[scale=1.5]{>};}}}] plot[domain=0.5:2] ({-\x},{1/\x});
      \end{scope}
      \draw[postaction={decorate,decoration={markings,mark=at position 0.5 with {\arrow[scale=1.5]{<};}}}] (-2,0)--(0,0);
      \draw[postaction={decorate,decoration={markings,mark=at position 0.5 with {\arrow[scale=1.5]{<};}}}] (2,0)--(0,0);
      \draw[postaction={decorate,decoration={markings,mark=at position 0.5 with {\arrow[scale=1.5]{>};}}}] (0,2)--(0,0);
      \node[draw,circle,inner sep=1pt,fill] at (0,0){};
      \node at (0,-0.2)[below]{$\mathrm{count}=0$};
    \end{scope}
    \begin{scope}[shift={(15,0)}]
      \node (A) at (0,2) [above] {$(-,-)$};
      \node at (A.north) [above] {$V=-\cl{z}\bd_{s}$};
      \fill[black!10!white] (2,0) arc (0:180:2)--cycle;
      \begin{scope}
        \clip (2,0) arc (0:180:2)--cycle;
        \draw[postaction={decorate,decoration={markings,mark=at position 0.5 with {\arrow[scale=1.5]{<};}}}] plot[domain=0.5:2] ({\x},{1/\x});
        \draw[postaction={decorate,decoration={markings,mark=at position 0.5 with {\arrow[scale=1.5]{<};}}}] plot[domain=0.5:2] ({-\x},{1/\x});
      \end{scope}
      \draw[postaction={decorate,decoration={markings,mark=at position 0.5 with {\arrow[scale=1.5]{>};}}}] (-2,0)--(0,0);
      \draw[postaction={decorate,decoration={markings,mark=at position 0.5 with {\arrow[scale=1.5]{>};}}}] (2,0)--(0,0);
      \draw[postaction={decorate,decoration={markings,mark=at position 0.5 with {\arrow[scale=1.5]{<};}}}] (0,2)--(0,0);
      \node[draw,circle,inner sep=1pt,fill] at (0,0){};
      \node at (0,-0.2)[below]{$\mathrm{count}=-1$};
    \end{scope}
  \end{tikzpicture}
  \caption{The four models for a boundary zero of $V$. The first sign is from the linearization of $V:\Sigma\to T\Sigma$ and the second sign is from the linearization of the restriction $V:\bd\Sigma\to T\bd\Sigma$.}
  \label{fig:two-signs-2}
\end{figure}
By definition, the weighted count of the zeros of $V$ is the Euler characteristic $\mathrm{X}$.

It will be important to fix a Hermitian metric $\mu$ on $T\dot\Sigma$. We can do this so that $\abs{\bd_{s}}=1$ on all the coordinate charts introduced above (including the coordinate charts on the ends $C_{z}$, of course).

We are almost ready to define the operator $D^{\sigma}$. Two further simplifications we can do are the following:
\begin{enumerate}
\item Via a small deformation of $V$ away from the ends and its zeros, we may suppose that in the coordinate charts $z=s+it$ centered on the zeros of $M$, $V$ takes the form $\bd_{s}$, and
\item via a small deformation of $M$ away from the ends and its zeros, we may suppose that $M=Y\otimes Y$ for a unitary frame (with $Y|_{\bd\Sigma}\in F$) on the coordinate charts near the zeros of $V$.
\end{enumerate}

To summarize our setup, we have the following:
\begin{enumerate}[label=(\alph*)]
\item Holomorphic coordinate charts $z=s+it$ centered on each zero of $M$ and $V$. Boundary holomorphic coordinate charts are valued in $D(1)\cap \cl{\mathbb{H}}$.
\item Unitary metrics on $E$ and $T\dot\Sigma$ extending the metrics in the ends. Moreover we fix unitary sections $Y$ for $(E,F)$ defined on the domains of the coordinate charts from (a), and also suppose that $\abs{\bd_{s}}=1$ in each chart.
\item Near each zero of $M$, $V=\bd_{s}$ and $M$ equals $-zY\otimes Y$ or $\cl{z}Y\otimes Y$,
\item Near each interior zero of $V$, $M=Y\otimes Y$ and $V$ equals $-z\bd_{s}$ or $\cl{z}\bd_{s}$,
\item Near each boundary zero of $V$, $M=Y\otimes Y$ and $V$ equals $\pm z\bd_{s}$ or $\pm\cl{z}\bd_{s}$, (the $\pm$ signs are independent).
\end{enumerate}

Now fix $D_{0}$ to be a Cauchy-Riemann operator on $(E,F)$ which equals $\bd_{s}+i\bd_{t}$ in $C$ (with respect to $\tau$) and equals $\bd_{s}+i\bd_{t}$ with on the local trivializations induced by (a) and (b) above. This operator $D_{0}$ is not Fredholm, since its asymptotics are degenerate. We will perturb $D_{0}$ by the following lower order term
\begin{equation*}
  \xi\in \Gamma(E)\mapsto B(\xi)=\mu(-,V)\otimes M_{*}(\xi)\in \Gamma(\Lambda^{0,1}\otimes E).
\end{equation*}
Note that since $M_{*}$ is a section of $\mathrm{Hom}^{0,1}(E,E)$, $\xi\mapsto B(\xi)$ is anti-linear. We define
\begin{equation*}
  D^{\sigma}=D_{0}+\sigma B.
\end{equation*}
Before we proceed, let us verify that $D^{\sigma}$ has the correct ``$\mathrm{al}$'' asymptotics for $\sigma>0$.

In any of the asymptotic coordinate charts, we have $M_{*}=\mu(-,X)X$ and $V=\bd_{s}$ hence $M_{*}(uX)=\cl{u}X$ and
\begin{equation*}
  \begin{aligned}
    D^{\sigma}(uX)&=(\bd_{s}u+i\bd_{t}u)(\d s+i\d t)\otimes X+\sigma \cl{u}\mu(-,\bd_{s})\otimes X.\\
    &=(\bd_{s}u+i\bd_{t}u+\sigma\cl{u})(\d s+i\d t)\otimes X,
  \end{aligned}
\end{equation*}
where we have used the fact that $\mu(-,\bd_{s})=\d s-i\d t$ in the ends (indeed, this holds in all of our coordinate charts by the assumption that $\abs{\bd_{s}}=1$). Thus the local representation of $D^{\sigma}$ indeed equals $\bd_{s}+i\bd_{t}+\sigma C$, as desired.

As explained at the start of this section, this implies that the Fredholm index of $D^{\sigma}$ is constant for $\sigma>0$. Our task therefore reduces to the following lemma, which we will prove by deforming $\sigma\to+\infty$:
\begin{lemma}\label{lemma:final-baby-boy}
  The Fredholm index of $D^{\sigma}$ is equal
  \begin{equation*}
    \mathrm{ind}(D^{\sigma})=\mathrm{X}+\mu^{\tau}_{\mathrm{Mas}}.
  \end{equation*}
\end{lemma}
This lemma will complete the proof of Proposition \ref{prop:20}.

\subsubsection{Computing the local coordinate representations of $D^{\sigma}$}
\label{sec:computing-d-sigma}

In this section we will derive various formulas for $D^{\sigma}$ in coordinate charts. We have just shown that
\begin{equation}
  \label{eq:cr1}
  \text{in the ends $C_{z}$ we have: }D^{\sigma}=\bd_{s}+i\bd_{t}+\sigma C. 
\end{equation}
Near the zeros of $V$ and $M$, we compute the coordinate representation of $D^{\sigma}$ using the $s+it$ coordinate and the frame $Y$.

On a chart centered on a zero of $M$, we have $M_{*}=\alpha\mu(-,Y)Y$, where $\alpha=-z$ or $\alpha=\cl{z}$, and $V=\bd_{s}$. Similarly, near an interior zero of $V$, we have $M_{*}=\mu(-,Y)Y$ and $V=\alpha \bd_{s}$. In either case, we conclude:
\begin{equation}\label{eq:cr2}
  \begin{aligned}
    \text{at interior positive zeros: }D^{\sigma}(u)&=\bd_{s}u+i\bd_{t}u-\sigma z\cl{u},\\
    \text{at interior negative zeros: }D^{\sigma}(u)&=\bd_{s}u+i\bd_{t}u+\sigma \cl{z}\cl{u}.
  \end{aligned}
\end{equation}
Next we compute the coordinate representation of $D^{\sigma}$ near the boundary zeros, which we partition by their pair of signs $(\pm,\pm)$ as in Figure \ref{fig:two-signs-2}:
\begin{equation}\label{eq:cr3}
  \begin{aligned}
    \text{at $(+,\pm)$ type zeros: }D^{\sigma}(u)&=\bd_{s}u+i\bd_{t}u\pm\sigma z\cl{u},\\
    \text{at $(-,\pm)$ type zeros: }D^{\sigma}(u)&=\bd_{s}u+i\bd_{t}u\pm\sigma \cl{z}\cl{u}.
  \end{aligned}
\end{equation}

\subsection{Bochner-Weitzenb\"ock estimates and a linear compactness result}
\label{sec:weitzen}
Following \cite[Section 7]{taubes} and \cite[Chapter 5]{wendl-sft}, we show that $D^{\sigma}=D_{0}+\sigma B$ satisfies a ``Bochner-Weitzenb\"ock'' type estimate which will imply that kernel elements $\xi\in \ker D^{\sigma}$ and cokernel elements $\eta\in \ker D^{\sigma,\ast}$ concentrate near zeros of $B$. The key step is the following $L^{2}$ estimate:
\begin{lemma}[Bochner-Weitzenb\"ock estimates]
  Let $\xi\in W^{1,2}(E,F)$, then
  \begin{equation*}
    \norm{D_{0}\xi}^{2}_{L^{2}}+\sigma^{2}\norm{B(\xi)}_{L^{2}}^{2}\le \norm{D^{\sigma}\xi}_{L^{2}}^{2}+\sigma \norm{\xi}_{L^{2}}\norm{D^{*}_{0}(B(\xi))+B^{*}(D_{0}(\xi))}_{L^{2}}.
  \end{equation*}
  Moreover, $\xi\mapsto D_{0}^{*}(B(\xi))+B^{*}(D_{0}(\xi))$ is a \emph{zeroth} order operator (which is translation invariant in the ends). Similarly, if $\eta\in W^{1,2}(\Lambda^{1,0}\otimes E,F^{*})$ then
  \begin{equation*}
    \norm{D^{\ast}_{0}\eta}^{2}_{L^{2}}+\sigma^{2}\norm{B^{*}\eta}_{L^{2}}^{2}\le \norm{D^{\sigma,\ast}\eta}_{L^{2}}^{2}+\sigma \norm{\xi}_{L^{2}}\norm{D_{0}(B^{*}(\eta))+B(D^{*}_{0}(\eta))}_{L^{2}}.
  \end{equation*}
  and $\eta\mapsto D_{0}(B^{*}(\eta))+B(D_{0}^{*}(\eta))$ is also a zeroth order operator (also translation invariant in the ends). We therefore conclude a constant $C=C(D_{0},B)$ so that for all $\xi,\eta$ as above we have
  \begin{equation}\label{eq:bochner-weitzenboch}
    \begin{aligned}
      \norm{B\xi}^{2}_{L^{2}}&\le \sigma^{-2}\norm{D^{\sigma}\xi}^{2}_{L^{2}}+C\sigma^{-1}\norm{\xi}^{2}_{L^{2}},\\
      \norm{B^{*}(\eta)}^{2}_{L^{2}}&\le \sigma^{-2}\norm{D^{\sigma,\ast}\eta}^{2}_{L^{2}}+C\sigma^{-1}\norm{\eta}^{2}_{L^{2}}.
    \end{aligned}
  \end{equation}
  In particular, if $D^{\sigma_{n}}\xi_{n}$ and $\xi_{n}$ remain bounded in $L^{2}$ and $\sigma_{n}\to\infty$, then $B\xi_{n}$ must converge to zero in $L^{2}$. This forces the mass of $\xi_{n}$ to concentrate near the zeros of $B$.
\end{lemma}
\begin{proof}
  Thanks to Proposition \ref{prop:smooth-approximation} and the subsequent remarks, it suffices to consider the case when $\xi$ is smooth and takes values in $F$ along the boundary. 

  Let $\ip{-,-}$ denote the $L^{2}$ inner product. Naively, the estimate is proved by the following computation
  \begin{equation}\label{eq:comput-naiv}
    \begin{aligned}
      \norm{D^{\sigma}\xi}^{2}&=\ip{\xi,D^{\sigma,\ast}D^{\sigma}\xi}\\
      &=\ip{\xi,D^{\ast}_{0}D_{0}\xi}+\sigma \ip{\xi,D_{0}^{\ast}(B(\xi))+B^{*}(D_{0}(\xi))}+\sigma^{2}\norm{B(\xi)}^{2}\\
      &=\norm{D_{0}\xi}^{2}+\sigma \ip{\xi,D^{\ast}_{0}(B(\xi))+B^{*}(D_{0}(\xi))}+\sigma^{2}\norm{B(\xi)}^{2}.
    \end{aligned}
  \end{equation}
  Rearranging easily yields the desired result. Unfortunately, we cannot expect to be able to apply the formal adjoint property in the first and third equality unless $D^{\sigma}\xi$ and $D_{0}\xi$ take boundary values in $F^{*}$. One way to circumvent this issue would to be assume that $D_{0}\xi$ takes boundary values in $F^{*}$. The lower order term $B$ has been constructed so that $D^{\sigma}\xi$ would automatically also take boundary values in $F^{*}$. It seems plausible that smooth sections $\xi$ which take boundary values in $F$ \emph{and} for which $D^{0}\xi$ takes boundary values in $F^{*}$ are dense in $W^{1,2}(E,F)$.\footnote{This is suggested by the following observation: Locally write $\xi=uY$. The approximation result Proposition \ref{prop:smooth-approximation} shows that $u$ can be approximated in $W^{1,2}$ by $u_{n}=\Phi_{n}\ast E(u)$. By picking $\Phi_{n}$ appropriately, these approximations satisfy $(\bd_{s}+i\bd_{t})u_{n}\in \R$. In general we would require that we can approximate $u$ by smooth functions $u_{n}$ taking real values on the boundary and also satisfying $\db u_{n}+\alpha u_{n}+\beta \cl{u_{n}}\in \R$, where $\alpha,\beta$ are arbitrary complex valued functions.} However, we will not pursue this density approach here. Rather, we prefer to make the observation that we have applied the formal adjoint property \emph{twice}, once for $D^{\sigma}$ and once for $D_{0}$, and the failures of formal adjointness will cancel each other out.

  Indeed, $D^{\sigma}-D_{0}$ is a zeroth order operator whose formal adjoint is $D^{\sigma,\ast}-D_{0}^{*}$. Formal adjoints for zeroth order operators do not require any integration by parts, hence
  \begin{equation*}
    \ip{D^{\sigma}\xi -D_{0}\xi,D^{\sigma}\xi+D_{0}\xi}=\ip{\xi,(D^{\sigma,\ast}-D_{0}^{*})(D^{\sigma}+D_{0})\xi}.
  \end{equation*}
  This simplifies to
  \begin{equation*}   \norm{D^{\sigma}\xi}^{2}-\norm{D_{0}\xi}^{2}=\ip{\xi,D^{\sigma,\ast}D^{\sigma}\xi}-\ip{\xi,D^{0,\ast}D^{0}\xi}+\ip{\xi,(D_{0}^{*}D^{\sigma}-D^{\sigma,\ast}D_{0})\xi}.
  \end{equation*}
  Clearly $D_{0}^{*}D^{\sigma}-D^{\sigma,\ast}D_{0}=\sigma(D_{0}^{\ast}B-B^{*}D_{0})$, and hence
  \begin{equation*}
    \begin{aligned}
      \ip{\xi,(D_{0}^{*}D^{\sigma}-D^{\sigma,\ast}D_{0})\xi}&=\ip{\xi,D_{0}^{*}\sigma B\xi}-\ip{\xi,\sigma B^{*}D_{0}\xi}\\
      &=\ip{D_{0}\xi,\sigma B\xi}-\ip{\sigma B\xi,D_{0}\xi}=0,        
    \end{aligned}
  \end{equation*}
  where we have used the fact that $B\xi$ takes values in $F^{*}$ (which follows easily from our construction of $B$ and the fact $\xi$ takes values in $F$). Thus  
  \begin{equation*}
    \norm{D^{\sigma}\xi}^{2}-\ip{\xi,D^{\sigma,\ast}D^{\sigma}\xi}=\norm{D_{0}\xi}^{2}-\ip{\xi,D^{0,\ast}D^{0}\xi}.
  \end{equation*}
  This implies that the conclusion of \eqref{eq:comput-naiv} holds, (even if the individual steps do not hold). The first estimate from the statement of the Lemma then follows easily. The second estimate is proved in the same way.

  To show that $L(\xi)=B^{*}D_{0}(\xi)+D_{0}^{*}B(\xi)$ is a zeroth order operator, we will show that $L(f\xi)=fL(\xi)$ for all real-valued functions $f$ and all sections $\xi$ (this implies that $L$ is described as a tensor). It suffices to prove this in the case when $f$ is supported in a coordinate chart $z=s+it$ with frame $Y$.
  
  We digress for a moment to derive formulas for $D_{0}^{*}$ and $B^{*}$ on this coordinate chart. Write $\xi=u Y$ and $B=\varphi(\d s-i\d t)\mu(-,Y)Y$. We can assume that $Y$ is a unitary frame, i.e.\ $\abs{Y}=1$, but we do not assume that $\abs{\bd_{s}}=1$. 

  Let $\eta$ be an arbitrary smooth section of $\Lambda^{1,0}\otimes E$ taking values in $F^{*}$ along the boundary. Write $\eta=w(\d s - i\d t)\otimes Y$. Then we easily compute (similarly to how we argued in Section \ref{sec:formal-adjoints}):
  \begin{equation}\label{eq:fleo}
    B(\xi)=\cl{u}\varphi(\d s-i\d t)\otimes Y\implies \mathrm{Re}\,\mu(B(\xi),\eta)=\mathrm{Re}\,u\cl{\varphi}w\abs{\d s-i\d t}^{2}.
  \end{equation}
  Therefore we must have $B^{*}(\eta)=\varphi\abs{\d s-i\d t}^{2}\cl{w}Y$, since this choice yields the desired pointwise relationship:
  \begin{equation*}
    \mathrm{Re}\,\mu(\xi,B^{*}(\eta))=\mathrm{Re}\,\cl{u}\cl{w}\varphi\abs{\d s-i\d t}^{2}=\mathrm{Re}\,\mu(B(\xi),\eta).
  \end{equation*}
  Now in \eqref{eq:formal-formula} we have computed a formula for $D_{0}^{*}$:
  \begin{equation*}
    \abs{\d s-i\d t}^{-2}D_{0}^{*}(w(\d s-i\d t)\otimes Y)=(-\bd_{s}w+i\bd_{t}w+Sw)Y,
  \end{equation*}
  for some matrix valued function $S$. The important part is that the leading order part is $-\bd=-\bd_{s}+i\bd_{t}$. We then combine \eqref{eq:fleo} with the above formula for $D_{0}^{*}$ to obtain
  \begin{equation}\label{eq:gleo}
    D_{0}^{*}(B(f\xi))=-\bd f \cdot \varphi \abs{\d s-i\d t}^{2}\cl{u}Y+fD_{0}^{*}(B(\xi)).
  \end{equation}
  This computes half of $L(f\xi)$. For the other half, we use the defining property of Cauchy-Riemann operators to conclude
  \begin{equation}\label{eq:hleo}
    B^{*}(D_{0}(f\xi))=B^{*}(\db f \cdot (\d s-i\d t)\otimes \xi)+fB^{*}(D_{0}(\xi)),
  \end{equation}
  where $\db=\bd_{s}+i\bd_{t}$. Recall that we assume $f$ is real-valued. Then our formula for $B^{*}(\eta)$ with $\eta=\db f \cdot (\d s-i\d t)\otimes \xi=\db f\cdot u\cdot (\d s-i\d t)\otimes Y$ implies
  \begin{equation*}
    B^{*}(\db f\cdot (\d s-\d t)\otimes \xi)=\cl{\db f}\varphi\abs{\d s-i\d t}^{2} \cl{u}Y=\bd f \cdot \varphi \abs{\d s-i\d t}^{2}\cl{u} Y.
  \end{equation*}
  Adding together \eqref{eq:gleo} and \eqref{eq:hleo}, the $\pm \bd f\cdot \varphi \abs{\d s-i\d t}^{2}\cl{u}Y$ terms cancel and we obtain
  \begin{equation*}
    L(f\xi)=D_{0}^{*}(B(f\xi))+B^{*}(D_{0}(f\xi))=f[D_{0}^{*}(B(\xi))+B^{*}(D_{0}(\xi))]=fL(\xi),
  \end{equation*}
  as desired. A similar argument shows that $BD_{0}^{*}+D_{0}B^{*}$ is also a zeroth order operator. This completes the proof.
\end{proof}

\subsubsection{Local Bochner-Weitzenb\"ock estimates for sections supported near the zeros}
\label{sec:local-models}
In this section we will do a case-by-case analysis of the operator $D^{\sigma}$ near the zeros. See \cite[Section 5.6]{wendl-sft} for similar results. To simplify the calculations ahead, let's write $\db=\bd_{s}+i\bd_{t}$. In the next section we will explain how to rescale $D^{\sigma}=\db\pm \sigma\alpha(z)C$ to $D^{1}=\db\pm \alpha(z)C$. In this section we will focus only on the rescaled operator $D^{1}$.

There are four possibilities for $D^{1}$, namely $\db\pm zC$ and $\db \pm \cl{z}C$. We have the following estimates for these operators:
\begin{lemma}[Local Bochner-Weitzenb\"ock]\label{lemma:baby-bochner}
  Let $v\in W^{1,2}(\C,\C)$ or $v\in W^{1,2}(\cl{\mathbb{H}},\C,\R)$. Then we have the following estimates:
  \begin{equation*}
    \begin{aligned}
      \norm{\db v}^{2}_{L^{2}}+\norm{zv}^{2}_{L^{2}}&\le \norm{\db v\pm z\cl{v}}_{L^{2}}^{2}+2\norm{v}^{2}_{L^{2}}\\
      \norm{\db v}^{2}_{L^{2}}+\norm{zv}^{2}_{L^{2}}&\le \norm{\db v\pm \cl{z}\cl{v}}_{L^{2}}^{2}.
    \end{aligned}
  \end{equation*}  
\end{lemma}
\begin{proof}
  Using the smooth approximation result Proposition \ref{prop:smooth-approximation}, we may suppose that $v$ is smooth, compactly supported, and takes real values along the boundary.
  
  To prove the inequalities, we will need to integrate by parts two times. Let us focus on the first estimate. We start by computing:
  \begin{equation*}
    (-\bd\pm zC)(\db\pm zC)v=-\bd\db v\pm z\cl{\db v}-\pm z\bd \cl{v}-\pm 2\cl{v}+\abs{z}^{2}v.
  \end{equation*}
  Using the fact that $\cl{\db v}=\bd \cl{v}$ we conclude that two terms cancel and we are left with
  \begin{equation}\label{eq:naive-babu}
    (-\bd\pm zC)(\db\pm zC)v=-\bd\db v \mp 2\cl{v}+\abs{z}^{2}v.
  \end{equation}
  The naive idea is to multiply both sides by $\mathrm{Re}\,\mu_{0}(v,-)$, integrate, and use the formal adjoint property for $-\bd\pm zC=(\db\pm zC)^{*}$ and $-\bd=\db^{*}$. This naive argument would require that $\db v$ and $\db v\pm z\cl{v}$ take real values along the boundary, which we do not assume. However, as in the previous section, the fact that we integrate by parts twice will imply that the failures of formal adjointness will cancel out.

  Indeed, we compute
  \begin{equation*}
    \mathrm{Re}\int \mu_{0}(v,-\bd\db v)\d s \d t=\mathrm{Re}\int\mu_{0}(v,-\bd_{s}\db v)\d s\d t+\mathrm{Re}\int \mu_{0}(v,i\bd_{t}\db v) \d s\d t.
  \end{equation*}
  It is clear that the can integrate by parts with respect to $\bd_{s}$, and conclude
  \begin{equation*}
    \mathrm{Re}\int \mu_{0}(v,-\bd\db v)\d s \d t=\mathrm{Re}\int\mu_{0}(\bd_{s}v,\db v)\d s\d t+\mathrm{Re}\int \mu_{0}(v,i\bd_{t}\db v) \d s\d t.
  \end{equation*}
  We can also integrate by parts with respect to $\bd_{t}$, at the expense of a boundary integral term, and (after some simplification) end up with:
  \begin{equation*}
    \mathrm{Re}\int \mu_{0}(v,-\bd\db v)\d s \d t=\mathrm{Re}\int\mu_{0}(\db v,\db v)\d s\d t-\mathrm{Re}\int_{\R}\mu_{0}(v,i\db v) \d s\d t.
  \end{equation*}
  We do the same computation with $\db$ replaced by $D=\db\pm zC$, and conclude that
  \begin{equation*}
    \mathrm{Re}\int \mu_{0}(v,D^{*}Dv)\d s \d t=\mathrm{Re}\int\mu_{0}(D v,D v)\d s\d t-\mathrm{Re}\int_{\R}\mu_{0}(v,iDv)\d s\d t.
  \end{equation*}
  Finally, we observe that
  \begin{equation*}
    \begin{aligned}
      \mathrm{Re}\int_{\R}\mu_{0}(v,iDv)\d s\d t&=\mathrm{Re}\int_{\R}\mu_{0}(v,i\db v)\d s\d t\pm \mathrm{Re}\int_{\R}\mu_{0}(v,iCz\cl{v})\d s\d t\\ &=\mathrm{Re}\int_{\R}\mu_{0}(v,i\db v)\d s\d t,
    \end{aligned}
  \end{equation*}
  where we have used the fact that $Cz\cl{v}$ takes real values along the boundary. Therefore
  \begin{equation*}
    \mathrm{Re}\int \mu_{0}(v,D^{*}Dv)\d s \d t-\norm{Dv}_{L^{2}}^{2}=\mathrm{Re}\int \mu_{0}(v,\db^{*} \db v)\d s \d t-\norm{\db v}_{L^{2}}^{2}.
  \end{equation*}
  Applying $\mathrm{Re}\,\mu_{0}(v,-)$ to \eqref{eq:naive-babu} and integrating implies
  \begin{equation*}
    \norm{\db v\pm z\cl{v}}_{L^{2}}^{2}=\norm{\db v}^{2}_{L^{2}}\mp 2\mathrm{Re}\int \mu_{0}(v,\cl{v})+\norm{zv}_{L^{2}}^{2}.
  \end{equation*}
  We rearrange and estimate to conclude that
  \begin{equation*}
    \norm{\db v}^{2}_{L^{2}}+\norm{zv}_{L^{2}}^{2} \le \norm{\db v\pm z\cl{v}}_{L^{2}}^{2}+ 2\norm{v}_{L^{2}}^{2},
  \end{equation*}
  as desired. The second estimate in the statement of the lemma (with $D=\db\pm \cl{z}C$) is proved in the same manner.
\end{proof}

\subsection{Classifying the kernels of $D^{1}$ (six cases).}
\label{sec:D1-kernel-class}
The second estimate in Lemma \ref{lemma:baby-bochner} implies that $\db v\pm \cl{z}\cl{v}=0$ has no non-zero solutions -- this takes care of three of the six kinds of operators.

Our next lemma shows that $\db v\pm z\cl{v}=0$ has either a one-dimensional space of solutions or a zero-dimensional space of solutions.
\begin{lemma}\label{lemma:kernel-class-1}
  Suppose that $v:\C\to \C$ is in $L^{2}$, then
  \begin{equation*}
    \begin{aligned}
      \db v-z\cl{v}=0&\iff v=ci\exp(-\frac{1}{2}\abs{z}^{2})\text{ for $c\in \R$},\\
      \db v+z\cl{v}=0&\iff v=c\exp(-\frac{1}{2}\abs{z}^{2})\text{ for $c\in \R$}.
    \end{aligned}
  \end{equation*}
  On the other hand if $v:\cl{\mathbb{H}}\to \C$ is in $L^{2}$ and takes real values along the boundary, then
  \begin{equation*}
    \begin{aligned}
      \db v-z\cl{v}=0&\iff v=0\\
      \db v+z\cl{v}=0&\iff v=c\exp(-\frac{1}{2}\abs{z}^{2})\text{ for $c\in \R$}.
    \end{aligned}    
  \end{equation*}
\end{lemma}
Morally, this says that positive interior zeros and $(+,+)$ zeros contribute one dimension to the kernel, but $(+,-)$ zeros do not contribute to the kernel.
\begin{proof}
  Observe that if we set $v^{\prime}=iv$, then
  \begin{equation*}
    \bd_{s}v^{\prime}+i\bd_{t}v^{\prime}-z\cl{v^{\prime}}=i(\bd_{s}v+i\bd_{t}v+z\cl{v}),
  \end{equation*}
  and hence it suffices to study the equation $\db v-z\cl{v}=0$. Following \cite[Proposition 5.22]{wendl-sft}, we prove that the real part of $v$ must vanish identically.


  The second estimate from Lemma \ref{lemma:baby-bochner} implies that $\db{v}\in L^{2}$ and $zv\in L^{2}$ (proof: both $\norm{\db(\rho(z\delta)v)}_{L^{2}}$ and $\norm{z\rho(z\delta)v}_{L^{2}}$ remain bounded as $\delta\to 0$). The $L^{2}$ elliptic estimates then imply that $v\in W^{1,2}$.

  Let $y=\mathrm{Re}(v)$. Since $-\Delta v+2\cl{v}+\abs{z}^{2}v=0$, we have
  \begin{equation*}
    0=-\Delta y+(2+\abs{z}^{2})y.
  \end{equation*}
  Apply $\mathrm{Re}\,\mu_{0}(\rho(z\delta)y,-)$ to both sides, and integrate by parts to conclude
  \begin{equation*}
    0=\int \rho(z\delta)\abs{\db y}^{2}\d s\d t+\mathrm{Re}\int \mu_{0}(\db (\rho(z\delta))\cdot y,\db y)\d s\d t+\int \rho(z\delta)(2+\abs{z}^{2})\abs{y}^{2}\d s\d t.
  \end{equation*}
  When we integrate by parts, we use $\bd_{t}y=0$ (which holds in our case). We can now take the limit $\delta\to 0$, since we have verified that $\db y, zy\in L^{2}$, and conclude that
  \begin{equation*}
    0=\norm{\db
      y}_{L^{2}}^{2}+\norm{(2+\abs{z}^{2})^{1/2}y}_{L^{2}}^{2}\implies y=0,
  \end{equation*}
  using $\db(\rho(z\delta))=O(\delta)$. It follows that any $L^{2}$ solution of $\db v-z\cl{v}=0$ as in the statement of the lemma is purely imaginary (and hence vanishes along the boundary, if the boundary exists). 

  We now observe that $v=i\exp(-\frac{1}{2}\abs{z}^{2})$ is certainly in $L^{2}$ and solves $\db v-z\cl{v}=0$. Clearly any other solution $v^{\prime}$ will be $v^{\prime}=gv$ for some \emph{holomorphic} $g$; moreover, by what we have shown above, $g$ must be real. There are no non-constant holomorphic functions defined on $\mathbb{C}$ or $\mathbb{H}$ which take only real values (the rank of the derivative matrix would be always $0$). Thus $g=c$ must be a real number. 

  In the case when $v$ is defined on $\cl{\mathbb{H}}$, the only possibility is $c=0$, since otherwise $v$ would take non-zero imaginary values along the boundary.

  Finally, we return to the second equation from the statement, $\bd_{s}v+i\bd_{t}v+z\cl{v}=0$. We have shown that this solution is conjugate to the first equation under multiplication by $i$. Therefore all solutions on the disk or half-plane are given by $v=c\exp(-\frac{1}{2}\abs{z}^{2})$ for some real $c$. In this case we \emph{can} have non-zero $c$ when $v$ is defined on $\cl{\mathbb{H}}$.

  This completes the proof.
\end{proof}

\subsubsection{The formal adjoint near the zeros}
\label{sec:local-models-2}
Since we chose our metric so that $\abs{\bd_{s}}=1$ in all of the special coordinate charts centered at the zeros of $B$, we can easily compute the coordinate representations of $D^{\sigma,\ast}$:
\begin{equation*}
  \begin{aligned}
    D^{\sigma}=\db\pm \sigma zC&\implies D^{\sigma,*}(u)=-\bd \pm \sigma zC\\
    D^{\sigma}=\db\pm \sigma\cl{z} C&\implies D^{\sigma,*}(u)=-\bd \pm \sigma\cl{z}C.
  \end{aligned}
\end{equation*}
Now let $D^{\sigma,\dagger}=-C\circ D^{\sigma,\ast}\circ C$. The above yields:
\begin{equation*}
  \begin{aligned}
    D^{\sigma}=\db\pm \sigma zC&\implies D^{\sigma,\dagger}(u)=\db \mp \sigma\cl{z}C\\
    D^{\sigma}=\db\pm \sigma\cl{z}C&\implies D^{\sigma,\dagger}(u)=\db \mp \sigma zC.
  \end{aligned}
\end{equation*}
Thus we can think of $D^{\sigma}\mapsto D^{\sigma,\dagger}$ as defining a ``duality involution'' on the set of six local model equations. This is illustrated in Figure \ref{fig:dual-yt}.

To explain the labeling scheme used in the figure, we partition the zero set of $B$, denoted $\mathrm{Z}$, into six kinds of zeros:
\begin{equation*}
  \mathrm{Z}=\mathrm{Z}^{+}\cup \mathrm{Z}^{-}\cup \mathrm{Z}^{++}\cup \mathrm{Z}^{+-}\cup \mathrm{Z}^{-+}\cup \mathrm{Z}^{--},
\end{equation*}
where $\mathrm{Z}^{\pm}$ are interior positive/negative zeros, and $\mathrm{Z}^{\pm\pm}$ are boundary zeros (let's agree for this notation that the two $\pm$ signs are independent). The convention for assigning labels is via the linearization: the first sign is the linearization of $B$ allowing arbitrary deformations, and the second sign is for the linearization only allowing deformations along the boundary. The local form of $D^{\sigma}$ near a zero $\zeta$ and the corresponding count is summarized in Figure \ref{fig:dual-yt}.

It follows from the construction in Section \ref{sec:defining-d-sigma} that the sum of the counts of all the zeros in $\mathrm{Z}$ is equal to $\mathrm{X}+\mu^{\tau}_{\mathrm{Mas}}$.

\begin{figure}[H]
  \centering
  \begin{tikzpicture}[scale=.6]
    \draw (-2.5,-8.5)rectangle+(5,12);
    \draw[shift={(5.2,0)}] (-2.5,-8.5)rectangle+(5,12);
    \draw[shift={(10.4,0)}] (-2.5,-8.5)rectangle+(10,12);
    
    \begin{scope}[shift={(0,0)}]
      \draw[black!50!white,fill=black!5!white] (2,0) arc (0:180:2);
      \node (A) at (0,2) [above] {$\db+\sigma zC$};
      \draw[postaction={decorate,decoration={markings,mark=at position 0.5 with {\arrow[scale=1.5]{<};}}}] (-2,0)--(0,0);
      \draw[postaction={decorate,decoration={markings,mark=at position 0.5 with {\arrow[scale=1.5]{<};}}}] (2,0)--(0,0);
      \draw[postaction={decorate,decoration={markings,mark=at position 0.5 with {\arrow[scale=1.5]{<};}}}] (0,2)--(0,0);
      \draw[postaction={decorate,decoration={markings,mark=at position 0.5 with {\arrow[scale=1.5]{<};}}}] (45:2)--(0,0);
      \draw[postaction={decorate,decoration={markings,mark=at position 0.5 with {\arrow[scale=1.5]{<};}}}] (135:2)--(0,0);
      \node[draw,circle,inner sep=1pt,fill] at (0,0){};
      \node (B) at (0,-0.2)[below]{$\mathrm{count}=+1$};
    \node at (B.south)[below] {$\mathrm{Z}^{++}$};
\end{scope}
    \begin{scope}[shift={(5.2,0)}]
      \node (A) at (0,2) [above] {$\db-\sigma zC$};
      \draw[black!50!white,fill=black!5!white] (2,0) arc (0:180:2);
      \draw[postaction={decorate,decoration={markings,mark=at position 0.5 with {\arrow[scale=1.5]{>};}}}] (-2,0)--(0,0);
      \draw[postaction={decorate,decoration={markings,mark=at position 0.5 with {\arrow[scale=1.5]{>};}}}] (2,0)--(0,0);
      \draw[postaction={decorate,decoration={markings,mark=at position 0.5 with {\arrow[scale=1.5]{>};}}}] (0,2)--(0,0);
      \draw[postaction={decorate,decoration={markings,mark=at position 0.5 with {\arrow[scale=1.5]{>};}}}] (45:2)--(0,0);
      \draw[postaction={decorate,decoration={markings,mark=at position 0.5 with {\arrow[scale=1.5]{>};}}}] (135:2)--(0,0);
      \node[draw,circle,inner sep=1pt,fill] at (0,0){};
      \node (B) at (0,-0.2)[below]{$\mathrm{count}=0$};
    \node at (B.south)[below] {$\mathrm{Z}^{+-}$};
  \end{scope}
    \begin{scope}[shift={(10.4,0)}]
      \node (A) at (3,0.5) [right] {$\db-\sigma zC$};
      \node (B) at (A.south) [below] {$\mathrm{count}=+1$};
      \node at (B.south)[below]{$\mathrm{Z}^{+}$};
     
     \draw[black!50!white,fill=black!5!white] (0,0) circle (2);
     \draw[postaction={decorate,decoration={markings,mark=at position 0.5 with {\arrow[scale=1.5]{>};}}}] (-2,0)--(0,0);
     \draw[postaction={decorate,decoration={markings,mark=at position 0.5 with {\arrow[scale=1.5]{>};}}}] (2,0)--(0,0);
     \draw[postaction={decorate,decoration={markings,mark=at position 0.5 with {\arrow[scale=1.5]{>};}}}] (0,2)--(0,0);
     \draw[postaction={decorate,decoration={markings,mark=at position 0.5 with {\arrow[scale=1.5]{>};}}}] (0,-2)--(0,0);
     \draw[postaction={decorate,decoration={markings,mark=at position 0.5 with {\arrow[scale=1.5]{>};}}}] (45:2)--(0,0);
     \draw[postaction={decorate,decoration={markings,mark=at position 0.5 with {\arrow[scale=1.5]{>};}}}] (135:2)--(0,0);
     \draw[postaction={decorate,decoration={markings,mark=at position 0.5 with {\arrow[scale=1.5]{>};}}}] (225:2)--(0,0);
     \draw[postaction={decorate,decoration={markings,mark=at position 0.5 with {\arrow[scale=1.5]{>};}}}] (315:2)--(0,0);
     \node[draw,circle,inner sep=1pt,fill] at (0,0){};
   \end{scope}
    \begin{scope}[shift={(0,-6)}]
      \node (A) at (0,2) [above] {$\db-\sigma\cl{z}C$};
      \draw[black!50!white,fill=black!5!white] (2,0) arc (0:180:2);
      \begin{scope}
        \clip (2,0) arc (0:180:2)--cycle;
        \draw[postaction={decorate,decoration={markings,mark=at position 0.5 with {\arrow[scale=1.5]{<};}}}] plot[domain=0.5:2] ({\x},{1/\x});
        \draw[postaction={decorate,decoration={markings,mark=at position 0.5 with {\arrow[scale=1.5]{<};}}}] plot[domain=0.5:2] ({-\x},{1/\x});
      \end{scope}
      \draw[postaction={decorate,decoration={markings,mark=at position 0.5 with {\arrow[scale=1.5]{>};}}}] (-2,0)--(0,0);
      \draw[postaction={decorate,decoration={markings,mark=at position 0.5 with {\arrow[scale=1.5]{>};}}}] (2,0)--(0,0);
      \draw[postaction={decorate,decoration={markings,mark=at position 0.5 with {\arrow[scale=1.5]{<};}}}] (0,2)--(0,0);
      \node[draw,circle,inner sep=1pt,fill] at (0,0){};
      \node (B) at (0,-0.2)[below]{$\mathrm{count}=-1$};
    \node at (B.south)[below] {$\mathrm{Z}^{--}$};
\end{scope}
    \begin{scope}[shift={(5.2,-6)}]
      \node (A) at (0,2) [above] {$\db+\sigma\cl{z}C$};
      \draw[black!50!white,fill=black!5!white] (2,0) arc (0:180:2);
      \begin{scope}
        \clip (2,0) arc (0:180:2)--cycle;
        \draw[postaction={decorate,decoration={markings,mark=at position 0.5 with {\arrow[scale=1.5]{>};}}}] plot[domain=0.5:2] ({\x},{1/\x});
        \draw[postaction={decorate,decoration={markings,mark=at position 0.5 with {\arrow[scale=1.5]{>};}}}] plot[domain=0.5:2] ({-\x},{1/\x});
      \end{scope}
      \draw[postaction={decorate,decoration={markings,mark=at position 0.5 with {\arrow[scale=1.5]{<};}}}] (-2,0)--(0,0);
      \draw[postaction={decorate,decoration={markings,mark=at position 0.5 with {\arrow[scale=1.5]{<};}}}] (2,0)--(0,0);
      \draw[postaction={decorate,decoration={markings,mark=at position 0.5 with {\arrow[scale=1.5]{>};}}}] (0,2)--(0,0);
      \node[draw,circle,inner sep=1pt,fill] at (0,0){};
      \node (B) at (0,-0.2)[below]{$\mathrm{count}=0$};
    \node at (B.south)[below] {$\mathrm{Z}^{-+}$};
\end{scope}
    \begin{scope}[shift={(10.4,-6)}]
      \node (A) at (3,0.5) [right] {$\db+\sigma\cl{z}C$};
      \node (B) at (A.south) [below] {$\mathrm{count}=-1$};
      \node at (B.south)[below]{$\mathrm{Z}^{-}$};
      \draw[black!50!white,fill=black!5!white] (0,0) circle (2);
      \begin{scope}
        \clip (0,0) circle (2);
        \draw[postaction={decorate,decoration={markings,mark=at position 0.5 with {\arrow[scale=1.5]{>};}}}] plot[domain=0.5:2] ({\x},{1/\x});
        \draw[postaction={decorate,decoration={markings,mark=at position 0.5 with {\arrow[scale=1.5]{>};}}}] plot[domain=0.5:2] ({\x},{-1/\x});
        \draw[postaction={decorate,decoration={markings,mark=at position 0.5 with {\arrow[scale=1.5]{>};}}}] plot[domain=0.5:2] ({-\x},{1/\x});
        \draw[postaction={decorate,decoration={markings,mark=at position 0.5 with {\arrow[scale=1.5]{>};}}}] plot[domain=0.5:2] ({-\x},{-1/\x});            
      \end{scope}
      \draw[postaction={decorate,decoration={markings,mark=at position 0.5 with {\arrow[scale=1.5]{<};}}}] (-2,0)--(0,0);
      \draw[postaction={decorate,decoration={markings,mark=at position 0.5 with {\arrow[scale=1.5]{<};}}}] (2,0)--(0,0);
      \draw[postaction={decorate,decoration={markings,mark=at position 0.5 with {\arrow[scale=1.5]{>};}}}] (0,2)--(0,0);
      \draw[postaction={decorate,decoration={markings,mark=at position 0.5 with {\arrow[scale=1.5]{>};}}}] (0,-2)--(0,0);
      \node[draw,circle,inner sep=1pt,fill] at (0,0){};
    \end{scope}
  \end{tikzpicture}
  \caption{The six kinds of zeros and the coordinate representation of $D^{\sigma}$ in each chart. Two zeros are in the same box if the operators are dual in the sense defined above.}
  \label{fig:dual-yt}
\end{figure}
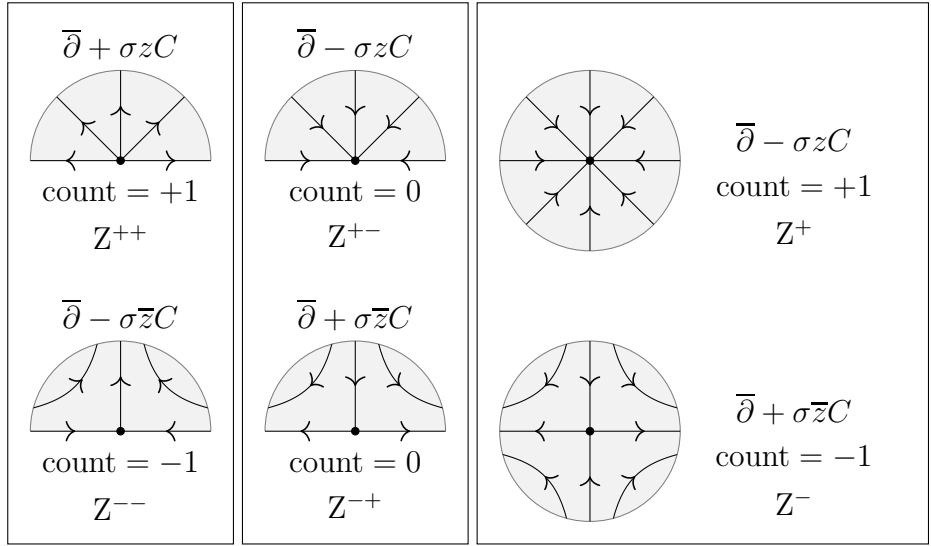

Applying Lemmas \ref{lemma:baby-bochner} and \ref{lemma:kernel-class-1} to $D^{1,\dagger}$ yields the following result for $D^{1,\ast}$.
\begin{cor}\label{cor:kernel-class-dual}
  Suppose $v:\C\to \C$ is in $L^{2}$. Then
  \begin{equation*}
    \begin{aligned}
      -\bd v-z\cl{v}=0&\iff v=0\\
      -\bd v+\cl{z}\cl{v}=0&\iff v=ic\exp(-\frac{1}{2}\abs{z}^{2})\text{ for some $c\in \R$}.
    \end{aligned}
  \end{equation*}
  Now suppose that $v:\cl{\mathbb{H}}\to \C$ is in $L^{2}$ and takes real values along the boundary. Then
  \begin{equation*}
    \begin{aligned}
      -\bd v\pm z\cl{v}=0&\iff v=0\\
      -\bd v+\cl{z}\cl{v}=0&\iff v=0\\
      -\bd v-\cl{z}\cl{v}=0&\iff v=c\exp(-\frac{1}{2}\abs{z}^{2})\text{ for some $c\in \R$}.\\
    \end{aligned}
  \end{equation*}
\end{cor}
Heuristically, this says that the zeros with count $-1$ in Figure \ref{fig:dual-yt} contribute a one-dimensional subspace to the kernel of the formal adjoint $D^{\sigma,\ast}$ (and all other zeros contribute nothing). 

\subsection{Linear compactness and a stabilization of $D^\rho$}
\label{sec:linear-compactness-and-stab}
In this section we will relate the kernel and cokernel of $D^{\rho}$ to the kernels and cokernels of the local models $D^{1}$. We begin with an explanation of the rescaling scheme we use.

\subsubsection{Modified rescaling maps}
\label{sec:modified-rescaling-maps}

Suppose that $\zeta$ is a zero and let $z$ be the special coordinate chart centered at $\zeta$. By convention, $z$ is either $\cl{\mathbb{H}}\cap D(1)$ or $D(1)$ valued. Let $\rho$ be a bump function supported in $D(1)$ which is $1$ on $D(1/2)$. 

Let $\Phi_{\sigma}:L^{2}(\C,\C)\to L^{2}(\dot\Sigma,\C)$ be the modified rescaling map:
\begin{equation*}
  \Phi_{\sigma}(v)=\rho\cdot\sigma^{1/2}v(\sigma^{1/2}z).
\end{equation*}
Observe that $\norm{\Phi_{\sigma}(v)}_{L^{2}}\le \norm{v}_{L^{2}}=\lim_{\sigma\to\infty}\norm{\Phi_{\sigma}(v)}_{L^{2}}$. Dually, we let $\Pi_{\sigma}=\Phi_{\sigma}^{\ast}$ be the adjoint. It is easy to obtain the following explicit formula for $\Pi_{\sigma}$:
\begin{equation*}
  \Pi_{\sigma}(u)(z)=\sigma^{-1/2}\rho(\sigma^{-1/2}z)u(\sigma^{-1/2}z).
\end{equation*}

\begin{figure}[H]
  \centering
  \begin{tikzpicture}
    \fill[black!10!white] (0.5,1) circle (0.2);
    \draw[fill=black!10!white] (0,-0.4) arc (-90:90:0.2);
    \draw (0,-2)--+(0,4) (1,-2)--+(0,4) (0.5,1) circle (0.2) (4,0) circle (2);
    \path (4,0)--+(100:2)coordinate(B) (4,0)--+(220:2)coordinate(C);
    \draw[dashed] (0.5,1) +(100:0.2)--(B) (0.5,1) +(240:0.2)--(C);

    \draw (-4,-2) arc (-90:90:2)--cycle;
    \draw[dashed] (-4,0)+(60:2)--(0,0) (-4,0)+(-60:2)--(0,-.4);

    \node[draw,fill=black,inner sep=1pt,circle] (B) at (4,0) {};
    \node at (B.east)[right]{$\zeta\in \mathrm{Z}^{\pm}$};

    \node[draw,fill=black,inner sep=1pt,circle] (B) at (-4,0) {};
    \node at (B.west)[left]{$\zeta\in \mathrm{Z}^{\pm\pm}$};

  \end{tikzpicture}
  \caption{Rescaling sections near the zeros of $B$. The map $\Phi_{\sigma}$ takes a section on the large domain and compresses it to fit inside the small domain (and then cuts it off by $\rho$). The map $\Pi_{\sigma}$ does the opposite, it first cuts off by $\rho$ and then expands the domain of the section. The factors have been chosen so that $\norm{\Phi_{\sigma}(v)}_{L^{2}}=\norm{\rho(\sigma^{-1/2}z)v}_{L^{2}}$.}
  \label{fig:rescaling-near-zeros}
\end{figure}
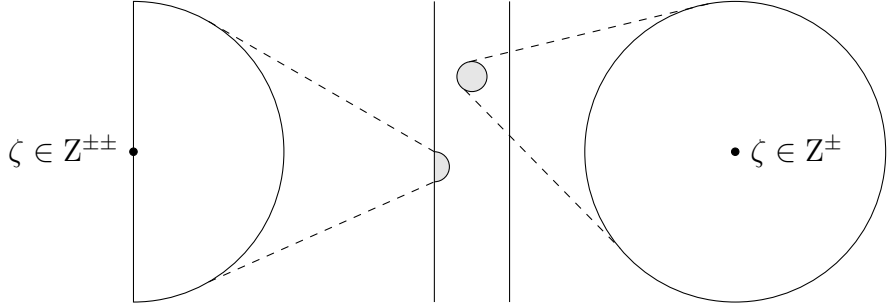

The relevance of $\Pi_{\sigma},\Phi_{\sigma}$ is how they interact with $D^{\sigma}$.
Suppose that $D^{\sigma}=\db+\sigma \alpha(z)C$ and let $D^{1}=\db+\alpha(z)C$ (where $\alpha=\pm z,\pm \cl{z}$). Then we easily compute
$$D^{\sigma}\circ \Phi_{\sigma}(v)=\sigma^{1/2}\Phi_{\sigma}(D^{1}(v))+(\db \rho)\sigma^{1/2}v(\sigma^{1/2}z).$$ Recall that the $L^{2}$ norm of $\lambda v(\lambda z)$ is constant as function of $\lambda$. A similar computation can be done using $\Pi_{\sigma}$, and we conclude:
\begin{equation}\label{eq:impo-lat-1}
  \begin{aligned}
    \norm{D^{\sigma}(\Phi_{\sigma}(v))-\sigma^{1/2}\Phi_{\sigma}(D^{1}(v))}_{L^{2}}\le c(\rho)\norm{v}_{L^{2}(D(\sigma)\setminus D(\sigma/2))},\\
    \norm{\sigma^{1/2}D^{1}(\Pi_{\sigma}(u))-\Pi_{\sigma}(D^{\sigma}(u))}_{L^{2}}\le c(\rho)\norm{u}_{L^{2}(D(1)\setminus D(1/2))}.
  \end{aligned}
\end{equation}
We similarly note the behavior of $D^{\sigma,\ast}$ under $\Phi_{\sigma}$ and $\Pi_{\sigma}$:
\begin{equation}\label{eq:impo-lat-2}
  \begin{aligned}
    \norm{D^{\sigma,\ast}(\Phi_{\sigma}(v))-\sigma^{1/2}\Phi_{\sigma}(D^{1,\ast}(v))}_{L^{2}}\le c(\rho)\norm{v}_{L^{2}(D(\sigma)\setminus D(\sigma/2))},\\
    \norm{\sigma^{1/2}D^{1,\ast}(\Pi_{\sigma}(u))-\Pi_{\sigma}(D^{\sigma,\ast}(u))}_{L^{2}}\le c(\rho)\norm{u}_{L^{2}(D(1)\setminus D(1/2))}.
  \end{aligned}
\end{equation}
These estimates will be important later on. They essentially say that a uniform bound on $\norm{D^{\sigma}(u)}_{L^{2}}$ and $\norm{u}_{L^{2}}$ implies that $\norm{D^{1}(v)}_{L^{2}}=O(\sigma^{-1/2})$ where $v=\Pi_{\sigma}(u)$.

\subsubsection{A linear compactness result}
\label{sec:compactness-lin}
In this section we will prove a compactness theorem which concerns sequences $\xi_{n}$ with $\norm{D^{\sigma_{n}}(\xi_{n})}<C$ and $\sigma_{n}\to\infty$. To set the stage, let $z_{\zeta}$ be the chosen holomorphic coordinate centered on the zero $\zeta$ (as above), and recall the modified rescaling maps:
\begin{equation*}
\text{$\Phi_{\sigma,\zeta}(v)=\rho\cdot \sigma^{1/2}v(\sigma^{1/2}z_{\zeta})$, and  $\Pi_{\sigma,\zeta}(u)=\sigma^{-1/2}\rho(\sigma^{-1/2}z_{\zeta})u(\sigma^{-1/2}z_{\zeta})$.}
\end{equation*}
Let $\Pi_{\sigma}=\oplus_{\zeta\in \mathrm{Z}}\Pi_{\sigma,\zeta}$ be considered as a map
\begin{equation*}
  \Pi_{\sigma}:L^{2}(\dot\Sigma,E)\to \bigoplus_{\zeta\in \mathrm{Z}^{\pm}}L^{2}(\C,\C) \oplus \bigoplus_{\zeta\in \mathrm{Z}^{\pm\pm}}L^{2}(\cl{\mathbb{H}},\mathbb{C})=H.
\end{equation*}
The same formula also defines $\Pi_{\sigma}$ on $L^{2}(\dot\Sigma,\Lambda^{1,0}\otimes E)$. We can think of $H$ as the Hilbert space of $L^{2}$ sections on a disjoint union of finitely many copies of $\C$ and $\cl{\mathbb{H}}$.

We define an operator $D^{1}:H\to H$ (with dense domain) whose restriction to each factor equals the choice of $\db\pm \alpha(z)C$ for $\alpha(z)=z,\cl{z}$ given by Figure \ref{fig:dual-yt}. We similarly define $D^{1,\ast}:H\to H$ where the local form is $-\bd\pm \alpha(z)C$, as appropriate.

The results of Lemmas \ref{lemma:baby-bochner}, \ref{lemma:kernel-class-1} and Corollary \ref{cor:kernel-class-dual} give a complete classification of the elements in $\ker D^{1}$ and $\ker D^{1,\ast}$. See \eqref{eq:kernel-cokernel-elts} in the next section for a summary of the kernel of $D^{1}$ and $D^{1,\ast}$.

We let $\mathrm{R}_{\sigma}(\xi)=\xi-\sum_{\zeta\in \mathrm{Z}} \rho(z_{\zeta})\xi$ which we think of as the ``remainder'' after cutting off. It follows easily from the definitions that
\begin{equation}\label{eq:remainder}
  \begin{aligned}
    D^{\sigma}(\mathrm{R}_{\sigma}(\xi))&=\db \rho\otimes \xi+\mathrm{R}_{\sigma}(D^{\sigma}(\xi)).\\
    \norm{\xi}_{L^{2}}&\le \norm{\mathrm{R}_{\sigma}(\xi)}_{L^{2}}+\norm{\Pi_{\sigma}(\xi)}_{L^{2}}\le 2\norm{\xi}_{L^{2}}.
  \end{aligned}
\end{equation}

\begin{prop}[Linear compactness]
  Let $\xi_{n}\in W^{1,2}(E,F)$ be a sequence so that $\norm{\xi_{n}}_{L^{2}}+\norm{D^{\sigma_{n}}(\xi_{n})}_{L^{2}}$ remains bounded for some sequence $\sigma_{n}\to\infty$. Then
  \begin{enumerate}[label=(\alph*)]
  \item $\norm{R_{\sigma_{n}}(\xi_{n})}_{L^{2}}\to 0$.
  \item After passing to a subsequence, $\Pi_{\sigma_{n}}(\xi_{n})\to \mathbf{k}$ in $L^{2}$ for some element $\mathbf{k}\in \ker D^{1}$.
  \end{enumerate}
  The same holds with $(\xi_{n},E,F,D^{\sigma_{n}},D^{1},\mathbf{k})$ replaced by $(\eta_{n},\Lambda^{0,1}\otimes E,F^{*},D^{\sigma_{n},\ast},D^{1,\ast},\mathbf{c})$. 
\end{prop}
\begin{proof}
  We will only prove the $\xi_{n}$ case, leaving the $\eta_{n}$ case to the reader. To avoid too much clutter, we suppress some notation and write $\sigma:=\sigma_{n}$, $\xi_{n}:=\xi$. Keep in mind that $\rho$ is a fixed bump function.

  Let's begin the proof. Using \eqref{eq:remainder} together with the Bochner-Weitzenb\"ock estimate \eqref{eq:bochner-weitzenboch} implies that
  \begin{equation*}
    \norm{B\mathrm{R}_{\sigma}(\xi)}^{2}_{L^{2}}\le \sigma^{-2}\norm{\db \rho}^{2}_{C^{0}}\norm{\xi}_{L^{2}}^{2}+\sigma^{-2}\norm{D^{\sigma}(\xi)}_{L^{2}}^{2}+C\sigma^{-1}\norm{\xi}_{L^{2}}^{2}.
  \end{equation*}
  However, $\mathrm{R}_{\sigma}(\xi)$ is supported on $\dot\Sigma\setminus D(\zeta_{1},1/2)\setminus D(\zeta_{2},1/2)\setminus \cdots $ and it follows that $\abs{B}>b>0$ for some fixed constant $b$ on the support of $\mathrm{R}_{\sigma}(\xi)$. Therefore we conclude that
  \begin{equation*}
    \norm{\mathrm{R}_{\sigma}(\xi)}^{2}_{L^{2}}\le b^{-1}((C\sigma^{-1}+c_{\rho}\sigma^{-2})\norm{\xi}_{L^{2}}^{2}+\sigma^{-2}\norm{D^{\sigma}(\xi)}^{2}_{L^{2}})=O(\sigma^{-1}).
  \end{equation*}
  This proves part (a).

  For part (b), we use \eqref{eq:impo-lat-1} to conclude
  \begin{equation*}
    \norm{D^{1}(\Pi_{\sigma}(\xi))}\le\sigma^{-1/2}(\norm{D^{\sigma}(\xi))}_{L^{2}}+c_{\rho}\norm{\xi}_{L^{2}})=O(\sigma^{-1/2}).
  \end{equation*}
  Let $v_{n}=\Pi_{\sigma}(\xi)$. Then $\norm{v_{n}}_{L^{2}}$ is bounded and $\norm{D^{1}(v_{n})}=O(\sigma_{n}^{-1/2})$. We will now use the local Bochner Weitzenb\"ock estimates (Lemma \ref{lemma:baby-bochner}) to conclude that we have
  \begin{equation}\label{eq:fancy-man}
    \begin{aligned}
      \norm{v_{n}}_{L^{2}}+\norm{\db v_{n}}_{L^{2}}+\norm{zv_{n}}_{L^{2}}&=O(1)\\
      \norm{D^{1}(v_{n})}_{L^{2}}&=O(\sigma_{n}^{-1/2}).
    \end{aligned}
  \end{equation}
  The first estimate above is actually enough to imply that a subsequence of $v_{n}$ converges to some limit $v_{\infty}$ in $L^{2}$; we will explain this step momentarily. The second estimate will imply that $D^{1}(v_{\infty})=0$. This will complete the proof.

  Before we move on, note that the $L^{2}$ elliptic estimates for $\db$ and the first estimate above implies that $v_{n}$ is uniformly bounded in $W^{1,2}$.

  We can phrase the next part of our argument rather generally. If we let
  \begin{equation*}
    W=\set{v\in H\text{ and }\norm{v}_{W^{1,2}}+\norm{zv}_{L^{2}}\le C},
  \end{equation*}
  (with the obvious induced norm) then the inclusion $W\to H$ is compact; we will prove this below. To see how it applies to our problem, observe that the $L^{2}$ estimates for $\db$ and the first part of \eqref{eq:fancy-man} imply that $\norm{v_{n}}_{W^{1,2}}+\norm{zv_{n}}_{L^{2}}$ is bounded, and hence $v_{n}$ is bounded in $W$. Therefore, after passing to a subsequence, $v_{n}$ converges to some limit $v_{\infty}$ in $L^{2}$. If $\varphi$ is any test function (taking real values along the boundary) then we have
  \begin{equation*}
    \ip{D^{1,\ast}\varphi,v_{\infty}}=\lim\ip{D^{1,\ast}\varphi,v_{n}}\to 0,
  \end{equation*}
  and hence $D^{1}v=0$ weakly. By our elliptic regularity results $v$ is smooth, takes real values along the boundary, and $D^{1}v=0$ holds pointwise, as desired. We can then set $\mathbf{k}=v_{\infty}$ to complete the proof.

  It remains to show why $W\to H$ is a compact inclusion. It is well-known that $W^{1,2}(\Omega(r))\subset L^{2}(\Omega(r))$ is a compact inclusion for $\Omega(r)=D(r)$ or $\Omega(r)=D(r)\cap \cl{\mathbb{H}}$. Thus, by a diagonal argument, we can pass to a subsequence $v_{n}$ and that $v_{n}\to v_{\infty}$ for some limit $v\in L^{2}_{\mathrm{loc}}$ (in the $L^{2}_{\mathrm{loc}}$ topology).

  We easily estimate
  \begin{equation*}
    \norm{v_{n}}^{2}_{L^{2}(\Omega(2^{k})\setminus \Omega(2^{k-1}))}\le \frac{1}{4^{k-1}}\norm{zv_{n}}^{2}_{L^{2}(\Omega(2^{k})\setminus \Omega(2^{k-1}))}.
  \end{equation*}
  Since $\Omega(2r)\setminus \Omega(r)$ is precompact, we must have
  \begin{equation*}
    \norm{v_{\infty}}^{2}_{L^{2}(\Omega(2^{k})\setminus\Omega(2^{k-1}))}=\lim\norm{v_{n}}^{2}_{L^{2}(\Omega(2^{k})\setminus \Omega(2^{k-1}))}\le \frac{C^{2}}{4^{k-1}}.
  \end{equation*}
  Since the right hand side is summable, we conclude that $v_{\infty}$ is actually in $L^{2}$. Now for all $k$ we have
  \begin{equation*}
    \begin{aligned}
      \norm{v_{\infty}-v_{n}}_{L^{2}}^{2}&\le \norm{v-v_{n}}_{L^{2}(\Omega(2^{k}))}^{2}+\sum_{\ell>k} \norm{v_{\infty}}_{L^{2}(\Omega(2^{\ell})\setminus\Omega(2^{\ell-1}))}^{2}+\norm{v_{n}}_{L^{2}(\Omega(2^{\ell})\setminus\Omega(2^{\ell-1}))}^{2}.\\
      &\le \norm{v_{\infty}-v_{n}}_{L^{2}(\Omega(2^{k}))}^{2}+2C^{2}4^{-k}
    \end{aligned}
  \end{equation*}
  Pick $k$ large enough that the last term is less than $\epsilon$, and then take the limit $n\to\infty$, yielding
  \begin{equation*}
    \limsup\norm{v_{\infty}-v_{n}}_{L^{2}}^{2}\le \epsilon.
  \end{equation*}
  This implies that $v_{n}\to v_{\infty}$ in $L^{2}$, completing the proof. 
\end{proof}

\subsubsection{Stabilizing $D^{\sigma}$ and computing its index}
\label{sec:stabilizing-sigma}
In this section we will stabilize $D^{\sigma}$ by adding a cokernel element $\mathbf{c}_{\zeta}$ for each zero $\zeta$ with count $-1$ (Figure \ref{fig:dual-yt}). We will also ``co''-stabilize it by adding a kernel element $\mathbf{k}_{\zeta}$ for each $\zeta$ with count $+1$.

We define the following elements of $L^{2}(\C,\C)$ and $L^{2}(\cl{\mathbb{H}},\C)$:
\begin{equation}
  \label{eq:kernel-cokernel-elts}
  \begin{aligned}
    \text{at $\zeta\in \mathrm{Z}^{+}$}&\hspace{1cm}\mathbf{k}_{\zeta}=i\exp(-\frac{1}{2}\abs{z}^{2})\text{ and }\mathbf{c}_{\zeta}=0,\\
    \text{at $\zeta\in \mathrm{Z}^{-}$}&\hspace{1cm}\mathbf{k}_{\zeta}=0\text{ and }\mathbf{c}_{\zeta}=i\exp(-\frac{1}{2}\abs{z}^{2}),\\
    \text{at $\zeta\in \mathrm{Z}^{++}$}&\hspace{1cm}\mathbf{k}_{\zeta}=\exp(-\frac{1}{2}\abs{z}^{2})\text{ and }\mathbf{c}_{\zeta}=0,\\
    \text{at $\zeta\in \mathrm{Z}^{--}$}&\hspace{1cm}\mathbf{k}_{\zeta}=0\text{ and }\mathbf{c}_{\zeta}=\exp(-\frac{1}{2}\abs{z}^{2}),\\
    \text{at $\zeta\in \mathrm{Z}^{+-}\cup \mathrm{Z}^{-+}$}&\hspace{1cm}\mathbf{k}_{\zeta}=0\text{ and }\mathbf{c}_{\zeta}=0,   
  \end{aligned}
\end{equation}
The results of Lemmas \ref{lemma:baby-bochner}, \ref{lemma:kernel-class-1} and Corollary \ref{cor:kernel-class-dual} show that $\mathrm{span}_{\zeta\in \mathrm{Z}}(\mathbf{k}_{\zeta})=\ker D^{1}\subset H$, and $\mathrm{span}_{\zeta\in \mathrm{Z}}(\mathbf{c}_{\zeta})=\ker D^{1,\ast}\subset H$.

Keeping track of the counts of the various kinds of zeros, we see that
\begin{equation}\label{eq:final-county-boi}
  \dim \ker D^{1}-\dim \ker D^{1,\ast}=\mathrm{X}+\mu^{\tau}_{\mathrm{Mas}}.
\end{equation}
Throughout the subsequent arguments, we will use $\mathbf{k}$ and $\mathbf{c}$ to denote linear combinations of the above basic kernel and cokernel elements.

We consider $\Phi_{\sigma}(\mathbf{k})$ and $\Phi_{\sigma}(\mathbf{c})$ as elements of $W^{1,2}(E,F)$ and $W^{1,2}(\Lambda^{0,1}\otimes E,F^{*})$, using the special coordinate charts $z_{\zeta}$ and frames $Y$ defined above.

We define the stabilized operator by the formula:
\begin{equation*}
  \begin{aligned}
    &D^{\sigma}_{\mathrm{st}}:W^{1,2}(E,F)\oplus \ker D^{1,\ast}\to L^{2}(\Lambda^{0,1}\otimes E)\oplus \ker D^{1}\\
    &D^{\sigma}_{\mathrm{st}}(\xi,\mathbf{c})=(D^{\sigma}(\xi)+\Phi_{\sigma}(\mathbf{c}),\textstyle\sum_{\zeta}\norm{\mathbf{k}_{\zeta}}^{-2}\ip{\Pi_{\sigma}(\xi),\mathbf{k}_{\zeta}}\mathbf{k}_{\zeta}).
  \end{aligned}
\end{equation*}
Note that the second factor is simply an orthogonal projection. The following result will complete the proof of the index formula.
\begin{prop}\label{prop:gluing-iso}
  The operator $D^{\sigma}_{\mathrm{st}}$ is an isomorphism for $\sigma$ sufficiently large. 
\end{prop}
See \cite[Section 5.7]{wendl-sft} for a similar result.
\begin{proof}
  We summarize the strategy. First prove that $D^{\sigma}_{\mathrm{st}}$ is eventually uniformly injective, in the sense that there are constants $C,\sigma_{0}$ so that
  \begin{equation}\label{eq:eventually-inj}
    \sigma>\sigma_{0}\implies \norm{(\xi,\mathbf{c})}_{L^{2}}\le C\norm{D^{\sigma}_{\mathrm{st}}(\xi,\mathbf{c})}_{L^{2}}.
  \end{equation}  
  Second, we show that $D^{\sigma}_{\mathrm{st}}$ is eventually surjective. Then $D^{\sigma}_{\mathrm{st}}$ is eventually an isomorphism, as desired.

  We prove \eqref{eq:eventually-inj} by contradiction; suppose not and then conclude a sequence $\sigma_{n}\to\infty$ and elements $(\xi_{n},\mathbf{c}_{n})$ so that $\norm{(\xi_{n},\mathbf{c}_{n})}_{L^{2}}=1$ but $\norm{D^{\sigma_{n}}_{\mathrm{st}}(\xi_{n},\mathbf{c}_{n})}_{L^{2}}\to 0$. Let's agree to abbreviate $\sigma=\sigma_{n}$ to avoid excessive subscripts during the course of this argument.

  It is clear that
  \begin{equation*}
    \norm{D^{\sigma}(\xi_{n})}_{L^{2}}\le \norm{D^{\sigma}_{\mathrm{st}}(\xi_{n},\mathbf{c}_{n})}_{L^{2}}+C\norm{\mathbf{c}_{n}}_{L^{2}},
  \end{equation*}
  for a fixed constant $C$. In particular, we can apply our compactness result to $\xi_{n}$ and conclude that, after passing to a subsequence $\Pi_{\sigma}(\xi_{n})$ converges to $\mathbf{k}$ and $\mathrm{R}_{\sigma}(\xi_{n})$ converges to $0$. However, since $\mathbf{k}_{\zeta}$ form an orthogonal basis for $\ker D^{1}$ we have
  \begin{equation*}    \mathbf{k}=\lim_{n\to\infty}\sum_{\zeta}\norm{\mathbf{k}_{\zeta}}^{-2}\ip{\Pi_{\sigma}(\xi_{n}),\mathbf{k}_{\zeta}}\mathbf{k}_{\zeta}.
  \end{equation*}
  Therefore $D^{\sigma}_{\mathrm{st}}(\xi,\mathbf{c})\to 0$ implies that $\mathbf{k}=0$. Therefore $\Pi_{\sigma}(\xi_{n})$ converges to zero in $L^{2}$, and since we know $\mathrm{R}_{\sigma}(\xi_{n})\to 0$, we conclude $\xi_{n}$ converges to zero in $L^{2}$.

  In order to contradict our initial assumption, it suffices to show that the inner product $\ip{\Phi_{\sigma}(\mathbf{c}),D^{\sigma}(\xi)}$ converges to zero (because then $\norm{\mathbf{c}_{n}}^{2}\le \norm{D^{\sigma}_{\mathrm{st}}(\xi_{n},\mathbf{c}_{n})}^{2}+\epsilon$ must hold eventually, by Pythagoras' theorem, for arbitrary $\epsilon$). Using the adjointness property and \eqref{eq:impo-lat-1}, we have
  \begin{equation*}    \ip{\Phi_{\sigma}(\mathbf{c}_{n}),D^{\sigma}(\xi_{n})}=\ip{\mathbf{c}_{n},\Pi_{\sigma}(D^{\sigma(\xi_{n})})}=\sigma^{1/2}\ip{\mathbf{c}_{n},D^{1}(\Pi_{\sigma}(\xi_{n}))}+o(1)=o(1),
  \end{equation*}
  where we use the fact that $\mathbf{c}_{n}\in \ker D^{1,\ast}$. This completes the proof by contradiction, and hence we have \eqref{eq:eventually-inj}.

  To prove that $D^{\sigma}_{\mathrm{st}}$ is eventually surjective, we also argue by contradiction. Suppose that it were not. Then by standard properties of Hilbert spaces, we could find a unit norm sequence $\eta_{n},\mathbf{k}_{n}$ (with $\sigma_{n}\to\infty$) so that
  \begin{equation*}
    \ip{D^{\sigma}(\xi)+\Phi_{\sigma}(\mathbf{c}),\eta_{n}}+\ip{\Pi_{\sigma}(\xi),\mathbf{k}_{n}}=0\text{ for all $n,\xi,\mathbf{c}$},
  \end{equation*}
  Using $\Pi_{\sigma}^{*}=\Phi_{\sigma}$ and $\mathbf{c}=0$, we conclude that $D^{\sigma,\ast}(\eta_{n})=-\Phi_{\sigma}(\mathbf{k}_{n})$. Since this is bounded in $L^{2}$, we can apply the compactness result to conclude that $\Pi_{\sigma}(\eta_{n})$ converges to a solution of $\ker D^{1,\ast}$. However the assumption that
  \begin{equation*}
    \ip{\Phi_{\sigma}(\mathbf{c}),\eta_{n}}=0,
  \end{equation*}
  for \emph{all} $\mathbf{c}\in \ker D^{1,\ast}$, allows us to conclude that $\Pi_{\sigma}(\eta_{n})$ converges to $0$. It follows that $\eta_{n}$ converges to zero (since we already know $\mathrm{R}_{\sigma}(\eta_{n})$ converges to zero). Now set $\xi_{n}=\Phi_{\sigma}(\mathbf{k}_{n})$ and $\mathbf{c}=0$ to conclude that
  \begin{equation*}
    \begin{aligned}
      0=&\ip{D^{\sigma}(\Phi_{\sigma}(\mathbf{k}_{n})),\eta_{n}}+\ip{\Phi_{\sigma}(\mathbf{k}_{n}),\Phi_{\sigma}(\mathbf{k}_{n})}.\\
      =&\ip{\Phi_{\sigma}(D^{1}(\mathbf{k}_{n}))+o(1),\eta_{n}}+\ip{\Phi_{\sigma}(\mathbf{k}_{n}),\Phi_{\sigma}(\mathbf{k}_{n})}.\\
      =&\ip{o(1),\eta_{n}}+\ip{\Phi_{\sigma}(\mathbf{k}_{n}),\Phi_{\sigma}(\mathbf{k}_{n})}.\\
      \implies &\norm{\Phi_{\sigma}(\mathbf{k}_{n})}=o(1)\implies \norm{\mathbf{k}_{n}}=o(1).
    \end{aligned}
  \end{equation*}
  We have shown that both $\eta_{n},\mathbf{k}_{n}$ converge to zero, which contradicts our assumption that they were unit norm. This completes the proof.
  \end{proof}
\begin{remark}
  It follows easily from Proposition \ref{prop:gluing-iso} that $$\mathrm{ind}(D^{\sigma})=\dim \ker D^{1}-\dim \ker D^{1,\ast}.$$ To see why, write $D^{\sigma}_{\mathrm{st}}$ in matrix form. Deform the operator by keeping the $1,1$ entry fixed and setting all the other entries to zero. This deformation does not change the Fredholm index. It is easy to compute the Fredholm index after the deformation.

  Equation \eqref{eq:final-county-boi} then implies that $\mathrm{ind}(D^{\sigma})=\mathrm{X}+\mu^{\tau}_{\mathrm{Mas}}$, which completes the proof of Lemma \ref{lemma:final-baby-boy}. This in turn completes the proof of Proposition \ref{prop:20} (the index formula for $\mathrm{ind}(D^{\mathrm{al}})$). Applying our earlier result Proposition \ref{prop:main-gluing} (relating $\mathrm{ind}(D)$ and $\mathrm{ind}(D^{\mathrm{al}})$) completes the proof of our main result, Theorem \ref{theorem:index-formula}.
\end{remark}
\appendix

\section{On the parity of the Conley-Zehnder indices}
The purpose of this appendix is to explain how the parity of the Conley-Zehnder index of an asymptotic operator $A$ changes with the asymptotic trivialization.

Here is the setup. Suppose that $A_{1}=-i\bd_{t}-S(t)$ is an asymptotic operator. First we suppose that $A_{1}$ is defined on the interval $[0,1]$.

Let $\Omega(t)\in U(n)$ be a path of unitary matrices with the property that $\Omega(0),\Omega(1)$ preserve $\R^{n}$. Referring to Section \ref{sec:ahs}, we see that the transition function between any two asymptotic trivializations $\tau_{1},\tau_{2}$ always takes this form.

In this fashion $\Omega(0),\Omega(1)$ can be thought of as elements of $\mathrm{O}(n)$.

\begin{prop}
  Let $A_{2}=\Omega(t)^{-1}A_{1}\Omega(t)$. Then
  \begin{equation*}
    \begin{aligned} \Omega(0)\Omega(1)\in \mathrm{SO}(n)\implies \mu_{\mathrm{CZ}}(A_{2})-\mu_{\mathrm{CZ}}(A_{1})=0\text{ mod }2,\\
      \Omega(0)\Omega(1)\in\mathrm{O}(n)\setminus \mathrm{SO}(n)\implies \mu_{\mathrm{CZ}}(A_{2})-\mu_{\mathrm{CZ}}(A_{1})=1\text{ mod }2.
    \end{aligned}
  \end{equation*}
\end{prop}
\begin{proof}
  To set the stage, consider the Cauchy-Riemann operator on $(\R\times [0,1],\C^{n},\R^{n})$ which equals
  \begin{equation*}
    D=\bd_{s}-(1-\beta(s))A_{1}-\beta(s)A_{2}.
  \end{equation*}
  Here $\beta(s)$ is a cut-off function which vanishes for $s\le 0$ and equals $1$ for $s\ge 1$. We also suppose that $\beta'(s)>0$ on $(0,1)$.

  Clearly, the index formula gives
  \begin{equation*}    \mathrm{ind}(D)=\mu_{\mathrm{CZ}}(A_{2})-\mu_{\mathrm{CZ}}(A_{1}).
  \end{equation*}

  Now let $X_{0}$ be the standard frame of $\C^{n}$ and  consider a non-standard frame $X(s,t)$ with the property that
  \begin{equation*}
    \begin{aligned}
      X(s,t)&=\Omega(t)^{-1}X_{0}\text{ for $s\ge 1$}\\
      X(s,t)&=X_{0}\text{ for $s\le 0$}.
    \end{aligned}
  \end{equation*}
  A priori, we do not define $X$ on the region $[0,1]\times [0,1]$. Observe that with respect to this non-standard frame both asymptotics of $D$ are equal to $\bd_{s}-A_{1}$. As a consequence, the index formula gives
  \begin{equation*}
    \mathrm{ind}(D)=\mu_{\mathrm{Mas}}^{\tau},
  \end{equation*}
  where $\tau=X^{-1}$. Recalling the definition, we conclude $\mu^{\tau}_{\mathrm{Mas}}$ is the signed count of zeros of a generic section of $\det(\C^{n})^{\otimes 2}$ which (a) equals
  \begin{equation*}
    (X_{1}\wedge \cdots \wedge X_{n})^{\otimes 2}
  \end{equation*}
  in the ends, and (b) restricts to the canonical generator of $\det(\R^{n})^{\otimes 2}$ along the boundary.

  The next step is to prove the result in the case when $\Omega(0),\Omega(1)\in \mathrm{SO}(n)$. Since $\mathrm{SO}(n)$ is connected, we can extend $X$ over the boundary $[0,1]\times \set{0,1}$ so that it is always a frame of $\R^{n}$. Now let
    \begin{equation*}
    \mathfrak{s}=X_{1}\wedge \cdots \wedge X_{n}=:\mathrm{det}(X)
  \end{equation*}
  as a section of $\det(\C^{n})$. It is clear that $\mathfrak{s}$ restricts to some generator of $\det(\R^{n})$ along the boundary.

  Extend $\mathfrak{s}$ to $[0,1]\times [0,1]$ \emph{in two ways}, defining $\mathfrak{s}_{1},\mathfrak{s}_{2}$. We require that the zeros of $\mathfrak{s}_{1},\mathfrak{s}_{2}$ are disjoint. Then $\mathfrak{s}_{1}\otimes \mathfrak{s}_{2}$ defines a section of $\det(\C^{n})^{\otimes 2}$ whose zeros count $\mu^{\tau}_{\mathrm{Mas}}$. However, it is easy to show that the signed count of zeros of $\mathfrak{s}_{1}$ and $\mathfrak{s}_{2}$ agree:
  \begin{equation*}
    \#\mathfrak{s}_{1}^{-1}(0)=\#\mathfrak{s}_{2}^{-1}(0).
  \end{equation*}
  The argument proving this equality is similar to results in \cite{milnor}.
  
  Hence
  \begin{equation*}
    \mu^{\tau}_{\mathrm{Mas}}=\#\mathfrak{s}_{1}^{-1}(0)+\#\mathfrak{s}_{2}^{-1}(0)=0\text{ mod }2.
  \end{equation*}
  This proves the theorem in the case $\Omega(0),\Omega(1)\in \mathrm{SO}(n)$.

  Next, suppose that $\Omega(t)=\mathrm{diag}(e^{i\pi t},1,\dots,1)$. Then
  \begin{equation*}
    \det(X)=\left\{
    \begin{aligned}
      e^{-i\pi t}\det(X_{0})\text{ for $s\ge 1$,}\\
      \det(X_{0})\text{ for $s\le 0$}.
    \end{aligned}\right.\implies    \det(X)^{\otimes 2}=\left\{
    \begin{aligned}
      e^{-i2\pi t}\det(X_{0})^{\otimes 2}\text{ for $s\ge 1$,}\\
      \det(X_{0})^{\otimes 2}\text{ for $s\le 0$}.
    \end{aligned}\right.
  \end{equation*}
  Therefore we can extend $\mathfrak{s}=\det(X)^{\otimes 2}$ as
  \begin{equation*}
    \mathfrak{s}=[(1-\beta(s))+\beta(s)e^{-2\pi i t}]\det(X_{0})^{\otimes 2}.
  \end{equation*}
  It is clear that this restricts to the canonical generator of $\det(\R^{n})^{\otimes 2}$ along the boundary. Our assumption that $\beta'(s)>0$ for $s\in (0,1)$ implies that $\mathfrak{s}$ has a single transverse zero when $\beta(s)=1/2$ and $t=1/2$. As a consequence we conclude that
  \begin{equation*}
    \mu_{\mathrm{CZ}}(A_{2})-\mu_{\mathrm{CZ}}(A_{1})=\mu^{\tau}_{\mathrm{Mas}}=\#\mathfrak{s}^{-1}(0)=\pm 1.
  \end{equation*}
  A similar argument works with $\Omega(t)=(e^{i\pi(1-t)},1,\dots,1)$.

  Now suppose that $\Omega(0)\in \mathrm{SO}(n)$ but $\Omega(1)\not\in \mathrm{SO}(n)$. Then we can replace
  \begin{equation*}
    \tilde{\Omega}(t)=\mathrm{diag}(e^{i\pi t},1,\cdots,1)\Omega(t).
  \end{equation*}
  This modified operator has $\tilde{\Omega}(0),\tilde{\Omega}(1)\in \mathrm{SO}(n)$, and so the first part of our proof applies, and so we conclude
  \begin{equation*}
    \mu_{\mathrm{CZ}}(\tilde{\Omega}^{-1}A_{1}\tilde{\Omega})-    \mu_{\mathrm{CZ}}(A_{1})=0\text{ mod }2.
  \end{equation*}
  The second part of our proof implies that
  \begin{equation*}
    \mu_{\mathrm{CZ}}(\tilde{\Omega}^{-1}A_{1}\tilde{\Omega})-        \mu_{\mathrm{CZ}}(\Omega^{-1}A_{1}\Omega)=1\text{ mod }2.
  \end{equation*}
  Combining these allows us to conclude the desired result. The other cases are handled similarly.
\end{proof}

The $\R/\Z$ case is simpler. If $A_{1}$ is defined on $\R/\Z$ and $\Omega(t)$ is a loop in $U(n)$ then we have:
\begin{prop}[closed case]
  Let $A_{2}=\Omega(t)^{-1}A_{1}\Omega(t)$. Then
  \begin{equation*}    \mu_{\mathrm{CZ}}(A_{2})-\mu_{\mathrm{CZ}}(A_{1})=0\text{ mod }2.
  \end{equation*}
\end{prop}
\begin{proof}
  The argument is similar to the one given above. The proof is left to the reader. See also \cite[Section 3.4]{wendl-sft}.
\end{proof}

\section{On the invariance of the Euler characteristic term}
The analysis in Section \ref{sec:final-frontier} shows that the Euler characteristic term $\mathrm{X}(\Sigma,\Gamma_{\pm})$ is an invariant, as explained in Proposition \ref{prop:X-well-defined}. However, it is possible to give an elementary proof of the invariance of $\mathrm{X}(\Sigma,\Gamma_{\pm})$, using techniques similar to those found in \cite{milnor}.

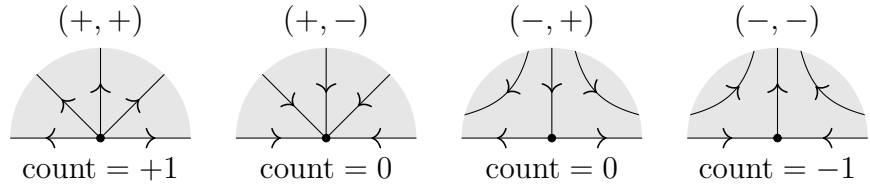
\begin{figure}[H]
  \centering
  \begin{tikzpicture}[scale=.6]
    \begin{scope}[shift={(0,0)}]
      \fill[black!10!white] (2,0) arc (0:180:2)--cycle;
      \node (A) at (0,2) [above] {$(+,+)$};

      \draw[postaction={decorate,decoration={markings,mark=at position 0.5 with {\arrow[scale=1.5]{<};}}}] (-2,0)--(0,0);
      \draw[postaction={decorate,decoration={markings,mark=at position 0.5 with {\arrow[scale=1.5]{<};}}}] (2,0)--(0,0);
      \draw[postaction={decorate,decoration={markings,mark=at position 0.5 with {\arrow[scale=1.5]{<};}}}] (0,2)--(0,0);
      \draw[postaction={decorate,decoration={markings,mark=at position 0.5 with {\arrow[scale=1.5]{<};}}}] (45:2)--(0,0);
      \draw[postaction={decorate,decoration={markings,mark=at position 0.5 with {\arrow[scale=1.5]{<};}}}] (135:2)--(0,0);
      \node[draw,circle,inner sep=1pt,fill] at (0,0){};
      \node at (0,-0.2)[below]{$\mathrm{count}=+1$};
    \end{scope}
    \begin{scope}[shift={(5,0)}]
      \node (A) at (0,2) [above] {$(+,-)$};

      \fill[black!10!white] (2,0) arc (0:180:2)--cycle;
      \draw[postaction={decorate,decoration={markings,mark=at position 0.5 with {\arrow[scale=1.5]{>};}}}] (-2,0)--(0,0);
      \draw[postaction={decorate,decoration={markings,mark=at position 0.5 with {\arrow[scale=1.5]{>};}}}] (2,0)--(0,0);
      \draw[postaction={decorate,decoration={markings,mark=at position 0.5 with {\arrow[scale=1.5]{>};}}}] (0,2)--(0,0);
      \draw[postaction={decorate,decoration={markings,mark=at position 0.5 with {\arrow[scale=1.5]{>};}}}] (45:2)--(0,0);
      \draw[postaction={decorate,decoration={markings,mark=at position 0.5 with {\arrow[scale=1.5]{>};}}}] (135:2)--(0,0);
      \node[draw,circle,inner sep=1pt,fill] at (0,0){};
      \node at (0,-0.2)[below]{$\mathrm{count}=0$};
    \end{scope}
    \begin{scope}[shift={(10,0)}]
      \node (A) at (0,2) [above] {$(-,+)$};

      \fill[black!10!white] (2,0) arc (0:180:2)--cycle;
      \begin{scope}
        \clip (2,0) arc (0:180:2)--cycle;
        \draw[postaction={decorate,decoration={markings,mark=at position 0.5 with {\arrow[scale=1.5]{>};}}}] plot[domain=0.5:2] ({\x},{1/\x});
        \draw[postaction={decorate,decoration={markings,mark=at position 0.5 with {\arrow[scale=1.5]{>};}}}] plot[domain=0.5:2] ({-\x},{1/\x});
      \end{scope}
      \draw[postaction={decorate,decoration={markings,mark=at position 0.5 with {\arrow[scale=1.5]{<};}}}] (-2,0)--(0,0);
      \draw[postaction={decorate,decoration={markings,mark=at position 0.5 with {\arrow[scale=1.5]{<};}}}] (2,0)--(0,0);
      \draw[postaction={decorate,decoration={markings,mark=at position 0.5 with {\arrow[scale=1.5]{>};}}}] (0,2)--(0,0);
      \node[draw,circle,inner sep=1pt,fill] at (0,0){};
      \node at (0,-0.2)[below]{$\mathrm{count}=0$};
    \end{scope}
    \begin{scope}[shift={(15,0)}]
      \node (A) at (0,2) [above] {$(-,-)$};

      \fill[black!10!white] (2,0) arc (0:180:2)--cycle;
      \begin{scope}
        \clip (2,0) arc (0:180:2)--cycle;
        \draw[postaction={decorate,decoration={markings,mark=at position 0.5 with {\arrow[scale=1.5]{<};}}}] plot[domain=0.5:2] ({\x},{1/\x});
        \draw[postaction={decorate,decoration={markings,mark=at position 0.5 with {\arrow[scale=1.5]{<};}}}] plot[domain=0.5:2] ({-\x},{1/\x});
      \end{scope}
      \draw[postaction={decorate,decoration={markings,mark=at position 0.5 with {\arrow[scale=1.5]{>};}}}] (-2,0)--(0,0);
      \draw[postaction={decorate,decoration={markings,mark=at position 0.5 with {\arrow[scale=1.5]{>};}}}] (2,0)--(0,0);
      \draw[postaction={decorate,decoration={markings,mark=at position 0.5 with {\arrow[scale=1.5]{<};}}}] (0,2)--(0,0);
      \node[draw,circle,inner sep=1pt,fill] at (0,0){};
      \node at (0,-0.2)[below]{$\mathrm{count}=-1$};
    \end{scope}
  \end{tikzpicture}
  \caption{The four models for a boundary zero of $V$. The first sign is from the linearization of $V:\Sigma\to T\Sigma$ and the second sign is from the linearization of the restriction $V:\bd\Sigma\to T\bd\Sigma$.}
  \label{fig:two-signs-2-baby}
\end{figure}

The argument is as follows: pick two admissible vector fields $V_{0}$ and $V_{1}$ and consider them as defining a partial section $V$ of $\pr^{*}T\dot\Sigma\to \dot\Sigma\times [0,1]$ (lying over $\Sigma\times \set{0,1}$).

Extend $V$ to a section of $\pr^{*}T\bd\dot\Sigma\to \bd\dot\Sigma\times [0,1]$, while keeping it transverse. Note that we only do a partial extension, $V$ is not defined on the interior of $\dot\Sigma\times [0,1]$. The zero set $Z(V)$ is a one-manifold in $\bd\dot\Sigma\times [0,1]$ whose endpoints are the zeros of $V_{0},V_{1}$. 

Note that the four kinds of zeros of $V_{0},V_{1}$ can be differentiated by the signs of their linearizations, as in Figure \ref{fig:two-signs-2-baby}. Let us agree to say two zeros \emph{cancel} if they are endpoints of an interval in $Z(V)$ and lie on the same side $\bd\Sigma\times \set{0}$ or $\bd\Sigma\times \set{1}$. We say two zeros \emph{match} if they are endpoints of an interval whose boundaries lie on opposite sides. Let us also agree to call an interval component of $Z(V)$ a \emph{pairing}.

The usual argument (see \cite{milnor}), which only involves the linearization in the boundary direction, then shows that $(\pm,+)$ zeros can only cancel with $(\pm,-)$ zeros, and $(\pm,+)$ zeros can only match with $(\pm,+)$ zeros, etc. Here the various $\pm$ signs are independent (i.e.,\ a $(-,+)$ zero can cancel with a $(+,-)$ zero).

If a $(+,+)$ cancels with $(-,-)$ (along some arc $\gamma$) then we can infinitesimally the extend $V$ into the interior near $\gamma$ so that is non-vanishing away from $\gamma$. This is because the vector fields $V_{0}$ and $V_{1}$ are both pointing upwards along the imaginary axis at the end points of $\gamma$. An explicit formula would be something like $V(s,t)=V(s,0)+t\nu$ where $\nu$ is an inwards pointing vector, and $t$ is distance to the boundary. This can be done compatibly with $V_{0}$ and $V_{1}$. Indeed this formula holds exactly in a standard coordinate chart $z=s+it$ if $V=z\bd_{s}$ or $V=-\cl{z}\bd_{s}$ and $\nu=i\bd_{t}$.

Similarly, we can do this non-vanishing local extension if:
\begin{enumerate}[label={$\bullet$}]
\item a $(+,-)$ cancels with a $(-,+)$,
\item a $(+,+)$ matches with a $(+,+)$,
\item a $(+,-)$ matches with a $(+,-)$,
\item a $(-,+)$ matches with a $(-,+)$,
\item a $(-,-)$ matches with a $(-,-)$.    
\end{enumerate}
Notice that in all cases the count is preserved. The fact that we extend $V$ near these intervals $\gamma$ in a non-vanishing fashion means that these zeros are ``protected'' from being cancelled by interior zeros once we further extend $V$ to the rest of $\dot\Sigma\times [0,1]$.

Unfortunately, there are other ``bad'' pairings which can occur, namely:
\begin{enumerate}[label={$\bullet$}]
\item a $(+,+)$ cancels with a $(+,-)$,
\item a $(-,+)$ cancels with a $(-,-)$,
\item a $(+,+)$ matches with a $(-,+)$,
\item a $(+,-)$ matches with a $(-,-)$.
\end{enumerate}
In each case, the count of boundary zeros is not preserved. Indeed there is always an imbalance of $\pm 1$. Let $\gamma^{\prime}$ be such a bad pairing.

The next step is to extend $V$ to a neighborhood of $\gamma^{\prime}$ in such a way that a single interior zero enters the neighborhood and hits $\gamma^{\prime}$ (which we consider leaving the surface) -- this loss of an interior zero will re-balance the count.

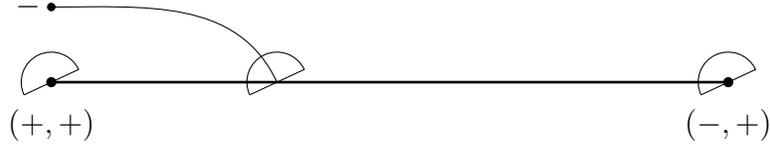
\begin{figure}[H]
  \centering
  \begin{tikzpicture}
    \draw (-3,1)node[draw,circle,inner sep=1pt,fill=black]{}node[left]{$-$}to[out=0,in=115](0,0);
    \draw[line width=1pt] (-3,0)node[draw,circle,inner sep=1pt,fill=black]{}node[below,shift={(0,-0.2)}]{$(+,+)$}--(6,0)node[draw,circle,inner sep=1pt,fill=black]{}node[below,shift={(0,-0.2)}]{$(-,+)$};
    \draw[rotate=25] (-.4,0)--(.4,0)arc (0:180:0.4);
    \draw[shift={(-3,0)}][rotate=25] (-.4,0)--(.4,0)arc (0:180:0.4);
    \draw[shift={(6,0)}][rotate=25] (-.4,0)--(.4,0)arc (0:180:0.4);
  \end{tikzpicture}
  \caption{An interior zero leaves through $\gamma^{\prime}$, re-balancing the count. One should imagine $(\cl{\mathbb{H}}\cap D(1))\times \gamma^{\prime}$ in the figure.}
  \label{fig:rebalance}
\end{figure}
To simplify the construction, let us suppose that $V$ is constant on $\bd\dot\Sigma\times [x_{0},x_{1}]$, and suppose that $p\times [x_{0},x_{1}]$ is contained in $\gamma^{\prime}$. This can be achieved after a small perturbation of $V$. Since $V$ is cut transversally, by assumption, we know that $p$ is a non-degenerate zero of $V_{x}=V|_{\bd\dot\Sigma\times \set{x}}$ for all $x\in [x_{0},x_{1}]$.

Then we can extend $V$ to a non-vanishing section on a neighborhood of the complement $$\gamma^{\prime}\setminus (p\times (x_{0},x_{1})),$$ by keeping the imaginary part of $V$ always positive or negative. After a perturbation on a fixed boundary coordinate chart of $p$ we may suppose that
\begin{equation*}
  V_{x_{0}}=z\text{ and }V_{x_{1}}=\cl{z}\text{ or }V_{x_{0}}=-z\text{ and }V_{x_{1}}=-\cl{z},
\end{equation*}
or vice-versa. This uses the fact that $V_{x_{0}}$ and $V_{x_{1}}$ have non-degenerate zeros. Suppose the chart is valued in $\cl{\mathbb{H}}\cap D(1)$.

On a neighborhood of the interval $\set{p}\times [x_{0},x_{1}]$ we will interpolate between positive imaginary part to negative part by the formula
\begin{equation*}
  V(x,z)=(1-\beta(x))z+\beta(x)\cl{z},
\end{equation*}
where $\beta$ monotonically increases from $0$ to $1$. Similar formulas work in the other cases. Note that we can explicitly describe the zero set of $V$ near $\gamma^{\prime}$ as
\begin{equation*}
  Z(V)=(\set{0}\times [x_{0},x_{1}])\cup i\R\cap \Omega(1)\times \set{\beta^{-1}(1/2)}.
\end{equation*}
i.e.,\ $x_{\ast}=\beta^{-1}(1/2)$ is a singular time where the zero set of $V$ looks like Figure \ref{fig:rebalance}. Let us call the set $(i\R\times \Omega(1))\times \set{x_{\ast}}$ a \emph{sink} for $\gamma^{\prime}$. Every ``bad pairing'' has exactly one sink attached to it.

To simplify the set up, let us suppose that $V$ is $x$-independent for $x\in [0,\epsilon]\cup [1-\epsilon,1]$, and place all of the sinks in these regions. Note that once we do the local extensions of $V$, it will no longer be $x$ independent in this region.

Now extend $V$ to the rest of $\dot\Sigma\times [0,1]$, so that it is a transverse section away from what we have already defined. A straightforward argument shows that every sink is one end of a (compact) interval $I\to Z(V)$; we think of these as pairings between sinks and interior zeros, or pairings between two sinks. The other components of $Z(V)$ are pairings which join interior zeros to interior zeros; these will preserve the counts (by \cite{milnor}). 

Let $I$ be a pairing between an interior zero and a sink. Orient $\dot\Sigma\times [0,1]$ so that $(\bd_{s},\bd_{t},\bd_{x})$ forms an oriented basis. Let $\nu(I)$ denote the normal bundle to $I$. The linearization of the vector field defines a map $\nu(I)\to T\dot\Sigma|_{I}$ which is an isomorphism except at the sink endpoint (i.e.,\ a single endpoint). Near the interior zero $p$, $\nu(I)$ is identified with $T\Sigma_{p}$.

Near the sink, $\nu(I)$ is identified with $\mathrm{span}\set{\bd_{s},\bd_{x}}$, and the linearization depends on what kind of pairing of $Z(V)\cap \bd\Sigma$ it intersects. Let us focus on only the matching parts of $Z(V)\cap \bd \Sigma$ (and not the cancelling parts). Keeping in the local model for the zero set near $\gamma^{\prime}$, we write the linearization near the sinks as
\begin{enumerate}[label={(\alph*)}]
\item $\bd_{s}\mapsto +\bd_{s}\text{ and }\bd_{x}\mapsto +\bd_{t}$, if $(-,+)$ matches with $(+,+)$ (net count: $+1$),
\item $\bd_{s}\mapsto -\bd_{s}\text{ and }\bd_{x}\mapsto -\bd_{t}$, if $(-,-)$ matches with $(+,-)$ (net count: $+1$),
\item $\bd_{s}\mapsto +\bd_{s}\text{ and }\bd_{x}\mapsto -\bd_{t}$, if $(+,+)$ matches with $(-,+)$ (net count: $-1$),
\item $\bd_{s}\mapsto -\bd_{s}\text{ and }\bd_{x}\mapsto +\bd_{t}$, if $(+,-)$ matches with $(-,-)$ (net count: $-1$).
\end{enumerate}
Keep in mind that matching is a directed relation. Similar tables hold for cancelling pairs of $V_{0}$ and $V_{1}$. 

We orient the interval $I$ so that it ends at a sink. Then $-\bd_{t}$ is identified with the tangent vector near the sink, and since $\set{-\bd_{t},\bd_{s},\bd_{x}}$ has the ambient orientation, we conclude that $\set{\bd_{s},\bd_{x}}$ is an oriented basis of $\nu(I)$ near the sink.

If the initial location of $I$ lies at $x=0$, then $\nu(I)\simeq T\Sigma$ is orientation preserving. If the initial location of $I$ lies at $x=1$ then the identification $\nu(I)\simeq T\Sigma$ is orientation reversing.

In particular, in case (a) or (b), since these linearizations preserve orientation, the initial zero of $I$ must be positive if it lies at $x=0$ or negative if it lies at $x=1$. In either case, the count is re-balanced. In case (c) or (d), the linearization near the sink reverses orientation, and we similarly show that the count is re-balanced.

Let us finally comment on what happens when sinks connect to sinks via an interval $J$. We suppose that both sinks are in one of the cases (a)-(d).

As we have explained, at the terminal end, $\set{\bd_{s},\bd_{x}}$ is an oriented basis of $\nu(J)$. At the initial end, $\set{\bd_{s},\bd_{x}}$ has the wrong orientation of $\nu(J)$. As a consequence, sinks of type (a) or (b) can only cancel with sinks of type (c) or (d) (and vice-versa). This implies that the counts are re-balanced.

For completeness, here are the tables for the linearizations at sinks for cancelling pairs (as opposed to matching pairs).

If a sink lies on a cancelling between two zeros in $V_{0}$, then its linearization is given by:
\begin{enumerate}[resume,label={(\alph*)}]
\item $\bd_{s}\mapsto +\bd_{s}$ and $\bd_{x}\mapsto +\bd_{t}$ if $(-,+)$ cancels with $(-,-)$ (net count: $+1$).
\item $\bd_{s}\mapsto +\bd_{s}$ and $\bd_{x}\mapsto -\bd_{t}$ if $(+,+)$ cancels with $(+,-)$ (net count: $-1$).  
\end{enumerate}

If a sink lies on a cancelling between two zeros in $V_{1}$, then the linearization is given by:
\begin{enumerate}[resume,label={(\alph*)}]
\item $\bd_{s}\mapsto +\bd_{s}$ and $\bd_{x}\mapsto +\bd_{t}$ if $(+,-)$ cancels with $(-,-)$ (net count: $-1$).
\item $\bd_{s}\mapsto +\bd_{s}$ and $\bd_{x}\mapsto -\bd_{t}$ if $(+,+)$ cancels with $(-,+)$ (net count: $+1$).  
\end{enumerate}

The rest of the details are left to the reader.

\bibliography{citations}
\bibliographystyle{alpha}
\end{document}